\documentclass[a4paper,12pt]{amsart}
\usepackage[margin=2cm]{geometry}
\usepackage[alphabetic]{amsrefs}

\usepackage{indentfirst}
\usepackage{fancyhdr}
\usepackage{amssymb}

\usepackage{pifont}
\newcommand{\cmark}{\ding{51}}
\newcommand{\xmark}{\ding{55}}

\usepackage{amsmath,mathtools} 
\usepackage{latexsym}
\usepackage{amsthm} 
\usepackage{eucal} 
\usepackage{enumerate} 
\usepackage{pdfpages}
\usepackage{mathrsfs}
\usepackage{stackrel}
\usepackage{enumerate}
\usepackage{graphicx}
\usepackage[colorlinks=true,urlcolor=blue,
citecolor=red,linkcolor=blue,linktocpage,pdfpagelabels,
bookmarksnumbered,bookmarksopen]{hyperref}

\numberwithin{equation}{section}
\theoremstyle{plain}

\usepackage[T1]{fontenc}
\usepackage{aurical}

\newcommand{\emme}{{{\mbox{\fontfamily{aurical}\large\Fontlukas\slshape{m\,}}}}}

\newtheorem{theorem}{Theorem}[section]
\newtheorem{lem}[theorem]{Lemma}
\newtheorem{cor}[theorem]{Corollary}
\newtheorem{prop}[theorem]{Proposition}
\theoremstyle{definition}

\newcommand{\ve}{\varepsilon}
\newcommand{\R}{\mathbb{R}}
\newcommand{\N}{\mathbb{N}}
\newcommand{\Z}{\mathbb{Z}}
\newcommand{\Ns}{\mathscr{N}_{s}}

\newcommand{\Q}{\mathcal{Q}}
\newcommand{\E}{\mathscr{E}}
\renewcommand{\leq}{\leqslant}
\renewcommand{\le}{\leqslant}
\renewcommand{\geq}{\geqslant}
\renewcommand{\ge}{\geqslant}
\renewcommand{\epsilon}{\varepsilon}
\allowdisplaybreaks

\begin{document}

\title[(Non)local logistic equations with Neumann conditions]{(Non)local logistic equations with Neumann conditions}\thanks{
{\em Serena Dipierro}:
Department of Mathematics
and Statistics,
University of Western Australia,
35 Stirling Hwy, Crawley WA 6009, Australia.
{\tt serena.dipierro@uwa.edu.au}\\
{\em Edoardo Proietti Lippi}: 
Department of Mathematics and Computer Science,
University of Florence,
Viale Morgagni 67/A, 50134 Firenze, Italy.
{\tt edoardo.proiettilippi@unifi.it}
\\
{\em Enrico Valdinoci}:
Department of Mathematics
and Statistics,
University of Western Australia,
35 Stirling Hwy, Crawley WA 6009, Australia. {\tt enrico.valdinoci@uwa.edu.au}\\
The authors are members of INdAM.
The first and third authors
are members of AustMS and
are supported by the Australian Research Council
Discovery Project DP170104880 NEW ``Nonlocal Equations at Work''.
The first author is supported by
the Australian Research Council DECRA DE180100957
``PDEs, free boundaries and applications''. 
The third author is supported by
the Australian Laureate Fellowship
FL190100081
``Minimal surfaces, free boundaries and partial differential equations''.
Part of this
work was carried
out during a very pleasant and fruitful visit of the second author to the
University of Western Australia, which we thank for the warm hospitality.}

\author{Serena Dipierro}
\author{Edoardo Proietti Lippi}
\author{Enrico Valdinoci}


\begin{abstract}
We consider here a problem of population dynamics
modeled on a logistic equation with both classical and nonlocal
diffusion, possibly in combination with a pollination term.

The environment considered is a niche with zero-flux,
according to a new type of Neumann condition.

We discuss the situations that are
more favorable for the survival of the species,
in terms of the first positive eigenvalue.

Quite surprisingly, the eigenvalue analysis for the one dimensional
case is structurally
different than the higher dimensional setting, and it
sensibly depends on the nonlocal character of the dispersal.

The mathematical framework of this problem
takes into consideration the equation
$$ -\alpha\Delta u +\beta(-\Delta)^su
=(m-\mu u)u+\tau\;J\star u \qquad{\mbox{in }}\; \Omega,$$
where $m$ can change sign.

This equation is endowed with a set of Neumann condition
that combines the classical normal derivative prescription
and the nonlocal condition introduced in
[S.~Dipierro, X.~Ros-Oton, E.~Valdinoci,
Rev. Mat. Iberoam. (2017)].

We will establish the existence of a minimal solution
for this problem and provide a throughout discussion
on whether it is possible to obtain non-trivial solutions
(corresponding to the survival of the population).

The investigation will rely
on a quantitative analysis of the first eigenvalue of
the associated problem and on precise asymptotics
for large lower and upper bounds of the resource.

In this, we also analyze the role played by the
optimization strategy in the distribution of the resources,
showing concrete examples
that are unfavorable for survival, in spite of the large resources
that are available in the environment.
\end{abstract}
\maketitle

\tableofcontents

\section{Introduction}

We consider here a biological population
with density~$u$ which is self-competing for
the resources in a given environment~$\Omega$.

These resources are described by a function~$m$, which
is allowed to change sign: the positive values of~$m$
correspond to areas of the environment favorable
for life and produce a positive birth rate, whereas
the negative values model
a hostile environment whose byproduct is a positive death rate 
of linear type.

The competition for the resource is encoded by a nonnegative
function~$\mu$. Resources and competitions are combined
into a standard logistic equation. 
In addition, the population is assumed to present
a combination of classical and nonlocal diffusion
(the cases of purely classical
and purely nonlocal diffusions are also included in our setting,
and the results obtained are new also for these cases).
The population is also endowed with an additional birth rate
possibly
provided by pollination\footnote{While we use
the name of pollination throughout this paper,
we observe that the pollination analysis performed is not
limited to vegetable species: indeed, for
animal species the convolution term that
we study can be
seen as a birth rate of nonlocal type produced, for instance,
by a mating call which attracts partners from surrounding
neighbors.}
and modeled by a convolution operator
(the case of no pollination is also included in our setting,
and the results obtained are new also for this case).

The environment~$\Omega$
describes an ecological niche and is endowed by
a zero-flux condition of Neumann type.
Given the possible presence of both classical and nonlocal
dispersal, this Neumann condition appears to be new
in the literature: when the diffusion is of purely classical
type this new prescription reduces
to the standard normal derivative condition along~$\partial\Omega$,
and when the diffusion is of purely nonlocal
type it coincides with the nonlocal Neumann condition
set in~$\R^n\setminus\overline\Omega$ that has been recently
introduced in~\cite{dirova} -- but in case
the population is subject to both the classical
and the nonlocal dispersion processes the Neumann
condition that we introduce here
takes into account the combination of both
the classical and the nonlocal prescriptions (interestingly,
without producing an overdetermined, or ill-posed, problem).
\medskip

The main question addressed in this paper is
whether or not the environmental niche is suited for
the survival of the population (notice that life
is not always promoted by the ambient resource,
since~$m$ can attain negative values). We will investigate
this question by using spectral analysis and providing
a detailed quantification of favorable and unfavorable
scenarios in terms of the first eigenvalue compared
with the
resource and pollination parameters.
\medskip

More precisely, the mathematical framework in which
we work goes as follows.
We consider a bounded open set~$\Omega\subset\R^n$
with boundary of class~$C^1$: that is, we suppose that
there exist~$R>0$ and~$p_1,\dots,p_K\in\partial\Omega$
such that~$\partial\Omega\subset B_R(p_1)\cup\dots\cup
B_R(p_K)$, and, for each~$i\in\{1,\dots,K\}$, 
\begin{equation}\label{PRECIS:C1}
{\mbox{the set~$\Omega\cap B_R(p_i)$
is~$C^1$-diffeomorphic to~$B_1^+:=\{
(x_1,\dots,x_n)\in B_1{\mbox{ s.t. }}x_n>0\}$.}}
\end{equation}
Given~$s\in(0,1)$,
$\alpha$, $\beta\in[0,+\infty)$,
with~$\alpha+\beta>0$, $m:\Omega\to\R$, $\mu:\Omega\to\left[\underline{\mu},+\infty\right)$,
with~$\underline{\mu}>0$,
$\tau\in
[0,+\infty)$ and~$J\in L^1(\R^n,[0,+\infty))$
with
\begin{equation}\label{Jeven}
J(x)=J(-x)\end{equation} and
\begin{equation}\label{JL1}
\int_{\R^n} J(x)\,dx=1,\end{equation}
we consider the mixed order logistic equation
\begin{equation}\label{problogis-PRE}
-\alpha\Delta u +\beta(-\Delta)^su
=(m-\mu u)u+\tau\;J\star u \qquad{\mbox{in }}\; \Omega,
\end{equation}
where
\[
J\star u(x):=\int_\Omega J(x-y)\,u(y)dy.
\]
When~$\beta=0$, we take the additional hypothesis that
\begin{equation}\label{KASM:LSKDD}{\mbox{$
\Omega$ is connected.}}\end{equation}
We observe that the operator in~\eqref{problogis-PRE}
is of mixed local and nonlocal type, and also of mixed
fractional and integer order type.
Interestingly, the nonlocal character of the operator
is encoded both in the fractional Laplacian
$$ (-\Delta)^s u(x):=\frac12\,\int_{\R^n}\frac{2u(x)
-u(x+\zeta)-u(x-\zeta)}{|\zeta|^{n+2s}}\,d\zeta$$
and in the convolution operator given by~$J$.

The use of the convolution operator
in biological models to
comprise the interaction of the population
with the resource at a certain range
has a very consolidated tradition,
see e.g.~\cite{MR2601079, MR3035974, MR3169773, MR3285831, MR3498523, MR3639140, cadiva}
and the references therein.

As for the nonlocal diffusive operator, for the sake of concreteness
we stick here to the prototypical case of the fractional Laplacian,
but the arguments that we develop are in fact usable in more general
contexts including various interaction kernels of singular type.

Given the presence of both the Laplacian and the fractional
Laplacian, the operator in~\eqref{problogis-PRE}
falls within the diffusive processes of mixed orders,
which have been widely addressed by several
methodologies and arose from a number of different motivations,
see for instance various viscosity solutions
approaches~\cite{MR2129093, MR2243708, MR2422079, MR2653895, MR2911421, MR2963799, MR3194684, 2017arXiv170605306D, 2019arXiv190702495A},
the Aubry-Mather theory for pseudo differential equations~\cite{MR2542727},
Cahn-Hilliard and Allen-Cahn-type equations~\cite{MR3051400, MR3485125},
probability and Harnack inequalities~\cite{MR2095633, MR2180302, MR2928344, MR2912450}, 
decay for
parabolic equations~\cite{MR3950697, MR3912710}, 
friction and dissipation effects~\cite{2018arXiv181107667D}, 
smooth approximation with suitable solutions~\cite{MR2667633},
Bernstein-type regularity
results~\cite{CABRE},
variational methods~\cite{biagvecc},
nonlinear operators~\cite{ABATANGELO}
and plasma physics~\cite{PhysRevE.87.063106}.\medskip

We endow the problem in~\eqref{problogis-PRE}
with a set of Neumann boundary
conditions that correspond to a ``zero-flux''
condition according to the stochastic process producing
the diffusive operator in~\eqref{problogis-PRE}.
This Neumann condition appears to be new in the literature
and depends on the different ranges of~$\alpha$ and~$\beta$
according to the following setting.
If~$\alpha=0$, we
consider the nonlocal Neumann condition introduced
in~\cite{dirova}, thus prescribing that
\begin{equation}\label{NEU-1}
\Ns u(x):=\int_\Omega \frac{u(x)-u(y)}{|x-y|^{n+2s}}\,dy
=0\qquad    {\mbox{for every }}\;x\in\R^n \setminus \overline\Omega.
\end{equation}
If instead~$\beta=0$, we prescribe the classical
Neumann condition
\begin{equation}\label{NEU-2}
\frac{\partial u}{\partial \nu}=0\qquad
{\mbox{on }}\;\partial\Omega.
\end{equation}
Finally, if~$\alpha\ne0$ and~$\beta\ne0$,
we prescribe both the classical and the nonlocal
Neumann conditions, by requiring that
\begin{equation}\label{NEU-3}
\begin{dcases}
\Ns u=0   & \qquad{\mbox{in }}\; \R^n \setminus \overline\Omega,
\\
\displaystyle\frac{\partial u}{\partial \nu}=0   &\qquad{\mbox{on }} \;\partial\Omega.\end{dcases}
\end{equation}
We remark that the prescription in~\eqref{NEU-3}
is {\em not} an ``overdetermined'' condition
(as it will be confirmed by the existence
result in Theorem~\ref{esistenza} below).

The set of boundary/external Neumann
conditions in~\eqref{NEU-1}, \eqref{NEU-2}
and~\eqref{NEU-3}, in dependence of the different
ranges of~$\alpha$ and~$\beta$,
will be denoted by ``$(\alpha,\beta)$-Neumann conditions'',
and, with this notation and~\eqref{problogis-PRE},
the main question studied in this paper focuses
on the problem
\begin{equation}\label{problogis}
\begin{dcases}
-\alpha\Delta u +\beta(-\Delta)^su
=(m-\mu u)u+\tau\;J\star u &\; {\mbox{ in }}\, \Omega,
\\
{\mbox{with $(\alpha,\beta)$-Neumann condition.}}\end{dcases}
\end{equation}
In this setting, the $(\alpha,\beta)$-Neumann conditions
provide an ``ecological niche''
for the population with density~$u$, making~$\Omega$
a natural environment in which a given species
can live and compete for a resource~$m$,
according to a competition function~$\mu$.
In this setting, the parameter~$\tau$,
as modulated by the interaction kernel~$J$,
describes an additional birth rate due to further intercommunication
than just with
the closest neighbors, as it happens, for instance,
in pollination.

As a matter of fact,
the role of the $(\alpha,\beta)$-Neumann conditions
is precisely to make the boundary and the exterior
of the niche~$\Omega$ ``reflective'':
namely when an individual exits the niche,
it is forced to immediately come back into the niche
itself,
following the same diffusive process,
see Section~2 in~\cite{dirova} (see also~\cite{VOND}
for a thoroughgoing probabilistic discussion
about this process).

As a technical remark, we also observe that
our $(\alpha,\beta)$-Neumann condition
is structurally different (even when~$\alpha=0$ and~$s=1/2$)
from the case of
bounded domains with reflecting barriers presented
in~\cite{MR3082317, MR3771424},
and the diffusive operator taken into account
in~\eqref{problogis} cannot be
obtained by the spectral decomposition of
the classical Laplacian in~$\Omega$
(except for the special case
of periodic environments, see e.g. Section~2.3
and Appendix~Q in
\cite{MR3967804}).
\medskip

The possible presence in~\eqref{problogis}
of two different diffusion
operators, one of classical and the other of fractional
flavor, has a clear biological interpretation,
namely the population with density~$u$
can possibly alternate both short and long-range
random walks, and this could be motivated,
for instance, by a superposition between
local exploration of the environment
and hunting strategies
(see e.g.~\cite{NATU, MR1636644, MR2332679, MR2411225, MR2897881, MR2924452, MR3026598, MR3590646, MR3590678}).
A detailed presentation of this superposition
of stochastic processes will be presented in Appendix~\ref{APPEB};
see also~\cite{NICE} for the detailed description
of the local/nonlocal reflecting barrier also in terms
of the population dynamics model.
\medskip

The notion of solution of~\eqref{problogis} is intended here
in the weak sense, as it will be discussed precisely
in formula~\eqref{WEAKSOL}.
See however~\cite{MR1911531, biagvecc} for a regularity theory for weak solutions
of the equations driven by the mixed order
operators as in~\eqref{problogis}.\medskip

Our first result in this setting is that the problem
in~\eqref{problogis} admits a minimal energy
solution (under very natural
and mild structural assumptions). To state it,
it is convenient to define
\begin{equation}\begin{split}\label{qbar}
\underline{q}:=\;&\begin{dcases}
\displaystyle\frac{2^*}{2^*-2} & {\mbox{ if $\beta=0$ and~$n>2$}},\\
\displaystyle\frac{2^*_s}{2^*_s-2} & {\mbox{ if $\beta\ne0$ and~$n>2s$}},\\
1 & {\mbox{ if $\beta=0$ and~$n\le2$, or
if $\beta\ne0$ and~$n\le2s$}},\end{dcases}\\=\;&
\begin{dcases}
\displaystyle\frac{n}{2} & {\mbox{ if $\beta=0$ and~$n>2$}},
\\
\displaystyle\frac{n}{2s} & {\mbox{ if $\beta\ne0$ and~$n>2s$}},\\
1 & {\mbox{ if $\beta=0$ and~$n\le2$, or
if $\beta\ne0$ and~$n\le2s$.}}
\end{dcases}
\end{split}\end{equation}
As customary, the exponent~$2^*_s$ denotes the fractional
Sobolev critical exponent for~$n>2s$
and it is equal to~$\frac{2n}{n-2s}$. Similarly,
the exponent~$2^*$ denotes the classical
Sobolev critical exponent for~$n>2$
and it is equal to~$\frac{2n}{n-2}$.\medskip

We remark that~$\underline{q}\ge n/2$, and we have:

\begin{theorem}\label{esistenza}
Assume that 
\begin{eqnarray*}&&
{\mbox{$m\in L^q(\Omega)$, for some $q\in \big(\underline{q},+\infty\big]$}}\\
&&{\mbox{and $(m+\tau)^3 \mu^{-2} \in L^1(\Omega)$.}}\end{eqnarray*}
Then, there
exists a nonnegative solution of \eqref{problogis}
which can be obtained as a minimum of an energy functional.
\end{theorem}

The precise definition of energy functional
used in Theorem~\ref{esistenza}
will be presented in~\eqref{ENE}: roughly speaking,
the energy associated to Theorem~\ref{esistenza}
will be the ``natural'' functional for the variational methods,
and its Euler-Lagrange equation will correspond to
the notion of weak solution.

While the functional analysis part of the proof of 
Theorem~\ref{esistenza}
relies on standard direct methods in the calculus of variations,
the more interesting part of the argument
makes use of a structural property of the nonlocal Neumann
condition that will be presented in 
Theorem~\ref{thcontiminimo}
(roughly speaking, the
nonlocal Neumann condition in~\eqref{NEU-1} will be instrumental
to minimize the Gagliardo seminorm, thus clarifying the energetic
role of the nonlocal reflection introduced in~\cite{dirova}).\medskip

Though the result in Theorem~\ref{esistenza}
has an obvious interest in pure mathematics,
our main analysis will focus on whether
problem~\eqref{problogis} does admit a {\em
nontrivial} solution
(notice indeed that~$u\equiv0$ is always a solution of~\eqref{problogis}).
In particular, in view of Theorem~\ref{esistenza},
a useful mathematical tool to detect nontrivial
solutions consists in proving that the minimal energy
configuration is not attained by the trivial solution
(hence, in this case, the solution produced by Theorem~\ref{esistenza}
is nontrivial).
The question of the existence of nontrivial
solutions has a central importance for the mathematical
model, since
it corresponds to the possibility of a population
to survive in the environmental condition
provided by the niche. Interestingly,
in our model, the survival of the population
can be enhanced by the possibility of exploiting
resources by long-range interactions. Indeed, we stress that
the nonlocal resource~$m$ in~\eqref{problogis-PRE}
is not necessarily positive (hence, the natural environment
can be ``hostile'' for the population):
in this configuration, we show that
the survival of the species is still possible
if the ``pollination'' birth rate~$\tau$ is sufficiently large.
The quantitative result that we have is the following:

\begin{theorem}\label{intm>0}
Assume that
\begin{eqnarray*}&&
{\mbox{$m\in L^q(\Omega)$, for some $q\in \big(\underline{q},+\infty\big]$}}\\
&&{\mbox{and $(m+\tau)^3 \mu^{-2} \in L^1(\Omega)$.}}\end{eqnarray*}
Then,
\begin{enumerate}
\item[(i)\;\;\;\;] if $m\equiv 0$ and~$\tau=0$,
then the 
only solution of \eqref{problogis} is the one identically zero;
\item[(ii)\;\;\;\;] if 
\begin{equation}\label{orerhgbfvasnm1}
\int_\Omega \big(m(x) +\tau\;J\star 1(x)\big)\,dx>0\end{equation}
and
\begin{equation}\label{orerhgbfvasnm2}
\mu\in L^1(\Omega),\end{equation}
then \eqref{problogis} admits a nonnegative
solution $u\not\equiv0$.
\end{enumerate}
\end{theorem}

A particular case of Theorem \ref{intm>0} is when the resource
$m$ is nonnegative. In this situation,
Theorem~\ref{intm>0}(i) gives that no survival is possible
without resources and pollination,
i.e. when both~$m$ and~$\tau$ vanish identically
(unless also~$\mu$ vanishes identically, then
reducing the problem to that of mixed operator harmonic functions),
whereas Theorem~\ref{intm>0}(ii)
guarantees survival
if at least one between the environmental resource
and the pollination is favorable to life.
Precisely, one can immediately deduce from
Theorem~\ref{intm>0} the following result:

\begin{cor}\label{COL:AMS}
Assume that 
\begin{eqnarray*}&&
{\mbox{$m\in L^q(\Omega)$, for some $q\in \big(\underline{q},+\infty\big]$,}}\\
&&{\mbox{$m$ is nonnegative,}}\\
&&{\mbox{and $(m+\tau)^3 \mu^{-2} \in L^1(\Omega)$.}}\end{eqnarray*}
Then,
\begin{enumerate}
\item[(i)\;\;\;\;] if $m\equiv 0$ and~$\tau=0$, then the 
only solution of \eqref{problogis} is the one identically zero;
\item[(ii)\;\;\;\;] if either~$m>0$ or $\tau(J\star 1)>0$
in a set of positive measure and
$\mu\in L^1(\Omega)$, then \eqref{problogis}
admits a nonnegative solution $u\not\equiv0$.
\end{enumerate}
\end{cor}

Problems related to Corollary~\ref{COL:AMS}
have been studied in
\cite{cadiva} under Dirichelet (rather than Neumann)
boundary conditions.
\medskip

{F}rom the biological point of view, assumption~\eqref{orerhgbfvasnm1}
states that the environment is ``in average''
favorable for the survival of the species.
It is therefore a natural question to investigate
the situation in which the environment
is ``mostly hostile to life''.
To study this phenomenon,
when~$m\in
L^q$ with~$q>n/2$, with~$m^+\not \equiv 0$ and
$$\int_\Omega m(x)\,dx<0,$$
we denote\footnote{As customary, we freely
use in this paper the standard notation
$$ m^+(x):=\max\{ 0, \,m(x)\}\qquad{\mbox{and}}\qquad
m^-(x):=\max\{ 0, \,-m(x)\}.$$} by~$\lambda_1$ the first 
positive eigenvalue associated with
the diffusive operator in \eqref{problogis}.
More precisely, we
consider the weighted eigenvalue problem
\begin{equation}\label{probauto}
\begin{dcases}
-\alpha\Delta u +\beta(-\Delta)^su= \lambda mu   & \;{\mbox{ in }}\,\Omega,
\\
{\mbox{with $(\alpha,\beta)$-Neumann condition.}}\end{dcases}
\end{equation}
As it will be discussed in detail in Propo\-si\-tion~\ref{PROAUTOVA} here
and in~\cite{ANCI},
problem~\eqref{probauto}
admits the existence of two 
unbounded sequences of eigenvalues, 
one positive and one negative.
In this setting,
the smallest strictly positive eigengevalue
will be denoted by~$\lambda_1$.
When we want to emphasize the dependence of~$\lambda_1$
on the resource~$m$, we will write it as~$\lambda_1(m)$.

We also denote by~$e$ an
eigenfunction corresponding to~$\lambda_1$
normalized such that 
$$\int_\Omega m(x)\,e^2(x)\,dx=1.$$
The first eigenvalue will be an important
threshold for the survival of the species,
quantifying the role of the necessary pollination parameter~$\tau$
in order to overcome the presence of an hostile behavior in average.
The precise result that we obtain is the following one:

\begin{theorem}\label{intm<0}
Assume that 
$m\in L^q(\Omega)$, for some $q \in \big(\underline{q},+\infty\big]$,
and~$(m+\tau)^3\mu^{-2}\in L^1(\Omega)$.

Then,
\begin{enumerate}
\item[(i)\;\;\;\;] if $m\le -\tau$,
then the 
only solution of \eqref{problogis} is the one identically zero;
\item[(ii)\;\;\;\;] if $m^+\not \equiv 0$, $\mu\in L^1(\Omega)$,
\begin{equation}\label{s3t4ty48yu}
\int_\Omega m(x)\,dx<0,
\end{equation}
and
\begin{equation}\label{LOWBO}
\lambda_1-1<\tau \int_\Omega(J\star e(x))e(x)\,dx,
\end{equation}
then \eqref{problogis} admits a nonnegative
solution $u\not\equiv0$.\end{enumerate}
\end{theorem}

Once again, in
Theorem~\ref{intm<0}, the case described in~(i)
is the one less favorable to life, since
the combination of both the resources and the pollination
is in average negative, while the case in~(ii)
gives a lower bound of the pollination parameter~$\tau$ which is
needed for the survival of the species, as quantified by~\eqref{LOWBO}.\medskip

We recall that the link
between the survival ability of a biological population
and the analysis
of the eigenvalues of a linearized problem
is a classical topic in mathematical biology,
see e.g.~\cite{Skellam, MR1478921, Berestycki-Hamel-Roques-1, Berestycki-Hamel-Roques-2, Kawasaki-Shigesada-2, Kawasaki-Shigesada-1, MR3082317, Mazari, MazariTH}
(yet, we believe that this is the first place in which
a detailed analysis of this type is carried over to the
case of mixed operators with our new
type of Neumann conditions).\medskip

In light of~\eqref{LOWBO}, a natural question
consists in quantifying the size of the first eigenvalue.
Roughly speaking, from~\eqref{LOWBO},
the smaller~$\lambda_1$, the smaller is the threshold
for the pollination guaranteeing survival, hence
configurations with
small first eigenvalues correspond to the ones of better chances
of life.

To address this problem,
since the eigenvalue~$\lambda_1=\lambda_1(m)$ depends on the resource~$m$,
it is convenient to consider an optimization problem
for $\lambda_1$ in terms of three structural parameters of
the resource~$m$, namely its minimum, its maximum and its average,
in order to detect under which conditions
on these parameters the first eigenvalue can be made
conveniently small.
More precisely, given~$\overline{m}$,
$\underline{m}\in (0,+\infty)$ and~$m_0\in (-\underline{m},0)$ 
we consider the class of resources
\begin{equation}\label{defemme}\begin{split}
\mathscr{M}=\mathscr{M}(\overline{m},\underline{m},m_0)
:=\;&\Big\lbrace m\in L^\infty(\Omega)\; {\mbox{ s.t. }}\;
\inf_\Omega m\ge -\underline{m},\quad
\sup_\Omega m\leq \overline{m},\\
&\qquad\qquad
\int_\Omega m(x)\,dx=m_0|\Omega|\quad{\mbox{ and }}\quad
m^+\not\equiv 0
\Big\rbrace.
\end{split}\end{equation}
We will also consider the smallest possible first eigenvalue
among all the resources in~$\mathscr{M}$, namely we set
\begin{equation}\label{LAMSPOTTP}
\underline{\lambda}:=\inf_{m\in\mathscr{M}}\lambda_1(m).
\end{equation}
When we want to emphasize the dependence
of~$\underline{\lambda}$ on the structural quantities~$\overline{m}$,
$\underline{m}$ and~$m_0$
that characterize~$\mathscr{M}$, we will adopt
the explicit notation~$\underline{\lambda}
(\overline{m},\underline{m},m_0)$.

Our main objective will be to detect whether or
not~$\underline{\lambda}$
can be made arbitrarily small in a number of different regimes:
we stress that the smallness of~$\underline{\lambda}$
corresponds to a choice of an optimal distribution of resources
that is particularly favorable for survival.\medskip

The first result that we
present in this direction is a general estimate
controlling~$\underline{\lambda}$ with~$O\left(  \frac{1}{\overline{m}}
\right)$, provided that the maximal hostility
of the environment
does not prevail with respect to the maximal
and average
resources. In terms of survival of the species, this is a rather
encouraging outcome, since it allows
the existence of nontrivial solutions
provided that the maximal resource is sufficiently large.
The precise result that we have is the following:

\begin{theorem}\label{dzero}
Let 
\begin{equation}\label{GARAPRE}
\frac{\underline{m}+m_0}{\underline{m}+\overline{m}}\geq d_0
\end{equation}
for some $d_0>0$.
Then,
\[
\underline{\lambda}(\overline{m},\underline{m},m_0)\leq \frac{C}{\overline{m}}
\]
for some $C=C(\Omega,d_0)>0$.
\end{theorem}

A direct consequence of Theorem~\ref{dzero}
gives that when the upper and lower bounds of
the resource are the same and get arbitrarily large,
then~$\underline{\lambda}$ gets arbitrarily small
(hence, in view of~\eqref{LOWBO},
there exists a resource distribution which is favorable to survival).
More\footnote{To avoid notational
confusion,
we reserved the name of~$m$ to
the resource in~\eqref{problogis-PRE}
and we denoted by~$\emme$ a ``free
variable''
dimensionally related to the resource.}
precisely:

\begin{cor}\label{mmtoinfty}
We have 
\[
\lim_{\emme\nearrow+ \infty} \underline{\lambda}(\emme,\emme,m_0)=0.
\]
\end{cor}

We now investigate
the behavior of~$\underline{\lambda}$
for large upper and lower bounds 
on the resource (maintaining constant the other parameters).
Interestingly, this behavior sensibly depends
on the dimension~$n$. In this setting,
we first consider the asymptotics in dimension~$n\ge2$:
we show that large upper and lower bounds are both
favorable for life for a given~$m_0<0$, 
according to the following two results:

\begin{theorem}\label{msopratoinfty}
Let $n\geq 2$. Then, 
\[
\lim_{\emme\nearrow +\infty} \underline{\lambda}(\emme,\underline{m},m_0)=0.
\]
\end{theorem}

\begin{theorem}\label{msottotoinfty}
Let $n\geq 2$. Then, 
\[
\lim_{\emme\nearrow +\infty} \underline{\lambda}(\overline{m},\emme,m_0)=0.
\]
\end{theorem}

While Theorem~\ref{msopratoinfty} is somehow intuitive
(large resources are favorable to survival),
at a first glance Theorem~\ref{msottotoinfty}
may look unintuitive,
since it seems to suggest that a
largely hostile environment is also favorable to survival:
but we remark that in Theorem~\ref{msottotoinfty},
being~$m_0$ given, an optimal strategy for~$m$
may well correspond to a very harmful environment confined
in a small portion of the domain, with a positive resource allowing
for the survival of the species.
\medskip

Quite surprisingly, the structural
analysis developed in
Theorems~\ref{msopratoinfty}  and~\ref{msottotoinfty}
is significantly different in dimension~$1$.
Indeed,
for $n=1$, we have that~$\underline{\lambda}$
does not become infinitesimal for large
upper and lower bounds on the resource,
unless the diffusion is purely nonlocal
with strongly nonlocal fractional parameter.
Namely,
we have the following two results.

\begin{theorem}\label{alpha>0}
Let $n=1$, $\alpha>0$ and $\beta \geq 0$. Then, for 
any~$\underline{m}>0$ and~$m_0\in(-\underline{m},0)$,
\begin{equation}\label{viene1}
\underline{\lambda}(\emme,\underline{m},m_0)\geq C
\end{equation}
for every $\emme>0$,
for some $C=C(\underline{m},m_0,\alpha, \beta,\Omega)>0$, and
\begin{equation}\label{viene2}
\lim_{\emme\searrow 0}\underline{\lambda}(\emme,\underline{m},m_0)=+\infty.
\end{equation}
Moreover, for any~$\overline{m}>0$ and~$m_0<0$,
\begin{equation}\label{viene3}
\underline{\lambda}(\overline{m},\emme,m_0)\geq C
\end{equation}
for every $\emme>-m_0$, for some $C=C(\overline{m},m_0,\alpha,\beta,\Omega)>0$.
\end{theorem}

\begin{theorem}\label{alpha=0}
Let $n=1$, $\alpha=0$ and $\beta>0$.

If $s\in (1/2,1)$, then, for 
any~$\underline{m}>0$ and~$m_0\in(-\underline{m},0)$
\begin{equation}\label{viene4}
\underline{\lambda}(\emme,\underline{m},m_0)\geq C
\end{equation}
for every $\emme>0$, for some
$C=C(\underline{m},m_0,\alpha, \beta,\Omega)>0$,
and
\begin{equation}\label{viene4bis}
\lim_{\emme\searrow0}
\underline{\lambda}(\emme,\underline{m},m_0)=+\infty.\end{equation}
Moreover, for any~$\overline{m}>0$ and~$m_0<0$
\begin{equation}\label{viene5}
\underline{\lambda}(\overline{m},\emme,m_0)\geq C
\end{equation}
for every $\emme>-m_0$,
for some~$C=C(\overline{m},m_0,\alpha, \beta,\Omega)>0$.

If $s\in (0,1/2]$, then,
\begin{equation}\label{LK:LK:01}
\lim_{\emme\nearrow +\infty}\underline{\lambda}(\emme,\underline{m},m_0)=0,
\end{equation}
and
\begin{equation}\label{LK:LK:02}
\lim_{\emme\nearrow +\infty}\underline{\lambda}(\overline{m},\emme,m_0)=0.
\end{equation}
\end{theorem}

An interesting feature of Corollary~\ref{mmtoinfty},
Theorems~\ref{msopratoinfty}
and~\ref{msottotoinfty},
\eqref{LK:LK:01}
and~\eqref{LK:LK:02}
in terms of real-world applications
is that their proofs are based on the explicit
constructions of suitable resources:
though perhaps not optimal,
these resources
are sufficiently well located to
ensure the maximal chances of survival for the population, 
and their explicit representation allows one to use
them concretely and to build on this specific knowledge.\medskip 

We also
think that the phenomenon detected
in Theorems~\ref{alpha>0}
and~\ref{alpha=0}
reveals an important role played by the nonlocal dispersal of
the species in dimension~$1$: indeed,
in this situation, the only configurations
favorable to survival are the ones in~\eqref{LK:LK:01}
and~\eqref{LK:LK:02}, that are induced by
purely nonlocal diffusion (that is~$\alpha=0$)
with a strongly nonlocal diffusion exponent (that is~$s\le1/2$,
corresponding to very
long flies in the underlying stochastic process).\medskip

To better visualize the results
in
Theorems~\ref{msopratoinfty}, \ref{msottotoinfty},
\ref{alpha>0} and~\ref{alpha=0}, we summarize them in Table~\ref{TABLE}.
For typographical convenience,
in Table~\ref{TABLE} we used the ``check-symbol''~\cmark~to denote the cases in which~$\underline\lambda$
gets as small as we wish (cases favorable to life)
and the ``x-symbol''~\xmark~to mark the situations in which~$\underline\lambda$
remains bounded away from zero
(cases unfavorable to life which require stronger
pollination for survival).
\medskip

\begin{table}[h!]
  \begin{center}
    
    \label{tab:table1}
    \begin{tabular}{l|c|c} 
 & \textbf{Large $\overline{m}$} & \textbf{Large $\underline{m}$}\\
      \hline
      $n\ge2$ & \cmark & \cmark\\
    $n=1$ and~$\alpha>0$ & \xmark & \xmark\\
      $n=1$, $\alpha=0$ and~$s>1/2$ & \xmark & \xmark\\
      $n=1$, $\alpha=0$ and~$s\le1/2$ & \cmark & \cmark\\
    \end{tabular}
\vskip0.3cm
\caption{\em Summarizing the results
in Theorems~\ref{msopratoinfty}, \ref{msottotoinfty},
\ref{alpha>0} and~\ref{alpha=0}.}\label{TABLE}
  \end{center}
\end{table}

We stress that the optimization of the resources plays a crucial role
in the survival results provided by Corollary~\ref{mmtoinfty},
Theorems~\ref{msopratoinfty} and~\ref{msottotoinfty}, 
and formulas~\eqref{LK:LK:01}
and~\eqref{LK:LK:02}:
that is, given~$m_0<0$, very large but {\em badly displayed} resources
may lead to non-negligible first eigenvalues
(differently from the case of optimal distribution of resources
discussed in Corollary~\ref{mmtoinfty},
Theorems~\ref{msopratoinfty}
and~\ref{msottotoinfty}, 
and formulas~\eqref{LK:LK:01}
and~\eqref{LK:LK:02}).

To state precisely this phenomenon, given~$m_0<0$ and~$\Lambda>-4m_0$, we
let
\begin{equation}\label{TB-x0} {\mathscr{M}}^\sharp_{\Lambda,m_0}:=\left\{
m\in {\mathscr{M}}(2\Lambda,2\Lambda,m_0){\mbox{ s.t. }}
\inf_\Omega m\le-\frac{\Lambda}2 {\mbox{ and }}
\sup_\Omega m\ge\frac{\Lambda}2
\right\}.\end{equation}
Roughly speaking, the resources~$m$ in ${\mathscr{M}}^\sharp_{\Lambda,m_0}$
have a prescribed average equal to~$m_0$ and attain
maximal positive and negative values
comparable with a large parameter~$\Lambda$,
and a natural question in this case is whether
large~$\Lambda$'s provide sufficient
conditions for the survival of the species.
The next result shows that this is not the case,
namely the abundance
of the resource without an optimal distribution strategy
is not sufficient for prosperity:

\begin{theorem}\label{oink8495-0243905i8uhjssjjcjnj7}
Given~$m_0<0$ and~$\Lambda>-4m_0$, we
have that
$$ \sup_{m\in{\mathscr{M}}^\sharp_{\Lambda,m_0}}\lambda_1(m)=+\infty.$$
\end{theorem}

Interestingly, the proof of Theorem~\ref{oink8495-0243905i8uhjssjjcjnj7}
will be ``constructive'', namely we will provide
an explicit example of a sequence of
badly displayed resources which make the first eigenvalue diverge:
a telling feature of this sequence is that it is highly oscillatory,
thus suggesting that a hectic and erratic alternation
of highly positive resources with very harmful surroundings
is potentially lethal for the development of the species.
\medskip

We recall that the investigation
of the roles of fragmentation and
concentration for
resources is a classical topic in mathematical biology,
and, in this sense, our result in Theorem~\ref{oink8495-0243905i8uhjssjjcjnj7}
confirms the main paradigm according to which
concentrated resources favor survival
(see e.g.~\cite{Berestycki-Hamel-Roques-1, Berestycki-Hamel-Roques-2, Lamboley-Laurain-Nadin-Privat}) --
however, there are several circumstances
in which this general paradigm is violated
and fragmentation is better than concentration,
see e.g. the
small diffusivity regime analyzed in~\cite{Lam-Liu-Lou, Mazari-Ruiz-Balet, Lou-Nagahara-Yanagida}.
In any case, the analysis
of fragmentation and
concentration for mixed operators with our Neumann condition
is, to the best of our knowledge, completely new.
\medskip

We also remark that the results presented here
are new even in the simpler cases in which
no classical diffusion and no pollination term is present
in~\eqref{problogis-PRE}, as well as in the cases
in which the death rate and the pollination
functions are constant.
\medskip

The rest of this paper is organized as follows.
In Section~\ref{SEC1} we will introduce the functional
framework in which we work and the notion of
weak solutions, also providing a new result
showing that the nonlocal Neumann condition naturally
produces functions with minimal Gagliardo seminorm
(this is a nonlocal phenomenon, which has
no counterpart in the classical setting,
and will play a pivotal role in the minimization process).

Then, in Section~\ref{EX:SECT} we prove
the existence results in Theorems~\ref{esistenza}
and~\ref{intm>0}.
In Section~\ref{sec:nuova} we study the eigenvalue problemin~\eqref{probauto},
and we give the proof of Theorem~\ref{intm<0}. Not to
overburden this paper, some technical proofs
related to the spectral theory of the problem are deferred to
the article~\cite{ANCI}.

In Section~\ref{OPT:M}, we deal with
the proofs of
Theorem~\ref{dzero},
Theorems~\ref{msopratoinfty}
and~\ref{msottotoinfty} when~$n\ge3$, and
Theorems~\ref{alpha>0} and~\ref{alpha=0}.

When~$n=2$, the proofs of Theorems~\ref{msopratoinfty}
and~\ref{msottotoinfty}
require some technical modification of logarithmic type,
hence their proofs is deferred to Appendix~\ref{KASN:ALSL}.

The proof of Theorem~\ref{oink8495-0243905i8uhjssjjcjnj7}
is contained in Section~\ref{oink8495-0243905i8uhjssjjcjnj7SS}.

Finally, Section~\ref{APPEB}
contains some
probabilistic motivations related to the diffusive
operators of mixed integer and fractional order.

\section{Functional analysis setting}\label{SEC1}

In this section we define the functional space in which we work.
First, we recall the space~$H^s_\Omega$ introduced in~\cite{dirova}
and defined as
\begin{equation}\label{DE-Hs}
H^s_\Omega:=
\left\lbrace  
u:\R^n\to\R \;{\mbox{ s.t. }}\;
u\in L^2(\Omega) \;{\mbox{ and }}\;
\iint_\Q\frac{|u(x)-u(y)|^2}{|x-y|^{n+2s}}\,dx\,dy <+\infty
\right\rbrace,
\end{equation}
where
\begin{equation*}
\Q:=\R^{2n}\setminus(\R^n\setminus\Omega)^2.\end{equation*}
As customary, by~$u\in L^2(\Omega)$
in~\eqref{DE-Hs} we mean that the restriction of the function~$u$
to~$\Omega$ belongs to~$L^2(\Omega)$ (we stress
that functions in~$H^s_\Omega$ are defined in the whole of~$\R^n$).
Also, all functions considered will be implicitly assumed
to be measurable.

Furthermore, we define 
\begin{equation}\label{Xdefab}
X_{\alpha,\beta}=X_{\alpha,\beta}(\Omega):=
\begin{dcases}
H^1(\Omega)     & {\mbox{ if }}\; \beta=0,
\\
H^s_\Omega		& {\mbox{ if }}\; \alpha=0,
\\
H^1(\Omega)\cap H^s_\Omega & {\mbox{ if }}\; \alpha \beta\neq 0.
\end{dcases}
\end{equation}
In light of this definition, $X_{\alpha,\beta}$ is a Hilbert 
space with respect to the scalar product
\begin{equation}\label{scalar}\begin{split}
(u,v)_{X_{\alpha,\beta}}&\;:=\int_\Omega u(x)v(x)\,dx+
\alpha \int_\Omega \nabla u(x) \cdot\nabla v(x)\,dx
 \\
&\qquad+\frac{\beta}{2}
\iint_\Q\frac{(u(x)-u(y))(v(x)-v(y))}{|x-y|^{n+2s}}\,dx\,dy,
\end{split}\end{equation}
for every $u,v\in X_{\alpha,\beta}$.

We also define the seminorm
\begin{equation}\label{seminorm}
[u]^2_{X_{\alpha,\beta}}:=\frac\alpha2 
\int_\Omega |\nabla u(x)|^2\,dx
+\frac{\beta}{4}
\iint_\Q\frac{|u(x)-u(y)|^2}{|x-y|^{n+2s}}\,dx\,dy.
\end{equation}

{F}rom the compact embeddings 
of the spaces $H^1(\Omega)$ and $H^s_\Omega$ (see e.g.
Corollary~7.2 in~\cite{MR2944369} when~$\alpha=0$),
we deduce the compact 
embedding of~$X_{\alpha,\beta}$ into~$L^p(\Omega)$, for every~$
p \in [1,2^*)$ if~$\alpha\neq 0$, and for every~$p \in [1,2^*_s)$ if~$\alpha =0$.
\medskip

We
say that~$u\in X_{\alpha,\beta}$
is a solution of~\eqref{problogis} if
\begin{equation}\label{WEAKSOL}
\begin{split}&
\alpha \int_\Omega \nabla u(x) \cdot\nabla v(x)\,dx 
+\frac{\beta}{2}
\iint_\Q\frac{(u(x)-u(y))(v(x)-v(y))}{|x-y|^{n+2s}}\,dx\,dy
\\&\qquad=
\int_\Omega \Big(\big(m(x)-\mu(x) u(x)\big)u(x)+\tau(x)\;J\star u(x)\Big)v(x)\,dx
\end{split}
\end{equation}
for all functions~$v\in X_{\alpha,\beta}$.
\medskip

Now we show that {\em
among all the functions in~$H^s_\Omega$,
the ones minimizing the Gagliardo seminorm are
those satisfying the nonlocal Neumann condition in~\eqref{NEU-1}}.
This is a useful result
in itself, which also clarifies the structural role
of the Neumann condition introduced in~\cite{dirova}:

\begin{theorem}\label{thcontiminimo}
Let $u:\R^n \to \R$ with~$u\in L^1(\Omega)$,
and set, for all~$x\in\R^n\setminus\overline\Omega$,
\[ E_u(x):=\int_\Omega \frac{u(z)}{|x-z|^{n+2s}}\,dz.
\]
Then, if we define 
\begin{equation}\label{EUE1}
\tilde{u}(x):=
\begin{dcases}
u(x)  &  {\mbox{ if }}x \in \Omega,
\\
\\
\displaystyle\frac{E_u(x)}{E_1(x)}  & {\mbox{ if }}x\in \R^n\setminus \overline\Omega,
\end{dcases}
\end{equation}
we have
\begin{equation}\label{CLaj9}
\iint_\Q\frac{|\tilde{u}(x)-\tilde{u}(y)|^2}{|x-y|^{n+2s}}\,dx\,dy
\leq \iint_\Q\frac{|u(x)-u(y)|^2}{|x-y|^{n+2s}}\,dx\,dy. 
\end{equation}
Also, the equality in~\eqref{CLaj9}
holds if and only if~$u$ satisfies~\eqref{NEU-1}.
\end{theorem}

\begin{proof} We remark that the notation~$E_1$
in~\eqref{EUE1} stands for~$E_u$ when~$u\equiv1$.
Moreover, without loss of generality, we can suppose that
\[
\iint_\Q\frac{|u(x)-u(y)|^2}{|x-y|^{n+2s}}\,dx\,dy <+\infty,
\]
otherwise the claim in~\eqref{CLaj9}
is obviously true.

In addition,
\begin{equation}\label{p8932AK}
\int_\Omega\int_\Omega \frac{|\tilde{u}(x)-\tilde{u}(y)|^2}{|x-y|^{n+2s}}\,dx\,dy 
=\int_\Omega\int_\Omega\frac{|u(x)-u(y)|^2}{|x-y|^{n+2s}}\,dx\,dy,
\end{equation}
so we only need to consider the integral on 
$(\R^n\setminus \Omega) \times\Omega$
(being the integral on
$\Omega\times(\R^n\setminus \Omega)$ the same).

Setting $\varphi(x):=u(x)-\tilde{u}(x)$, for every 
$y\in \R^n\setminus \overline\Omega$ we have
\begin{align}\label{contiminimo}
\int_\Omega &\frac{|u(x)-u(y)|^2}{|x-y|^{n+2s}}\,dx
=\int_\Omega \frac{|u(x)-\tilde{u}(y)-\varphi(y)|^2}{|x-y|^{n+2s}}\,dx \\
&=\int_\Omega \frac{|u(x)-\tilde{u}(y)|^2-2\varphi(y)(u(x)-\tilde{u}(y))+|\varphi(y)|^2}{|x-y|^{n+2s}}\,dx. \notag
\end{align}
Now, we observe that, for every~$y\in \R^n\setminus \overline\Omega$,
\[
\int_\Omega \frac{u(x)-\tilde{u}(y)}{|x-y|^{n+2s}}\,dx
=E_u(y)-\frac{E_u(y)}{E_1(y)}\,E_1(y)=0.
\]
Accordingly, \eqref{contiminimo} becomes
\[ 
\int_\Omega \frac{|u(x)-u(y)|^2}{|x-y|^{n+2s}}\,dx=
\int_\Omega \frac{|\tilde{u}(x)-\tilde{u}(y)|^2+|\varphi(y)|^2}{|x-y|^{n+2s}}\,dx 
\geq \int_\Omega \frac{|\tilde{u}(x)-\tilde{u}(y)|^2}{|x-y|^{n+2s}}\,dx,
\]
for every $y\in \R^n\setminus\overline \Omega$, and the equality holds if
and only if $\varphi(y)=0$.
Integrating over~$\R^n\setminus \Omega$
(or, equivalently, on~$\R^n\setminus\overline\Omega$), we get
\[
\int_{\R^n\setminus \Omega}\int_{\Omega}\frac{|u(x)-u(y)|^2}{|x-y|^{n+2s}}\,dx\,dy
\geq \int_{\R^n\setminus \Omega}\int_{\Omega}\frac{|\tilde{u}(x)-\tilde{u}(y)|^2}{|x-y|^{n+2s}}\,dx\,dy,
\]
and the equality holds if and only if $\varphi \equiv 0$ in 
$\R^n\setminus \Omega$.
{F}rom this observation and~\eqref{p8932AK}
we obtain~\eqref{CLaj9}, as desired.
\end{proof}

\section{Existence results and proofs of
Theorems~\ref{esistenza} and
\ref{intm>0}}\label{EX:SECT}

The proof of Theorem \ref{esistenza} is based on a minimization 
argument. More precisely, given the functional setting introduced
in Section~\ref{SEC1}
(recall in particular~\eqref{Xdefab}),
in order to deal with problem \eqref{problogis},
we consider the energy functional $\E:X_{\alpha,\beta}\to \R$
defined as
\begin{equation}\label{ENE}
\begin{split}
\E(u):=&\frac{\alpha}{2}\int_\Omega |\nabla u|^2\,dx+
\frac{\beta}{4}\iint_\Q
\frac{|u(x)-u(y)|^2}{|x-y|^{n+2s}}\,dx\,dy  \\
&+\int_\Omega\left(\frac{\mu|u|^3}{3}-\frac{m u^2}{2}
-\frac{\tau u(J\star u)}{2}\right)\,dx.
\end{split}\end{equation}

As a technical remark,
we observe that our objective here
is to distinguish between trivial
and nontrivial solutions, to detect appropriate
conditions for the survival of the solutions
and we do not indulge in the distinction
nonnegative and nontrivial versus strictly positive
solutions. For the reader interested
in this point, we mention however that, under
appropriate conditions, one could develop
a regularity theory (see e.g. Theorems~3.1.11
and~3.1.12 in~\cite{MR1911531})
that allows the use of a strong maximum principle
for smooth solutions (see e.g. Theorem~3.1.4
in~\cite{MR1911531}).\medskip

Now, we prove that the functional in~\eqref{ENE} is the one associated
with~\eqref{problogis}: 

\begin{lem}\label{soluzione}
The Euler-Lagrange equation associated to the energy functional $\E$
introduced in~\eqref{ENE}
at a non-negative function $u$ is \eqref{problogis}.
\end{lem}

\begin{proof}
We compute the first variation of $\E$, and we focus on the 
convolution term in~\eqref{ENE}
(being the computation for the other terms standard, see in particular
Proposition~3.7 in~\cite{dirova} to deal with the term
involving the Gagliardo seminorm, which is the one producing
the nonlocal Neumann condition).

For this, we set
\[
\mathcal{J}(u):=\frac\tau2\int_\Omega u(x) (J\star u(x))\,dx.
\]
For any $\phi\in X_{\alpha,\beta}$ and $\varepsilon\in (-1,1)$, we have
\begin{align*}
&\mathcal{J}(u+\varepsilon\phi) \\
&=\frac{\tau}{2}\int_\Omega (u+\varepsilon\phi)(x)
(J\star (u+\varepsilon\phi))(x)\,dx \\
&=\frac{\tau}{2}\int_\Omega u(x)(J\star u)(x)
+\varepsilon\big[u(x)(J\star \phi)(x)+\phi(x)(J\star u)(x) \big]
+\varepsilon^2\phi(x)(J\star \phi)(x)\,dx.
\end{align*}
Accordingly,
\begin{equation}\label{depsilon}
\frac{d\mathcal{J}}{d\varepsilon}(u+\varepsilon\phi)\Big|_{\ve=0}
=\frac{\tau}{2}\int_\Omega u(x)
(J\star \phi)(x)+\phi(x)(J\star u)(x)\,dx.
\end{equation}
Now, since $J$ is even (recall~\eqref{Jeven}), we see that
\begin{align*}
\int_\Omega u&(x)(J\star \phi)(x)\,dx=
\int_\Omega u(x)\left(\int_\Omega J(x-y)\phi(y)\,dy \right)\,dx \\
&=\int_\Omega \phi(y)\left(\int_\Omega J(y-x)u(x)\,dx \right)\,dy
=\int_\Omega \phi(x)(J\star u)(x)\,dx.
\end{align*}
Using this in \eqref{depsilon} we obtain that
\[
\frac{d\mathcal{J}}{d\varepsilon}(u+\varepsilon\phi)\Big|_{\ve=0}
=\tau \int_\Omega \phi(x)(J\star u)(x)\,dx,
\]
which concludes the proof.
\end{proof}

As a consequence of Lemma \ref{soluzione}, to find solutions 
of~\eqref{problogis},
we will consider the minimizing problem for the
functional~$\E$ in~\eqref{ENE}.
First, we show the following useful inequality:

\begin{lem}
Let $v,w\in L^2(\Omega)$. Then
\begin{equation}\label{disugconv}
\int_\Omega |v(x)|\,\big|(J\star w)(x)\big|\,dx\leq \|v\|_{L^2(\Omega)}
\|w\|_{L^2(\Omega)}.
\end{equation}
\end{lem}

\begin{proof}
By the Cauchy-Schwarz Inequality, we have
\begin{equation}\label{9tbugerghfgr}
\int_\Omega |v(x)|\,\big|(J\star w)(x)\big|\,dx\leq \|v\|_{L^2(\Omega)}
\|J\star w\|_{L^2(\Omega)}.
\end{equation}
Now, using the Young Inequality for convolutions with exponents 1
and 2 (see e.g. Theorem~9.1 in~\cite{MR3381284}), we obtain
\[
\|J\star w\|_{L^2(\Omega)}=\|J* (w\chi_\Omega)\|_{L^2(\R^n)}
\leq \|J\|_{L^1(\R^n)}\|w\chi_\Omega\|_{L^2(\R^n)}=
\|w\|_{L^2(\Omega)},
\]
where~\eqref{JL1} has been also used.
This and~\eqref{9tbugerghfgr} give~\eqref{disugconv}, as desired.
\end{proof}

We are now able to provide a minimization argument for
the functional in~\eqref{ENE}:

\begin{prop}\label{minimizer}
Assume that $m\in L^q(\Omega)$, for some
$q\in \big(\underline{q},+\infty\big]$, where~$\underline{q}$
has been introduced in~\eqref{qbar},
and that 
\begin{equation}\label{condmu}
(m+\tau)^3 \mu^{-2} \in L^1(\Omega).\end{equation}
Let also
\[
p:=\frac{2q}{q-1}.
\]
Then, the functional~$\E$ in~\eqref{ENE}
attains its minimum in~$X_{\alpha,\beta}$.
The minimal value is the same as the one occurring
among the functions $u\in L^p(\Omega)$
for which
\[
\iint_\Q
\frac{|u(x)-u(y)|^2}{|x-y|^{n+2s}}\,dx\,dy<+\infty
\]
and such that $\Ns u=0$ a.e. outside $\Omega$. 

Moreover, there exists a nonnegative minimizer~$u$, 
and it is a solution of \eqref{problogis}.
\end{prop}

\begin{proof}
First, we notice that $p\in \left[2, \frac{2\underline{q}}{\underline{q}-1}\right)$ and 
\begin{equation}\label{2pm}
\frac{2}{p}+\frac{1}{q}=1.
\end{equation}
By \eqref{disugconv} we have that
\begin{equation}\label{disuconv2}
\int_\Omega \frac{\tau u(J\star u)}{2}\,dx\leq 
\frac{\tau}{2}\|u\|_{L^2(\Omega)}\|u\|_{L^2(\Omega)}=\frac\tau2\,
\int_\Omega |u(x)|^2\,dx.
\end{equation}
Moreover, we use the Young Inequality with exponents $3/2$ and $3$ to 
obtain that
\[
\frac{(m+\tau)u^2}{2}
=\frac{\mu^{\frac{2}{3}}u^2}{2^\frac{4}{3}}\cdot
\frac{m+\tau}{2^{-\frac{1}{3}} \mu^\frac{2}{3}}
\leq \frac{\mu |u|^3}{6}+\frac{2}{3}\frac{|m+\tau|^3}{\mu^2}.
\]
{F}rom this and \eqref{disuconv2} we have that
\begin{equation}\begin{split}\label{o48bt49hg}&
\int_\Omega \frac{\mu |u|^3}{6}-\frac{m u^2}{2}-
\frac{\tau u(J\star u)}{2}\,dx\ge
\int_\Omega \frac{\mu |u|^3}{6}-\frac{m u^2}{2}-
\frac{\tau u^2}2\,dx\\&\qquad
\geq -\frac{2}{3} \int_\Omega \frac{|m+\tau|^3}{\mu^2}\,dx=:-\kappa.
\end{split}\end{equation}
We point out that the quantity~$\kappa$ is finite,
thanks to~\eqref{condmu}, and it does not depend on~$u$.

Recalling~\eqref{ENE}, formula~\eqref{o48bt49hg}
implies that 
\begin{equation}\label{98vbhgjgbkjfhb}
\E(u)\geq \frac{\alpha}{2}\int_\Omega |\nabla u|^2\,dx 
 +\frac{\beta}{4}\iint_\Q
\frac{|u(x)-u(y)|^2}{|x-y|^{n+2s}}\,dx\,dy
+\int_\Omega \frac{\mu |u|^3}{6}\,dx -\kappa.
\end{equation}

Now, we take a minimizing sequence $u_j$, and we observe that, in light of Theorem~\ref{thcontiminimo}, we can assume that
\begin{equation}\label{4hfdjvDRC435Y}
{\mbox{$\Ns u_j=0$ in~$\in \R^n\setminus\overline\Omega$, for every~$j\in\N$.}}\end{equation}
We can also suppose that
\begin{align*}
0&=\E(0)\geq \E(u_j) \\
&\geq \frac{\alpha}{2}\int_\Omega |\nabla u_j|^2\,dx
+\frac{\beta}{4}\iint_\Q
\frac{|u_j(x)-u_j(y)|^2}{|x-y|^{n+2s}}\,dx\,dy
+\int_\Omega \frac{\mu |u_j|^3}{6}\,dx -\kappa,
\end{align*}
where~\eqref{98vbhgjgbkjfhb} has been also exploited.
This implies that
\[ 
\frac{\alpha}{2}\int_\Omega |\nabla u_j|^2\,dx
+\frac{\beta}{4}\iint_\Q
\frac{|u_j(x)-u_j(y)|^2}{|x-y|^{n+2s}}\,dx\,dy
+\int_\Omega \frac{\mu |u_j|^3}{6}\,dx \leq \kappa.
\]
As a consequence,
\begin{equation}\label{9847bvbjbvdsjvdsao9t8496}
\frac{\alpha}{2}\int_\Omega |\nabla u_j|^2\,dx
+\frac{\beta}{4}\iint_\Q
\frac{|u_j(x)-u_j(y)|^2}{|x-y|^{n+2s}}\,dx\,dy\le\kappa.
\end{equation}
Moreover, by the
H\"older Inequality with exponents~$3/2$ and~$3$,
\begin{eqnarray*}
&& \|u_j\|_{L^2(\Omega)}^2\le \left(\int_\Omega|u_j|^3\,dx\right)^{2/3}
|\Omega|^{1/3}\le  \left(\int_\Omega \frac{\underline{\mu}|u_j|^3}6\,dx\right)^{2/3} \frac{6^{2/3}
|\Omega|^{1/3}}{\underline{\mu}^{2/3}}\\&&\qquad
\le  \left(\int_\Omega \frac{\mu|u_j|^3}6\,dx\right)^{2/3} \frac{6^{2/3}
|\Omega|^{1/3}}{\underline{\mu}^{2/3}}\le 
\frac{6^{2/3}
|\Omega|^{1/3}\kappa}{\underline{\mu}^{2/3}}.
\end{eqnarray*}
{F}rom this and~\eqref{9847bvbjbvdsjvdsao9t8496}, and
using compactness  arguments, we can assume, up to a subsequence, that $u_j$
converges to some $u\in L^p(\Omega)$ (for every~$p\in[1,2^*_s)$ if~$\alpha=0$, and for every~$p\in[1,2^*)$ if~$\alpha\neq0$, see e.g. 
Corollary~7.2 in~\cite{MR2944369})
and a.e. in $\Omega$, and also~$|u_j|\le h$ for some~$h\in L^p(\Omega)$ for every~$j\in\N$ (see e.g. Theorem~IV.9
in~\cite{MR697382}).

Hence, if~$x\in \R^n\setminus \overline\Omega$,
by the Dominated Convergence Theorem,
$$\int_\Omega \frac{u_j(y)}{|x-y|^{n+2s}}\,dy
\longrightarrow
\int_\Omega \frac{u(y)}{|x-y|^{n+2s}}\,dy,
$$
as~$j\nearrow+\infty$. Accordingly,
in light of~\eqref{4hfdjvDRC435Y},
when $x\in \R^n\setminus \overline\Omega$, we have
\begin{equation}\label{090o906} 
u_j(x)=\frac{\displaystyle\int_\Omega \frac{u_j(y)}{|x-y|^{n+2s}}\,dy}{\displaystyle\int_\Omega \frac{dy}{|x-y|^{n+2s}}}
\longrightarrow 
\frac{\displaystyle\int_\Omega \frac{u(y)}{|x-y|^{n+2s}}\,dy}{
\displaystyle\int_\Omega \frac{dy}{|x-y|^{n+2s}}}=:u(x),
\end{equation}
as~$j\nearrow+\infty$ (we stress that till now~$u$ was only defined in~$\Omega$,
hence the last step in~\eqref{090o906} is instrumental to define~$u$
also outside~$\Omega$). As a consequence, we obtain that~$u_j$
converges a.e. in $\R^n$.

Now, recalling \eqref{2pm}, we have that
\begin{align*}
\limsup_{j\nearrow +\infty}&\left|\int_\Omega m(u_j^2-u^2)\,dx\right|\le
\limsup_{j\nearrow +\infty}\int_\Omega |m(u_j^2-u^2)|\,dx\\&=
\limsup_{j\nearrow +\infty}\int_\Omega |m(u_j-u)(u_j+u)|\,dx \\
&\leq \limsup_{j\nearrow +\infty} \|m\|_{L^q(\Omega)}
\|u_j-u\|_{L^p(\Omega)}\|u_j+u\|_{L^p(\Omega)}=0,
\end{align*}
so that
\[
\lim_{j\nearrow +\infty}\int_\Omega m(u_j^2-u^2)\,dx=0.
\]
Also,
\begin{equation}\label{Jstar1}
\int_\Omega \big(u_j(J\star u_j)-u(J\star u)\big)\,dx=
\int_\Omega (u_j-u)(J\star u_j)\,dx+ 
\int_\Omega (J\star u_j-J\star u)u\,dx.
\end{equation}
Using \eqref{disugconv} with $v:=u_j-u$ and $w:=u$, we obtain
\begin{equation}\label{Jstar2}
\limsup_{j\nearrow +\infty}\int_\Omega |u_j-u|\,\big|J\star u_j\big|\,dx
\leq \limsup_{j\nearrow +\infty}\|u_j-u\|_{L^2(\Omega)}
\|u_j\|_{L^2(\Omega)}=0.
\end{equation}
Similarly, exploiting \eqref{disugconv} with $v:=u$ and 
$w:=u_j-u$, we have
\begin{align}\label{Jstar3}
\limsup_{j\nearrow +\infty}&\int_\Omega \big|J\star u_j-J\star u\big|\,|u|\,dx
=\limsup_{j\nearrow +\infty}\int_\Omega \big|J\star (u_j-u)\big|\,|u|\,dx \\
&\leq \limsup_{j\nearrow +\infty}\|u_j-u\|_{L^2(\Omega)}
\|u\|_{L^2(\Omega)}=0. \notag
\end{align}
{F}rom \eqref{Jstar1}, \eqref{Jstar2} and \eqref{Jstar3} we
conclude that
\begin{align*}
\lim_{j\nearrow +\infty}\int_\Omega (u_j(J\star u_j)-u(J\star u))\,dx
=0.
\end{align*}
We also have, by the Fatou Lemma and the lower semicontinuity
of the~$L^2$-norm,
\[
\liminf_{j\nearrow +\infty}\iint_\Q
\frac{|u_j(x)-u_j(y)|^2}{|x-y|^{n+2s}}\,dx\,dy
\geq \iint_\Q
\frac{|u(x)-u(y)|^2}{|x-y|^{n+2s}}\,dx\,dy,
\]
\[
\liminf_{j\nearrow +\infty}\int_\Omega |\nabla u_j|^2\,dx
\geq \int_\Omega |\nabla u
|^2\,dx
\]
and
\[
\liminf_{j\nearrow +\infty}\int_\Omega\frac{\mu |u_j|^3}{3}\,dx
\geq \int_\Omega\frac{\mu |u|^3}{3}\,dx.
\]
Gathering together these observations, we conclude that
\[
\liminf_{j\nearrow +\infty}\E(u_j)\geq \E(u),
\]
and therefore $u$ is the desired minimum.

Also, since $\E(|u|)\leq \E(u)$, we can suppose that $u$ is 
nonnegative. Finally, $u$ is a solution of~\eqref{problogis} 
thanks to Lemma~\ref{soluzione}.
\end{proof}

The claim of Theorem \ref{esistenza} follows from Propo\-si\-tion~\ref{minimizer}.
\medskip

Now, we provide the proof
of Theorem \ref{intm>0}, relying also on the existence result
in Theorem~\ref{esistenza}:

\begin{proof}[Proof of Theorem \ref{intm>0}]
Thanks to Theorem~\ref{esistenza}, we know that there exists
a nonnegative solution to~\eqref{problogis}.

We now prove the claim in~(i). For this, 
we assume that~$m\equiv 0$ and~$\tau=0$, and we argue
towards a
contradiction, supposing that there exists a nontrivial solution~$u$
of~\eqref{problogis}. 

We notice that, since~$u\ge0$ and~$\mu\ge\underline{\mu}>0$ in~$\Omega$,
\[
\int_\Omega \mu u^3\,dx >0.
\]
As a consequence, taking~$v:=u$ in~\eqref{WEAKSOL} we obtain that
\[
0\leq \alpha\int_\Omega |\nabla u|^2\,dx+ \frac{\beta}{2}\iint_\Q
\frac{|u(x)-u(y)|^2}{|x-y|^{n+2s}}\,dx\,dy
=-\int_\Omega \mu u^3\,dx<0,
\]
which is a contradiction, and therefore the claim in~(i) is proved.

Now we deal with the claim in~(ii). 
{F}rom Theorem~\ref{esistenza} we know that there exists a nonnegative
solution~$u$
to~\eqref{problogis} which is obtained by the minimization of
the functional~$\E$ in~\eqref{ENE} (recall Propo\-si\-tion~\ref{minimizer}).
We claim that
\begin{equation}\label{fiertyerugs7856PRE}
{\mbox{$u$ does not vanish identically.}}\end{equation}
To prove this, we show that
\begin{equation}\label{fiertyerugs7856}
{\mbox{$0$ is not a minimizer for~$\E$.}}\end{equation}
For this, we consider the constant function~$v\equiv1$ and a small parameter~$\varepsilon>0$.
Then
\begin{align*}
\E(\varepsilon v)&=-
\frac{\varepsilon^2}{2}
\left[ 
\int_\Omega m 
+\tau (J\star 1)\,dx \right]
+\frac{\varepsilon^3}{3} \int_\Omega \mu\,dx \\
&\leq -c_1\varepsilon^2+ c_2 \varepsilon^3,
\end{align*}
where
\[
c_1:=\frac{1}{2}\int_\Omega m +\tau\;(J\star 1)\,dx
\qquad{\mbox{ and }}\qquad
c_2:=\frac{1}{3}\|\mu\|_{L^1(\Omega)}.
\]
We remark that~$c_1>0$, thanks to~\eqref{orerhgbfvasnm1},
and~$c_2\in(0,+\infty)$,
in light of~\eqref{orerhgbfvasnm2}.
Then, for small $\varepsilon$ we have $\E(\varepsilon v)<0=\E(0)$.
This implies~\eqref{fiertyerugs7856}, which in turn
proves~\eqref{fiertyerugs7856PRE}.
\end{proof}

\section{Analysis of the eigenvalue problem in~\eqref{probauto} and proof of Theorem~\ref{intm<0}}
\label{sec:nuova}

In this section we focus on the proof of Theorem~\ref{intm<0}. For this, we need to
exploit the analysis of
the eigenvalue problem in~\eqref{probauto}
(some technical details are deferred to the article~\cite{ANCI}
for the reader's convenience).

The first result towards the proof of Theorem~\ref{intm<0}
concerns the existence of two unbounded sequences of eigenvalues,
one positive and one negative:

\begin{prop}\label{PROAUTOVA}
Let
\begin{equation}\label{feuwtywvv123445}
m\in L^q(\Omega),\quad {\mbox{for some $q \in \big(\underline{q},+\infty\big]$,}}
\end{equation}
where~$\underline{q}$ is given in~\eqref{qbar}.
Suppose that~$m^+$, $m^-\not\equiv 0$ and that
\begin{equation}\label{yt5645867600000}
\int_\Omega m(x)\,dx\neq 0.\end{equation}
Then, problem \eqref{probauto} admits two unbounded sequences of 
eigenvalues:
\[
\cdots\le\lambda_{-2}\leq \lambda_{-1}<\lambda_0=0
<\lambda_1\leq 
\lambda_2 \le\cdots\;\;.
\]
In particular, if 
$$\int_\Omega m(x)\,dx<0,$$ then
\begin{equation}\label{lopouygbv}
\lambda_1=\min_{u\in X_{\alpha,\beta}}
\left\lbrace [u]^2_{X_{\alpha,\beta}}\, {\mbox{ s.t. }}
\int_\Omega mu^2\,dx=1 \right\rbrace
\end{equation}
where we use the notation in~\eqref{seminorm}.
If instead
$$\int_\Omega m(x)\,dx>0,$$ then
\[
\lambda_{-1}=-\min_{u\in X_{\alpha,\beta}}
\left\lbrace [u]^2_{X_{\alpha,\beta}} \,{\mbox{ s.t. }}
\int_\Omega mu^2\,dx=-1 \right\rbrace.
\]
\end{prop}

The proof of Proposition~\ref{PROAUTOVA} is contained
in~\cite{ANCI}.

The first positive eigenvalue $\lambda_1$, as given by
Proposition~\ref{PROAUTOVA}, has the following 
properties:

\begin{prop}\label{prop:lambda}
Let~$m\in L^q(\Omega)$, for some $q \in \big(\underline{q},+\infty\big]$,
where~$\underline{q}$ is given in~\eqref{qbar}.
Suppose that~$m^+\not\equiv 0$ and
$$\int_{\Omega} m\,dx<0.$$
Then,
the first positive eigenvalue $\lambda_1$ of \eqref{probauto} is 
simple, and the first eigenfunction $e$ can be taken such that~$e\ge0$.

A similar statement holds if $m^-\not\equiv 0$ and
$$ \int_{\Omega} m\,dx>0.$$
\end{prop}

See \cite{ANCI} for the proof of
Proposition~\ref{prop:lambda}.

With this, we are now ready to give the proof of Theorem \ref{intm<0}:

\begin{proof}[Proof of Theorem \ref{intm<0}]
Thanks to Theorem~\ref{esistenza}, we know that there exists
a nonnegative solution to~\eqref{problogis}.

We first prove the claim in~(i). 
For this, we assume that $m\le -\tau$, and we suppose by 
contradiction that there exists a nontrivial solution~$u$ of~\eqref{problogis}. 

We observe that, applying~\eqref{disugconv} with~$v:=u$
and~$w:=u$, 
\begin{equation}\label{frietjbjcvjds}
 \tau \int_\Omega u\,\big(J\star u\big)\,dx
\leq \tau\|u\|_{L^2(\Omega)}^2=\tau\int_\Omega u^2\, dx.
\end{equation}
Hence, taking~$u$ as a test function in~\eqref{WEAKSOL},
using \eqref{frietjbjcvjds} and recalling that~$u\ge0$
and~$\mu\ge\underline\mu$, we get
\begin{align*}
0
&\leq \alpha\int_\Omega|\nabla u|^2\,dx+ \frac{\beta}{2}\iint_\Q
\frac{|u(x)-u(y)|^2}{|x-y|^{n+2s}}\,dx\,dy \\
&=\int_\Omega (m-\mu u)u^2\,dx
+\tau\int_\Omega (J\star u)u\,dx \\
&\le -\tau\int_\Omega u^2\,dx -\underline{\mu}\int_\Omega
u^3\,dx
+\tau\int_\Omega u^2\, dx\\
&<0.
\end{align*}
This is a contradiction, whence the first claim is proved.

Now we show the claim in~(ii). 
{F}rom Theorem~\ref{esistenza} we know that there exists a nonnegative
solution~$u$ to~\eqref{problogis} which is obtained by the minimization of
the functional~$\E$ in~\eqref{ENE} (recall Propo\-si\-tion~\ref{minimizer}).
We claim that
\begin{equation}\label{fiertyerugs7856PREbis}
{\mbox{$u$ does not vanish identically.}}\end{equation}
To prove this, we show that
\begin{equation}\label{fiertyerugs7856bis}
{\mbox{$0$ is not a minimizer for~$\E$.}}\end{equation}
For this, we take an
eigenfunction~$e$ associated to the first positive eigenvalue~$\lambda_1$,
as given by Proposition~\ref{prop:lambda}.
Namely, we take~$e\in X_{\alpha,\beta}$ such that
\begin{equation}\label{fiertyerugs7856bis33}
\alpha \int_\Omega \nabla e\cdot\nabla v\,dx
+\frac\beta2 \iint_\Q\frac{(e(x)-e(y))(v(x)-v(y))}{|x-y|^{n+2s}}\,dx\,dy 
=\lambda_1\int_\Omega mev\,dx,
\end{equation}
for every~$v\in X_{\alpha,\beta}$.

By taking~$v:=e$ in~\eqref{fiertyerugs7856bis33}, we obtain that
\begin{equation}\label{o4t349yhjfbdfew} \alpha \int_\Omega |\nabla e|^2\,dx
+\frac\beta2 \iint_\Q\frac{|e(x)-e(y)|^2}{|x-y|^{n+2s}}\,dx\,dy 
=\lambda_1\int_\Omega me^2\,dx.
\end{equation}

We also remark that, thanks to~\eqref{s3t4ty48yu}, we can use the characterization
of~$\lambda_1$ given in formula~\eqref{lopouygbv} of Proposition~\ref{PROAUTOVA},
and hence we can normalize~$e$ in such a way that
\begin{equation}\label{orauso}
\int_\Omega me^2\,dx=1.\end{equation}

By Corollary~1.4
in~\cite{ANCI}, we know that
\begin{equation}\label{jiehgvsdjb}
{\mbox{$e$ is bounded.}}\end{equation}

We also take $\varepsilon>0$. Then, by~\eqref{o4t349yhjfbdfew} and~\eqref{orauso},
\begin{align*}
\E(\varepsilon e) =\;&
\frac{\varepsilon^2}{2}\Bigg[\alpha\int_\Omega|\nabla e|^2\,dx
+\frac{\beta}{2} \iint_\Q
\frac{|e(x)-e(y)|^2}{|x-y|^{n+2s}}\,dx\,dy\\
&\qquad
-\int_\Omega m e^2\,dx -\int_\Omega \tau (J\star e)e\,dx
\Bigg]
+\frac{\varepsilon^3}{3}\int_\Omega \mu e^3\,dx \\
=\;&
\frac{\varepsilon^2}{2}\left[
(\lambda_1-1)\int_\Omega m e^2\,dx -\int_\Omega
\tau (J\star e)e\,dx\right]
+\frac{\varepsilon^3}{3}\int_\Omega \mu e^3\,dx \\
= \;&\frac{\varepsilon^2}{2}\left[
(\lambda_1-1) -\int_\Omega
\tau (J\star e)e\,dx\right]
+\frac{\varepsilon^3}{3}\int_\Omega \mu e^3\,dx 
\\
=& -\frac{c_1}2 \varepsilon^2+c_2 \frac{\varepsilon^3}3,
\end{align*}
where 
\[
c_1:=1-\lambda_1
+\tau\int_\Omega(J\star e)e\,dx
\qquad{\mbox{ and }}\qquad
c_2:=\int_\Omega \mu e^3\,dx .
\]
We notice that~$c_1>0$, thanks to~\eqref{LOWBO}, and~$c_2\in\R$,
in light of~\eqref{jiehgvsdjb}.
As a consequence, 
for small $\varepsilon$ we have that
$\E(\varepsilon e)<0=\E(0)$, which proves~\eqref{fiertyerugs7856bis}.
In turn, this implies~\eqref{fiertyerugs7856PREbis}, thus completing the
proof of~(ii).
\end{proof}

\section{Optimization on $m$ and
proofs of Theorems~\ref{dzero},
\ref{msopratoinfty}, \ref{msottotoinfty},
\ref{alpha>0} and~\ref{alpha=0}}\label{OPT:M}

This section is devoted to the understanding of the optimal configuration
of the resource~$m$, which is based on the analysis of the minimal
eigenvalue problem given in~\eqref{LAMSPOTTP}.

First of all, we will see that the optimal resource distribution
attaining the minimal eigenvalue in~\eqref{LAMSPOTTP}
is of bang-bang type, namely
concentrated on its minimal and maximal values~$\underline{m}$
and~$\overline{m}$.
This property is
based on the so called ``bathtub principle'',
see
Lemma~3.3 in~\cite{derl}
(or~\cite{lilo, Lou-Yanagida}), that we recall here for the convenience of the reader:

\begin{lem}\label{bathtub}
Let $f\in L^1(\Omega)$ and~${\mathscr{M}}$ be as in~\eqref{defemme}.
Then, the maximization problem 
\[
\sup_{m\in \mathscr{M}} \int_\Omega fm\,dx
\]
is attained by a suitable~$m\in\mathscr{M}$ given by
\[
m:=\overline{m}\chi_D-\underline{m}\chi_{\Omega\setminus D},
\]
for some subset $D\subset \Omega$
such that
\begin{equation}\label{misuraD}
|D|=\frac{\underline{m}+m_0}{\underline{m}+\overline{m}}|\Omega|.
\end{equation}
\end{lem}

We now show that, in the light of Lemma \ref{bathtub},
to optimize the eigenvalue~$\lambda_1$ in~\eqref{LAMSPOTTP}, we have to
consider $m\in\mathscr{M}$ of bang-bang type. More precisely, we define
\begin{equation}\label{tildem}\begin{split}
\tilde{\mathscr{M}}=
\tilde{\mathscr{M}}(\overline{m},\underline{m},m_0):=\;&
\Bigg\lbrace m\in\mathscr{M} {\mbox{ s.t. }}\,
m:=\overline{m}\chi_D-\underline{m}\chi_{\Omega\setminus D},\\&\qquad
{\mbox{for some subset $D\subset \Omega$ with}}\quad
|D|=\frac{\underline{m}+m_0}{\underline{m}+\overline{m}}|\Omega|
\Bigg\rbrace,\end{split}
\end{equation}
and we have the following result:

\begin{prop}\label{prop:bang}
We have that
\begin{equation*}
\underline{\lambda}=\inf_{m\in\tilde{\mathscr{M}}}\lambda_1(m).
\end{equation*}
\end{prop}

\begin{proof}
We define 
\[
\tilde{\lambda}:=\inf_{m\in\tilde{\mathscr{M}}}\lambda_1(m).
\]
and we claim that 
\begin{equation}\label{do3dnverit44y3GGGG}
\underline{\lambda}=\tilde{\lambda}.\end{equation}
To this end, we observe that, since~$\tilde{\mathscr{M}}\subset 
\mathscr{M}$, we have that
\begin{equation}\label{do3dnverit44y3GGGG1}
\underline{\lambda}\leq\tilde{\lambda}.\end{equation}
Moreover, by the definition of~$\underline{\lambda}$ in~\eqref{LAMSPOTTP},
we have that
for every~$\varepsilon>0$ there exists~$m_\varepsilon \in\mathscr{M}$
such that $\underline{\lambda}+\varepsilon\geq 
\lambda_1(m_\varepsilon)$. Then, we denote by $e_\varepsilon$ the
nonnegative eigenfunction associated to~$\lambda_1(m_\varepsilon)$,
and we conclude that
\begin{equation}\label{IUTD45}
\underline{\lambda}+\varepsilon \geq
\lambda_1(m_\varepsilon)=
\frac{[e_\varepsilon]^2_{X_{\alpha,\beta}}}{
\displaystyle\int_\Omega m_\epsilon e_\varepsilon ^2\,dx}.
\end{equation}
We also observe that, in light of Lemma~\ref{bathtub},
$$\int_\Omega m_\epsilon e_\varepsilon ^2\,dx\le 
\int_\Omega\big(\overline{m}
\chi_{D_\varepsilon}-\underline{m}\chi_{\Omega\setminus {D_\varepsilon}}\big)e_\varepsilon ^2\,dx,$$
for a suitable~${D_\varepsilon}\subset\Omega$ satisfying~\eqref{misuraD}.
Plugging this information into~\eqref{IUTD45}, and letting~$m^\star_\epsilon:=
\overline{m}
\chi_{D_\varepsilon}-\underline{m}\chi_{\Omega\setminus {D_\varepsilon}}$,
we obtain that
$$ \underline{\lambda}+\varepsilon \ge 
\frac{[e_\varepsilon]^2_{X_{\alpha,\beta}}}{
\displaystyle\int_\Omega(\overline{m}
\chi_{D_\varepsilon}-\underline{m}\chi_{\Omega\setminus {D_\varepsilon}})e_\varepsilon ^2\,dx}
\geq \lambda_1 (m^\star_\epsilon )\ge \tilde{\lambda}.
$$
Hence, taking the limit as $\varepsilon$ goes to~$0$,
we get that~$\underline{\lambda}\geq\tilde{\lambda}$.
This, combined with~\eqref{do3dnverit44y3GGGG1},
establishes~\eqref{do3dnverit44y3GGGG}, as desired.
\end{proof}

We recall that many biological models
describe optimal resourcesof bang-bang type, see
e.g.~\cite{Cantrell-Cosner-1, Cantrell-Cosner-2, Lou-Yanagida, Lamboley-Laurain-Nadin-Privat, Nagahara-Yanagida, Mazari-Nadin-Privat}.\medskip

In light of Proposition~\ref{prop:bang},
from now on, when optimizing the eigenvalue~$\lambda_1(m)$
as in~\eqref{LAMSPOTTP}, 
we will
suppose that~$m$ belongs to the set~$\tilde{\mathscr{M}}$
introduced in~\eqref{tildem}.

Now we provide the proof of Theorem \ref{dzero}.

\begin{proof}[Proof of Theorem \ref{dzero}]
We take a ball $B\subset \Omega$ such that 
\begin{equation}\label{GARA}
|B|\leq \frac{d_0}{2}|\Omega|.
\end{equation}
We can assume, up to a translation, that 
$\Omega\subset \lbrace x_n>0 \rbrace$, and, for every $\xi\ge0$,
we define the set
\[
\Omega_\xi:=B \cup(\lbrace x_n<\xi \rbrace\cap \Omega).
\]
We observe that $|\Omega_\xi|$ is nondecreasing with respect to
$\xi$, and we define
\[
\xi^*:=\sup \left\lbrace \xi\ge0:\, 
|\Omega_\xi|<\frac{\underline{m}+m_0}{\underline{m}+\overline{m}}|\Omega| \right\rbrace.
\]
We claim that, for every~$\underline{\xi}>0$,
\begin{equation}\label{CONTIN}
\lim_{\xi\to \underline{\xi}}|\Omega_{\xi}|=
|\Omega_{\underline{\xi}}|.\end{equation}
To this end, we first show that 
\begin{equation}\label{limpuntuale}
\lim_{\xi\to \underline{\xi}}\chi_{\Omega_{\xi}}(x)=
\chi_{\Omega_{\underline{\xi}}}(x)\qquad {\mbox{ for a.e. }}x\in\Omega.
\end{equation}
For this, we consider several cases.
If $x=(x',x_n)\in \Omega_{\underline{\xi}}$, then either $x\in B$ or 
$x_n<\underline{\xi}$. If~$x\in B$, then~$x\in \Omega_{\xi}$ for each~$\xi>0$,
and accordingly $\chi_{\Omega_{\xi}}(x)=1=\chi_{\Omega_{\underline{\xi}}}(x)$, which implies~\eqref{limpuntuale}.
If instead~$x_n<\underline{\xi}$, then there exists 
$\tilde{\xi}\in (x_n, \underline{\xi})$ such that, for every
$\xi \in (\tilde{\xi},\underline{\xi})$, we have
that~$\chi_{\Omega_{\xi}}(x)=1=\chi_{\Omega_{\underline{\xi}}}
(x)$, which proves~\eqref{limpuntuale} also in this case.

On the other hand, if $x\not \in \Omega_{\underline{\xi}}$,
then~$x\not \in B$ and~$x_n\geq \underline{\xi}$.
We notice that the set~$\lbrace x_n=\underline{\xi} \rbrace$ 
has zero Lebesgue measure, and therefore, in order to prove~\eqref{limpuntuale},
we can assume that~$x_n> \underline{\xi}$. Then, there exists
$\tilde{\xi}\in (\underline{\xi},x_n)$ such that, for every
$\xi \in (\underline{\xi},\tilde{\xi})$, we have
that~$x\not\in\Omega_\xi$, and so~$\chi_{\Omega_{\xi}}(x)=0=\chi_{\Omega_{\underline{\xi}}}(x)$. This completes the proof of~\eqref{limpuntuale}.

By~\eqref{limpuntuale} and the Dominated Convergence 
Theorem we obtain \eqref{CONTIN}, as desired.

We also
notice that if~$\xi=0$, then~$\Omega_\xi=B$, and therefore,
by~\eqref{GARAPRE} and~\eqref{GARA},
$$ |\Omega_\xi|=|B|\leq \frac{d_0}{2}|\Omega|\le
\frac{\underline{m}+m_0}{2\big(\underline{m}+\overline{m}\big)}
|\Omega|.
$$
This and the continuity statement in~\eqref{CONTIN} guarantee
that~$\xi^*>0$.

Moreover, the continuity in~\eqref{CONTIN} implies that
\begin{equation}\label{POI875i76}
|\Omega_{\xi^*}|=\frac{\underline{m}+m_0}{
\underline{m}+\overline{m}}|\Omega|.
\end{equation}

Now, we set~$D:=\Omega_{\xi^*}$, and we observe that~$D$
satisfies~\eqref{misuraD}, thanks to~\eqref{POI875i76}. Also,
we take~$v\in C^\infty_0(B)$, with~$v\not\equiv0$. Then, recalling that~$B\subset D$,
\[
\underline{\lambda}\leq 
\frac{[v]^2_{X_{\alpha,\beta}}}{\displaystyle
\int_\Omega\big(\overline{m}\chi_D-\underline{m}\chi_{\Omega\setminus D}\big)v^2\,dx}
=\frac{[v]^2_{X_{\alpha,\beta}}}{\overline{m}\displaystyle\int_B v^2\,dx}
\le\frac{C}{\overline{m}},
\]
for some positive constant~$C$ depending on~$\Omega$ and~$d_0$.
This completes the proof of Theorem~\ref{dzero}.
\end{proof}

With the aid of Theorem~\ref{dzero} we now prove
Corollary~\ref{mmtoinfty}, by arguing as follows:

\begin{proof}[Proof of Corollary \ref{mmtoinfty}]
We notice that
$$ \lim_{\emme\nearrow +\infty}
\frac{\emme+m_0}{2\emme}=\frac12.$$
By taking~$\emme=\underline{m}=\overline{m}$, this implies that
$$ \frac{\underline{m}+m_0}{\underline{m}+\overline{m}}
= \frac{\emme+m_0}{2\emme}\ge\frac14,
$$
as long as~$\emme$ is large enough. This says that the
assumption~\eqref{GARAPRE} in Theorem~\ref{dzero}
is satisfied with~$d_0:=\frac14$, and therefore, Theorem~\ref{dzero}
gives that
$$ \underline{\lambda}\big(\emme, \emme, m_0\big)\le \frac{C}{\emme},$$
for some~$C>0$ depending only on~$\Omega$.
As a consequence,
\[
\lim_{\emme\nearrow +\infty}\underline{\lambda}\big(\emme, \emme, m_0\big)
=0, 
\]
as desired.
\end{proof}

The next goal of this section is to prove Theorem~\ref{msopratoinfty}.
For concreteness, we give here the proof for~$n\geq 3$, and we defer
the case~$n=2$ to Appendix~\ref{KASN:ALSL}.

Without loss of generality, we suppose that
\begin{equation}\label{B2cont}
B_2\subset\Omega.\end{equation}
For~$n\ge3$ and for any~$\rho\in(0,1)$,
we define the function $\varphi:\R^n\to \R$ as
\begin{equation}\label{VARPHI}
\varphi(x):=
\begin{dcases}
c_\star+1 & {\mbox{ if }}x\in B_\rho ,
\\
c_\star+ \displaystyle\frac{\rho^\gamma}{1-\rho^\gamma}
\left(\frac{1}{|x|^\gamma}-1 \right) 
& {\mbox{ if }}x\in B_1 \setminus B_\rho,
\\
c_\star &{\mbox{ if }} x\in \R^n\setminus B_1,
\end{dcases}
\end{equation}
where $\gamma>0$ and
\begin{equation}\label{VECC7}
c_\star:=-\frac{\underline{m}+m_0}{m_0}.
\end{equation}
We observe that, since $m_0\in (-\underline{m},0)$, we have that~$c_\star>0$.

Also, we set
\begin{equation}\label{setD}
D:=B_\rho.\end{equation}

The idea to prove Theorem~\ref{msopratoinfty}
is to use the function~$\varphi$ in~\eqref{VARPHI}
and the resource~$m=\overline{m}\chi_D -\underline{m}\chi_{\Omega\setminus D}$,
with~$D$ as in~\eqref{setD},
as competitors for the minimization of~$\underline{\lambda}$
in~\eqref{LAMSPOTTP}.

In this setting, we notice that, since~$m\in\tilde{\mathscr{M}}$, 
recalling~\eqref{misuraD},
$$ \rho^n|B_1|=|B_\rho|=|D|=\frac{\underline{m}+m_0}{\underline{m}
+\overline{m}}|\Omega|.
$$ 
This says that
\begin{equation}\label{EQUIVARHO}
{\mbox{sending~$\overline{m}\nearrow+\infty$ is equivalent to
sending~$\rho\searrow0$,}}\end{equation} being~$\underline{m}$, $m_0$
and~$|\Omega|$ fixed quantities in this argument.

In light of these observations, the next lemmata will be devoted to
estimate in terms of~$\rho$
the quantities involving~$\varphi$
that appear in the minimization of~$\underline{\lambda}$.

We point out that, in dimension~$n=2$, the argument to prove
Theorem~\ref{msopratoinfty} will be similar, but we will need to introduce
a logarithmic-type function as in~\eqref{vARPHI2} instead of a 
polynomial-type function as in~\eqref{VARPHI}
(as it often happens when passing from dimension~$2$ to higher
dimensions).

The first result that we have in this setting
deals with the~$H^1$-seminorm of~$\varphi$:

\begin{lem}\label{gradphito0}
Let~$n\ge3$ and~$\varphi$ be as in~\eqref{VARPHI}. Then,
\[
\lim_{\rho\searrow0}\int_\Omega |\nabla \varphi|^2\,dx= 0. 
\]
\end{lem}

\begin{proof}
By the definition of~$\varphi$ in~\eqref{VARPHI},
we have that~$\nabla \varphi\neq 0$ only if 
$x\in B_1 \setminus B_\rho$. Accordingly, using polar coordinates,
\begin{equation}\begin{split}\label{s3505u56utyk}
\int_\Omega |\nabla \varphi|^2\,dx&
=\gamma^2\left( \frac{\rho^\gamma}{1-\rho^\gamma}\right)^2
\int_{B_1\setminus B_\rho}\frac{1}{|x|^{2\gamma+2}}\,dx \\
&=\gamma^2\left( \frac{\rho^\gamma}{1-\rho^\gamma}\right)^2
\int_\rho^1 r^{n-2\gamma-3}\,dr.
\end{split}\end{equation}
Now, we point out that
\begin{equation}\label{deot4t84364hdrg}
\int_\rho^1 r^{n-2\gamma-3}\,dr\le
\begin{dcases}\displaystyle
\frac{1}{n-2\gamma-2} &{\mbox{ if }}
\displaystyle\gamma<\frac{n-2}{2},
\\ 
\displaystyle
-\log \rho &{\mbox{ if }}\displaystyle\gamma=\frac{n-2}{2},
\\
\displaystyle -\frac{\rho^{n-2\gamma-2}}{n-2\gamma-2}
&{\mbox{ if }}\displaystyle\gamma>\frac{n-2}{2}.
\end{dcases}
\end{equation}
This and~\eqref{s3505u56utyk} entail that, for every $\gamma>0$, 
\[
\lim_{\rho\searrow0}\int_\Omega |\nabla \varphi|^2\,dx=0,
\]
which conludes the proof.
\end{proof}

Now, we deal with the Gagliardo seminorm of~$\varphi$.
For this, we point out the following useful
inequality:

\begin{lem}
Let $x$, $y\in \R^n\setminus\{0\}$ and $\gamma>0$. Then, there exists $C_\gamma>0$ 
such that
\begin{equation}\label{disuggamma}
\left|\frac{1}{|x|\gamma}-\frac{1}{|y|^\gamma}\right|
\leq C_\gamma\;\frac{||x|-|y||}{\min\lbrace|x|^{\gamma+1},|y|^{\gamma+1}\rbrace}.
\end{equation}
\end{lem}

\begin{proof}
We can assume that~$|x|\geq|y|$, being the other case analogous.
In this way, formula~\eqref{disuggamma} boils down to
\begin{equation}\label{INERM}
\frac{1}{|y|\gamma}-\frac{1}{|x|^\gamma}
\leq C_\gamma \frac{|x|-|y|}{|y|^{\gamma+1}}.
\end{equation}
To prove~\eqref{INERM}, we first claim that, for every~$t\ge1$,
\begin{equation}\label{INERM2}
1-\frac{1}{t^\gamma}\leq C_\gamma(t-1),
\end{equation}
for a suitable~$C_\gamma>0$. Indeed, we set
\[
f(t):=C t +\frac{1}{t^\gamma}-(C+1),
\]
for some positive constant~$C$ (to be chosen in what follows),
and we observe that
\begin{equation}\label{feohwegdjs}
f(1)=0,\end{equation} and
\[
f'(t)=C-\frac{\gamma}{t^{\gamma+1}}\ge C-\gamma,
\] 
for any~$t\ge1$. As a result, taking~$C:=\gamma+1$,
we obtain that~$f'(t)>0$. This and~\eqref{feohwegdjs}
give that~$f(t)\ge0$ for every~$t\ge1$, which implies~\eqref{INERM2}.

Taking~$t:=\frac{|x|}{|y|}$ in~\eqref{INERM2}, we obtain that
$$1-\frac{|y|^\gamma}{|x|^\gamma}\leq C_\gamma\left(\frac{|x|}{|y|}
-1\right).
$$
Multiplying this inequality by~$\frac1{|y|^\gamma}$
we deduce~\eqref{INERM}, as desired.
\end{proof}

With this, we now estimate the Gagliardo seminorm of~$\varphi$
as follows:

\begin{lem}\label{gradsphito0}
Let~$n\ge3$ and~$\varphi$ be as in~\eqref{VARPHI}. Then,
\[
\lim_{\rho\searrow0}
\iint_\Q \frac{|\varphi(x)-\varphi(y)|^2}{|x-y|^{n+2s}}\,dx\,dy= 0.
\]
\end{lem}

\begin{proof}
In what follows, we will assume that~$\rho\le1/4$.
By the definition of~$\varphi$ in~\eqref{VARPHI}, it plainly follows that
\begin{equation}\label{PRIMO}
\iint_{B_\rho\times B_\rho}
\frac{|\varphi(x)-\varphi(y)|^2}{|x-y|^{n+2s}}\,dx\,dy=0
\end{equation}
and
\begin{equation}\label{PRIMO2}
\iint_{(\R^n\setminus B_1)\times (\R^n\setminus B_1)}
\frac{|\varphi(x)-\varphi(y)|^2}{|x-y|^{n+2s}}\,dx\,dy=0.
\end{equation}
Moreover, by the change of variable~$z:=y-x$,
\begin{eqnarray*}
&&\iint_{B_\rho\times (\R^n\setminus B_1)}
\frac{|\varphi(x)-\varphi(y)|^2}{|x-y|^{n+2s}}\,dx\,dy
=\iint_{B_\rho\times (\R^n\setminus B_1)}\frac{1}{|x-y|^{n+2s}}\,dx\,dy
\\&&\qquad\leq \int_{B_\rho}dx \int_{\R^n\setminus B_{\frac{1}{2}}} \frac{1}{|z|^{n+2s}}\,dz
\leq C \int_{B_\rho}dx=C\rho^n,
\end{eqnarray*}
for some~$C>0$.
As a consequence,
\begin{equation}\label{PRIMO3}
\lim_{\rho\searrow0}\iint_{B_\rho\times (\R^n\setminus B_1)}
\frac{|\varphi(x)-\varphi(y)|^2}{|x-y|^{n+2s}}\,dx\,dy=0.
\end{equation}

Now, if~$x\in B_1 \setminus B_\rho$ and~$y\in B_\rho$, from~\eqref{VARPHI}
we have that
\[
|\varphi(x)-\varphi(y)|^2=
\left( \frac{\rho^\gamma}{1-\rho^\gamma}\right)^2
\left(\frac{1}{|x|^\gamma}-\frac{1}{\rho^\gamma}
\right)^2.
\]
Hence, utilizing also~\eqref{disuggamma} (applied here with~$|y|:=\rho$),
\begin{equation}\begin{split}\label{deit458y4gvjksb}&
\iint_{(B_1 \setminus B_\rho)\times B_\rho}
\frac{|\varphi(x)-\varphi(y)|^2}{|x-y|^{n+2s}}\,dx\,dy\\=\;&
\left( \frac{\rho^\gamma}{1-\rho^\gamma}\right)^2 
\iint_{(B_1 \setminus B_\rho)\times B_\rho}
\left(\frac{1}{|x|^\gamma}-\frac{1}{\rho^\gamma}\right)^2
\frac{1}{|x-y|^{n+2s}}\,dx\,dy \\ \le\;&
\frac{C}{\rho^{2\gamma+2}}\,
\left( \frac{\rho^\gamma}{1-\rho^\gamma}\right)^2
\iint_{(B_1 \setminus B_\rho)\times B_\rho}
\frac{(|x|-\rho)^2}{|x-y|^{n+2s}}\,dx\,dy
\end{split}\end{equation}
We observe that, since~$x\in B_1\setminus B_\rho$ and~$y\in B_\rho$,
$$ |x|-\rho\le |x|-|y|\le |x-y|,
$$
and therefore, plugging this information into~\eqref{deit458y4gvjksb},
\begin{eqnarray*}&&
\iint_{(B_1 \setminus B_\rho)\times B_\rho}
\frac{|\varphi(x)-\varphi(y)|^2}{|x-y|^{n+2s}}\,dx\,dy\\
&\leq&
\frac{C}{\rho^{2\gamma+2}}\,
\left( \frac{\rho^\gamma}{1-\rho^\gamma}\right)^2
\iint_{(B_1 \setminus B_\rho)\times B_\rho}
|x-y|^{2-n-2s}\,dx\,dy \\
&\leq&\frac{C}{\rho^{2\gamma+2}}\,\left( \frac{\rho^\gamma}{1-\rho^\gamma}\right)^2
\int_{B_\rho}dy\int_{B_2}|z|^{2-2s-n}\,dz \\
&\leq& \frac{C}{\rho^{2\gamma+2}}\,
\left( \frac{\rho^\gamma}{1-\rho^\gamma}\right)^2
\int_{B_\rho}dy
\\&\leq&C\,\rho^{n-2},
\end{eqnarray*}
up to renaming~$C>0$ from line to line.
As a result,
\begin{equation}\label{PRIMO4}
\lim_{\rho\searrow0}
\iint_{(B_1 \setminus B_\rho)\times B_\rho}
\frac{|\varphi(x)-\varphi(y)|^2}{|x-y|^{n+2s}}\,dx\,dy=0.
\end{equation}

In addition, since~$\R^n\setminus\Omega\subset \R^n
\setminus B_2$ (recall~\eqref{B2cont}),
changing variable~$z:=y-x$ and using polar coordinates,
\begin{equation}\label{frity58676ghdjfv}\begin{split}&
\iint_{(B_1 \setminus B_\rho)\times(\R^n\setminus \Omega)}
\frac{|\varphi(x)-\varphi(y)|^2}{|x-y|^{n+2s}}\,dx\,dy\\=\;&
\left( \frac{\rho^\gamma}{1-\rho^\gamma}\right)^2 
\iint_{(B_1 \setminus B_\rho)\times (\R^n\setminus\Omega)}
\left(\frac{1}{|x|^\gamma}-1\right)^2
\frac{1}{|x-y|^{n+2s}}\,dx\,dy \\
\leq\;& \left( \frac{\rho^\gamma}{1-\rho^\gamma}\right)^2
\int_{B_1 \setminus B_\rho}\left(\frac{1}{|x|^\gamma}-1\right)^2\,dx
\int_{\R^n\setminus B_1}\frac{1}{|z|^{n+2s}}\,dz \\
\leq\; &C\left( \frac{\rho^\gamma}{1-\rho^\gamma}\right)^2
\int_{B_1 \setminus B_\rho}\left(\frac{1}{|x|^\gamma}-1\right)^2\,dx
\\
=\;&C\left( \frac{\rho^\gamma}{1-\rho^\gamma}\right)^2
\int_{B_1 \setminus B_\rho}\left(\frac{1}{|x|^{2\gamma}}-\frac{2}{|x|^\gamma}
+1\right)\,dx
\\
\leq\; &C\left( \frac{\rho^\gamma}{1-\rho^\gamma}\right)^2
\int_\rho^1\big(r^{n-2\gamma-1} +r^{n-1}\big)\,dr\\
\le\;&C\left( \frac{\rho^\gamma}{1-\rho^\gamma}\right)^2
\left[1+
\int_\rho^1 r^{n-2\gamma-1}\,dr\right]
,
\end{split}\end{equation}
possibly changing~$C>0$ from line to line.
We also remark that
\begin{eqnarray*}
&& \int_\rho^1r^{n-2\gamma-1}\,dr\le
\begin{dcases}\displaystyle
\frac{1}{n-2\gamma}
&{\mbox{ if }}\;\displaystyle\gamma<\frac{n}{2},
\\ 
\displaystyle
-\log\rho
&{\mbox{ if }}\;\displaystyle\gamma=\frac{n}{2},
\\ \displaystyle
-\frac{\rho^{n-2\gamma}}{n-2\gamma}
&{\mbox{ if }}\;\displaystyle\gamma>\frac{n}2.
\end{dcases}
\end{eqnarray*}
This and~\eqref{frity58676ghdjfv} imply that
\begin{equation}\label{PRIMO7}
\lim_{\rho\searrow0}
\iint_{(B_1 \setminus B_\rho)\times(\R^n\setminus \Omega)}
\frac{|\varphi(x)-\varphi(y)|^2}{|x-y|^{n+2s}}\,dx\,dy=0.
\end{equation}

Furthermore, recalling~\eqref{VARPHI} and
making use of~\eqref{disuggamma}, we have that
\begin{eqnarray*}
&&\iint_{(B_1\setminus B_\rho)\times(\Omega\setminus B_1)}
\frac{|\varphi(x)-\varphi(y)|^2}{|x-y|^{n+2s}}\,dx\,dy \\&=&
\left( \frac{\rho^\gamma}{1-\rho^\gamma}\right)^2 
\iint_{(B_1\setminus B_\rho)\times(\Omega\setminus B_1)}
\left(\frac{1}{|x|^\gamma}-1\right)^2
\frac{1}{|x-y|^{n+2s}}\,dx\,dy  \\
& \leq& C\left( \frac{\rho^\gamma}{1-\rho^\gamma}\right)^2 
\iint_{(B_1\setminus B_\rho)\times(\Omega\setminus B_1)}
\frac{(1-|x|)^2}{|x|^{2\gamma+2}|x-y|^{n+2s}}\,dx\,dy.\end{eqnarray*}
Hence, noticing that, for every~$x\in B_1\setminus B_\rho$
and every~$y\in\Omega\setminus B_1$,
$$ 1-|x|\le |y|-|x|\le |x-y|,
$$ we conclude that
\begin{eqnarray*}
&&\iint_{(B_1\setminus B_\rho)\times(\Omega\setminus B_1)}
\frac{|\varphi(x)-\varphi(y)|^2}{|x-y|^{n+2s}}\,dx\,dy \\
&\le&C\left( \frac{\rho^\gamma}{1-\rho^\gamma}\right)^2 
\iint_{(B_1\setminus B_\rho)\times(\Omega\setminus B_1)}
\frac{1}{|x|^{2\gamma+2}|x-y|^{n+2s-2}}\,dx\,dy
\\
& \leq & C\left( \frac{\rho^\gamma}{1-\rho^\gamma}\right)^2 
\int_{B_1\setminus B_\rho}\frac{1}{|x|^{2\gamma+2}}\,dx\\
&\le & C\left( \frac{\rho^\gamma}{1-\rho^\gamma}\right)^2 
\int_{\rho}^1r^{n-2\gamma-3}\,dr.
\end{eqnarray*}
Accordingly, recalling~\eqref{deot4t84364hdrg},
we conclude that
\begin{equation}\label{PRIMO8}
\lim_{\rho\searrow0}\iint_{(B_1\setminus B_\rho)\times(\Omega\setminus B_1)}
\frac{|\varphi(x)-\varphi(y)|^2}{|x-y|^{n+2s}}\,dx\,dy=0.
\end{equation}

We now claim that
\begin{equation}\label{PRIMO10}
\lim_{\rho\searrow0}
\iint_{(B_1\setminus B_\rho)\times(B_1\setminus B_\rho)}
\frac{|\varphi(x)-\varphi(y)|^2}{|x-y|^{n+2s}}\,dx\,dy 
=0.
\end{equation}
For this, we observe that by~\eqref{VARPHI}
\begin{eqnarray*}
&&\iint_{(B_1\setminus B_\rho)\times(B_1\setminus B_\rho)}
\frac{|\varphi(x)-\varphi(y)|^2}{|x-y|^{n+2s}}\,dx\,dy
\\&&\qquad=\left(\frac{\rho^\gamma}{1-\rho^\gamma}\right)^2
\iint_{(B_1\setminus B_\rho)\times(B_1\setminus B_\rho)}\left(
\frac1{|x|^\gamma}-\frac1{|y|^\gamma}\right)^2
\frac{\,dx\,dy}{|x-y|^{n+2s}}\\&&\qquad =2
\left(\frac{\rho^\gamma}{1-\rho^\gamma}\right)^2
\iint_{{(B_1\setminus B_\rho)\times(B_1\setminus B_\rho)}\atop{\{|x|\le|y|\}
}}\left(
\frac1{|x|^\gamma}-\frac1{|y|^\gamma}\right)^2
\frac{\,dx\,dy}{|x-y|^{n+2s}}.
\end{eqnarray*}
Hence, from \eqref{disuggamma} we get
\[
\iint_{(B_1\setminus B_\rho)\times(B_1\setminus B_\rho)}
\frac{|\varphi(x)-\varphi(y)|^2}{|x-y|^{n+2s}}\,dx\,dy
\leq C\left(\frac{\rho^\gamma}{1-\rho^\gamma}\right)^2\iint_{(B_1\setminus B_\rho)\times(B_1\setminus B_\rho)
\atop{\{|x|\leq|y|\}}}
\frac{|x-y|^{2-n-2s}}{|x|^{2\gamma+2}}\,dx\,dy,
\]
up to renaming~$C>0$.

Since $\rho\in (0,1)$, we can take an integer $k$ such that 
\begin{equation}\label{s95834v9hgfj}
\frac{1}{2^{k+1}}<\rho\leq \frac{1}{2^k}.
\end{equation}
In this way, we have that
\begin{equation}\label{depgerhojbkldfnbrrxasf}\begin{split}
&\iint_{(B_1\setminus B_\rho)\times(B_1\setminus B_\rho)}
\frac{|\varphi(x)-\varphi(y)|^2}{|x-y|^{n+2s}}\,dx\,dy\\
&\qquad\leq C\,\left(\frac{\rho^\gamma}{1-\rho^\gamma}\right)^2\iint_{ {\big(B_1\setminus B_{1/2^{k+1}}\big)\times\big(
B_1\setminus B_{1/2^{k+1}}\big)}
\atop{\{|x|\leq|y|\}} }
\frac{|x-y|^{2-n-2s}}{|x|^{2\gamma+2}}\,dx\,dy \\
&\qquad\leq
C\,\rho^{2\gamma}\, \sum_{i,j=0}^{k}
\iint_{{\big(B_{1/2^i}\setminus B_{1/2^{i+1}}\big)
\times\big(B_{1/2^j}\setminus B_{1/2^{j+1}}\big)}\atop{\{|x|\le|y|\}}}
\frac{|x-y|^{2-n-2s}}{|x|^{2\gamma+2}}\,dx\,dy.
\end{split}\end{equation}
We also observe that when~$x\not\in B_{1/2^{i+1}}$,
$y\in B_{1/2^j}$
and~$|x|\leq|y|$,
we have that
\[
\frac1{2^{i+1}}\leq |x|\leq|y|\leq \frac1{2^j},
\]
and accordingly $j\le i+1$. This implies that
\begin{align*}&
\iint_{(B_1\setminus B_\rho)\times(B_1\setminus B_\rho)
\atop{\{|x|\leq|y|\}}}
\frac{|x-y|^{2-n-2s}}{|x|^{2\gamma+2}}\,dx\,dy\\
&\qquad\leq \sum_{i=0}^{k}\sum_{j=0}^{i+1} 
\iint_{{\big(B_{1/2^i}\setminus B_{1/2^{i+1}}\big)
\times\big(B_{1/2^j}\setminus B_{1/2^{j+1}}\big)}\atop{\{|x|\ge|y|\}}}
\frac{|x-y|^{2-n-2s}}{|x|^{2\gamma+2}}\,dx\,dy \\
&\qquad\leq \sum_{i=0}^{k}\sum_{j=0}^{i+1} 
\iint_{\big(B_{1/2^i}\setminus B_{1/2^{i+1}}\big)
\times\big(B_{1/2^j}\setminus B_{1/2^{j+1}}\big)}2^{(2\gamma+2)(i+1)}
\,|x-y|^{2-n-2s}\,dx\,dy \\
&\qquad\leq \sum_{i=0}^{k}\sum_{j=0}^{i+1} 
\int_{B_{1/2^i}\setminus B_{1/2^{i+1}}}2^{(2\gamma+2)(i+1)}\,dx
\int_{B_{\frac1{2^i}+\frac1{2^j}}}|z|^{2-n-2s}\,dz \\
&\qquad\leq C\,\sum_{i=0}^{k}\sum_{j=0}^{i+1}
2^{-ni+(2\gamma+2)i+(2s-2)j} \\
&\qquad\leq C\,\sum_{i=0}^{k} 2^{(2\gamma+2-n)i}\\
&\qquad \le 
\begin{dcases}\displaystyle
C
&{\mbox{ if }} \displaystyle\gamma<\frac{n-2}{2},
\\ \displaystyle
C\,k
&{\mbox{ if }}\displaystyle \gamma=\frac{n-2}{2},
\\ \displaystyle
C\,2^{(2\gamma+2-n)k}
& {\mbox{ if }}\displaystyle\gamma>\frac{n-2}{2},
\end{dcases}\\
&\qquad \le 
\begin{dcases}\displaystyle
C
&{\mbox{ if }} \displaystyle\gamma<\frac{n-2}{2},
\\ \displaystyle
C\,|\log\rho|
&{\mbox{ if }}\displaystyle \gamma=\frac{n-2}{2},
\\ \displaystyle
C\,\rho^{n-2\gamma-2}
& {\mbox{ if }}\displaystyle\gamma>\frac{n-2}{2},
\end{dcases}
\end{align*}
up to renaming~$C>0$, where we used~\eqref{s95834v9hgfj}.

Plugging this information into~\eqref{depgerhojbkldfnbrrxasf},
we obtain~\eqref{PRIMO10}.

Putting together~\eqref{PRIMO}, \eqref{PRIMO2}, \eqref{PRIMO3},
\eqref{PRIMO4}, \eqref{PRIMO7}, \eqref{PRIMO8} and~\eqref{PRIMO10},
we obtain the desired result.
\end{proof}

We now estimate the weighted~$L^2$-norm of the auxiliary function~$\varphi$:

\begin{lem}\label{denomphi}
Let $n\geq 3$ and~$\varphi$ be as in~\eqref{VARPHI}. Then,
\[
\lim_{\rho\searrow0}
\overline{m}\int_D \varphi^2\,dx-\underline{m}
\int_{\Omega\setminus D} \varphi^2\,dx=
-\frac{\underline{m}\;(\underline{m}+m_0)}{m_0}\,
|\Omega|>0.
\]
\end{lem}

\begin{proof}
Recalling~\eqref{VECC7}, \eqref{setD} and~\eqref{misuraD}, we see that
\begin{equation}\begin{split}\label{swrperdr46y54i}&
\overline{m}\int_D \varphi^2\,dx=\overline{m}\,(c_\star+1)^2\,|D|
= \overline{m}\,\left(-\frac{\underline{m}+m_0}{m_0}+1\right)^2\,
\frac{\underline{m}+m_0}{\underline{m}+\overline{m}}\,|\Omega|\\
&\qquad = 
\frac{\overline{m}\;\underline{m}^2\;(\underline{m}+m_0)}{
m_0^2\;(\underline{m}+\overline{m})}\,|\Omega|.
\end{split}\end{equation}
Moreover, we observe that
\begin{equation}\label{quella}
 \frac{\rho^\gamma}{1-\rho^\gamma}
\frac{1}{|x|^\gamma}\,\chi_{B_1\setminus B_\rho}\le \frac{1}{1-\rho^\gamma}
\le 2,\end{equation}
as long as~$\rho$ is small enough.
As a consequence,
recalling~\eqref{VARPHI} and~\eqref{setD}, and using the Dominated
Convergence Theorem, we find that
\begin{eqnarray*}&&\lim_{\rho\searrow0}
\underline{m}\,\int_{\Omega\setminus D}\varphi^2\,dx=\lim_{\rho\searrow0}
\underline{m}\,\int_{\Omega\setminus B_1}c_\star^2\,dx+
\underline{m}\,\int_{B_1\setminus B_\rho}
\left[c_\star+\frac{\rho^\gamma}{1-\rho^\gamma}
\left(\frac{1}{|x|^\gamma}-1  \right) \right]^2\,dx \\
&&\qquad
=\underline{m}\,c_\star^2\,|\Omega\setminus B_1|
+ \underline{m}\,c_\star^2\,|B_1|=
\underline{m}\,c_\star^2\,|\Omega|
= \underline{m}\,\left(\frac{\underline{m}+m_0}{m_0}\right)^2\,|\Omega|.
\end{eqnarray*}
{F}rom this and~\eqref{swrperdr46y54i}, and recalling~\eqref{EQUIVARHO},
we conclude that
\begin{align*}&
\lim_{\rho\searrow0}
\overline{m}\,\int_D \varphi^2\,dx
-\underline{m}\,\int_{\Omega\setminus D} \varphi^2\,dx=
\lim_{\overline{m}\nearrow+\infty}
\frac{\overline{m}\;\underline{m}^2\;(\underline{m}+m_0)}{
m_0^2\;(\underline{m}+\overline{m})}\,|\Omega|
-\underline{m}\,\left(\frac{\underline{m}+m_0}{m_0}\right)^2\,|\Omega|\\
&\qquad=
\frac{\underline{m}^2\;(\underline{m}+m_0)}{
m_0^2}\,|\Omega|
-\underline{m}\,\left(\frac{\underline{m}+m_0}{m_0}\right)^2\,|\Omega|
=\frac{\underline{m}\;(\underline{m}+m_0)}{m_0^2}\left[
\underline{m}-(\underline{m}+m_0)
\right]
|\Omega|\\
&\qquad=-\frac{\underline{m}\;(\underline{m}+m_0)}{m_0}\,
|\Omega|,
\end{align*}
which is positive,
since $m_0\in (-\underline{m},0)$, as desired.
\end{proof}

We are now in the position to give the proof of Theorem~\ref{msopratoinfty}
for~$n\geq 3$.

\begin{proof}[Proof of Theorem \ref{msopratoinfty} when~$n\ge3$]
The strategy of the proof is to use the auxiliary function~$\varphi$ as defined
in~\eqref{VARPHI} and the resource~$m:=\overline{m}\chi_{D}-
\underline{m}\chi_{\Omega\setminus D}$, with~$D$ as in~\eqref{setD},
as a competitor in the minimization problem~\eqref{LAMSPOTTP}.
Indeed, in this way we find that
\begin{equation*}
\underline{\lambda}(\overline{m},\underline{m},m_0)
\leq \frac{\displaystyle\alpha\int_\Omega |\nabla \varphi|^2\,dx+
\beta\iint_\Q \frac{|\varphi(x)-\varphi(y)|^2}{|x-y|^{n+2s}}\,dx\,dy}
{\displaystyle
\overline{m}\int_D |\varphi|^2 \,dx-\underline{m}
\int_{\Omega\setminus D}|\varphi|^2\,dx}.
\end{equation*}
Moreover, Lemmata~\ref{gradphito0} and~\ref{gradsphito0}
give that
\begin{equation*}
\lim_{\rho\searrow0}\alpha\int_\Omega |\nabla \varphi|^2\,dx+
\beta\iint_\Q \frac{|\varphi(x)-\varphi(y)|^2}{|x-y|^{n+2s}}\,dx\,dy=0.
\end{equation*}
This, combined with~\eqref{EQUIVARHO}
and Lemma~\ref{denomphi}, gives the desired result.
\end{proof}

Now we deal with the proof of Theorem \ref{msottotoinfty}. The main strategy
is similar to that of the proof of Theorem \ref{msopratoinfty}, but in this setting
we introduce a different auxiliary function (and this of course
impacts the technical computations needed to
obtain the desired results). Namely, we define
\begin{equation}\label{AUXPSI}
\psi(x):=
\begin{dcases}
c_\sharp-1 & {\mbox{ if }}x\in B_\rho ,
\\ \displaystyle
c_\sharp- \frac{\rho^\gamma}{1-\rho^\gamma}
\left(\frac{1}{|x|^\gamma}-1 \right) 
& {\mbox{ if }}x\in B_1 \setminus B_\rho,
\\
c_\sharp & {\mbox{ if }}x\in\R^n\setminus B_1,
\end{dcases}
\end{equation}
where
\begin{equation}\label{sharp}
c_\sharp:=\frac{m_0-\overline{m}}{m_0}.
\end{equation}
We point out that~$c_\sharp>0$, since~$m_0<0<\overline{m}$.
We also set
\begin{equation}\label{facile}
D:=\Omega\setminus B_\rho.\end{equation}

We remark that, in this setting, since~$m\in\tilde{\mathscr{M}}$,
recalling~\eqref{misuraD},
$$ |\Omega|-| B_\rho|=|\Omega\setminus B_\rho|
=|D|=\frac{\underline{m}+m_0}{\underline{m}
+\overline{m}}|\Omega|.
$$ 
This says that
\begin{equation}\label{CHIRO}
{\mbox{sending~$\underline{m}\nearrow+\infty$ is equivalent to
sending~$\rho\searrow0$,}}\end{equation} being~$\overline{m}$, $m_0$
and~$|\Omega|$ fixed quantities in this argument.
The reader may compare
the setting in~\eqref{setD}
and~\eqref{EQUIVARHO}
with the one in~\eqref{facile} and~\eqref{CHIRO}
to appreciate the structural difference between the two frameworks.\medskip

Now, we list some useful properties of the auxiliary function $\psi$.
Noticing that the function~$\psi$ in~\eqref{AUXPSI} differs by a constant
from the function~$-\varphi$ in~\eqref{VARPHI}, we obtain the following
two results directly from Lemmata~\ref{gradphito0}
and~\ref{gradsphito0}:

\begin{lem}\label{gradpsito0}
Let $n\geq 3$ and~$\psi$ be as in~\eqref{AUXPSI}. Then,
\[ \lim_{\rho\searrow0}
\int_\Omega |\nabla \psi|^2\,dx =0.
\]
\end{lem}

\begin{lem}\label{gradspsito0}
Let $n\geq 3$ and~$\psi$ be as in~\eqref{AUXPSI}. Then
\[
\lim_{\rho\searrow0}
\iint_\Q \frac{|\psi(x)-\psi(y)|^2}{|x-y|^{n+2s}}\,dx\,dy= 0.
\]
\end{lem}

We now deal with the weighted~$L^2$-norm of the auxiliary function~$\psi$:

\begin{lem}\label{denompsi}
Let $n\geq 3$ and~$\psi$ be as in~\eqref{AUXPSI}. Then,
\[
\lim_{\rho\searrow0}\overline{m}
\int_D \psi^2\,dx-\underline{m}\int_{\Omega\setminus D} \psi^2\,dx
=-\frac{\overline{m}\,(\overline{m}-m_0)}{m_0}
\,|\Omega|>0.
\]
\end{lem}

\begin{proof}
Recalling~\eqref{AUXPSI} and~\eqref{facile}, we find that
\[
\overline{m}\,
\int_D \psi^2\,dx=\overline{m}\,\int_{\Omega\setminus B_1}
c_\sharp^2\,dx + \overline{m}\,\int_{B_1\setminus B_\rho}\left[
c_\sharp- \frac{\rho^\gamma}{1-\rho^\gamma}
\left(\frac{1}{|x|^\gamma}-1 \right)\right]^2\,dx.
\]
Hence, recalling~\eqref{quella} and using the Dominated Convergence
Theorem and~\eqref{sharp}, we deduce that
\begin{equation}\label{questa}
\lim_{\rho\searrow0}\overline{m}\,
\int_D \psi^2\,dx=\overline{m}\,c_\sharp^2\,|\Omega\setminus B_1|
+ \overline{m}\,c_\sharp^2\,| B_1|=
\overline{m}\,c_\sharp^2\,| \Omega|
=\overline{m}\,\left(\frac{m_0-\overline{m}}{m_0} \right)^2\,| \Omega|
.
\end{equation}
Moreover, recalling~\eqref{facile}, \eqref{misuraD} and~\eqref{sharp},
\[
\underline{m}\,\int_{\Omega\setminus D} \psi^2\,dx=
\underline{m}\,
(c_\sharp-1)^2|\Omega\setminus D|= 
\underline{m}\,
\left(\frac{m_0-\overline{m}}{m_0}-1\right)^2\,\frac{\overline{m}-m_0}{
\underline{m}+\overline{m}}
\,|\Omega|=\frac{\underline{m}\,
\overline{m}^2\,(
\overline{m}-m_0)}{m_0^2
(\underline{m}+\overline{m})}
\,|\Omega|.
\]
As a consequence of this and~\eqref{questa},
and recalling~\eqref{CHIRO}, we have that
\begin{eqnarray*}&&\lim_{\rho\searrow0}\overline{m}\,
\int_D \psi^2\,dx-\underline{m}\int_{\Omega\setminus D} \psi^2\,dx
= \overline{m}\,\left(\frac{m_0-\overline{m}}{m_0} \right)^2\,| \Omega|-
\lim_{\underline{m}\nearrow+\infty}\frac{\underline{m}\,
\overline{m}^2\,(
\overline{m}-m_0)}{m_0^2
(\underline{m}+\overline{m})}
\,|\Omega|\\&&\qquad=
\overline{m}\,\left(\frac{m_0-\overline{m}}{m_0} \right)^2\,| \Omega|-
\frac{\overline{m}^2\,(
\overline{m}-m_0)}{m_0^2}
\,|\Omega|=
\frac{\overline{m}\,(\overline{m}-m_0)}{m_0^2} \left[(
\overline{m}-m_0)-\overline{m}
\right]
\,|\Omega|\\&&\qquad=-\frac{\overline{m}\,(\overline{m}-m_0)}{m_0}
\,|\Omega|,
\end{eqnarray*}
which is positive, since~$m_0<0<\overline{m}$.
\end{proof}

Now we are ready to give the proof of Theorem \ref{msottotoinfty}
for $n\geq 3$.

\begin{proof}[Proof of Theorem \ref{msottotoinfty}]
The strategy of the proof is to use the auxiliary function~$\psi$ as defined
in~\eqref{AUXPSI} and the resource~$m:=\overline{m}\chi_{D}-
\underline{m}\chi_{\Omega\setminus D}$, with~$D$ as in~\eqref{facile},
as a competitor in the minimization problem~\eqref{LAMSPOTTP}.
Indeed, in this way we find that
\begin{equation*}
\underline{\lambda}(\overline{m},\underline{m},m_0)
\leq \frac{\displaystyle\alpha\int_\Omega |\nabla \psi|^2\,dx+
\beta\iint_\Q \frac{|\psi(x)-\psi(y)|^2}{|x-y|^{n+2s}}\,dx\,dy}
{\displaystyle
\overline{m}\int_D |\psi|^2 \,dx-\underline{m}
\int_{\Omega\setminus D}|\psi|^2\,dx}.
\end{equation*}
Moreover, from Lemmata~\ref{gradpsito0} and~\ref{gradspsito0}
we have that
\begin{equation*}
\lim_{\rho\searrow0}\alpha\int_\Omega |\nabla \psi|^2\,dx+
\beta\iint_\Q \frac{|\psi(x)-\psi(y)|^2}{|x-y|^{n+2s}}\,dx\,dy=0.
\end{equation*}
This, combined with~\eqref{CHIRO}
and Lemma~\ref{denompsi} implies the desired result.
\end{proof}

Having completed the cases~$n\ge3$ and deferred the case~$n=2$
to Appendix~\ref{KASN:ALSL}, we now focus on the case~$n=1$,
by providing the proofs of Theorems~\ref{alpha>0} and~\ref{alpha=0}.

For this, when $n=1$ we first establish the following
lower bound for~$\underline{\lambda}$ (as defined in~\eqref{LAMSPOTTP}):

\begin{lem}\label{n1lemma}
Let $n=1$ and $\alpha>0$. Then
\begin{equation}\label{Dec9irsi88586red9944}
\underline{\lambda}(\overline{m},\underline{m},m_0)
\geq -\,
\frac{C\,m_0^3 (\underline{m}+\overline{m})^4}{\overline{m}\underline{m}^3
(\overline{m}-m_0)^2\big(
\overline{m}(
2\underline{m}+m_0)-
\underline{m}m_0\big)},
\end{equation}
for some $C=C(\alpha, \beta, \Omega)>0$.
\end{lem}

\begin{proof}  
Without loss of generality, we can set $\alpha=2$.
We take an arbitrary resource~$m$
in the set~$\tilde{\mathscr{M}}$ defined in~\eqref{tildem}.
Moreover,
we denote by~$e$ an eigenfunction associated to the
first eigengenvalue of problem~\eqref{probauto}, that is
\begin{equation}\label{n1lambda1}
\lambda_1(m)=\displaystyle\frac{\displaystyle\int_\Omega |e'|^2\,dx+
\frac\beta4\iint_\Q \frac{|e(x)-e(y)|^2}{|x-y|^{n+2s}}\,dx\,dy}
{\displaystyle\overline{m}\int_D e^2\,dx - 
\underline{m}\int_{\Omega\setminus D}e^2\,dx}.
\end{equation}
In light of Proposition~\ref{prop:lambda} here
and Corollary~1.4
in~\cite{ANCI}, up to a sign change,
we know that $e$ is nonnegative and bounded,
and therefore we set
$$ a:=\inf_{\Omega}e\qquad{\mbox{and}}\qquad
b:=\sup_{\Omega}e.$$
By construction, we have that~$a\in[0,b]$,
and we can also
normalize~$e$ such that~$b=1$; in this way
\begin{equation}\label{SUSU}
e(x)\le1\qquad{\mbox{ for each }}x\in\Omega.
\end{equation}
We also take~$x_k$, $y_k\in\Omega$ such that
$$ e(x_k)\to a\qquad{\mbox{and}}\qquad e(y_k)\to 1$$
as~$k\nearrow+\infty$. 

We observe that
\begin{equation}\label{MS:SDKK023987233j3j44}
\begin{split}
&{\mbox{if there exist~$\bar{x}$ and~$\bar{y}$
such that~$|e(\bar{x})-e(\bar{y})|\ge\displaystyle\frac{1-a}{10}$ which
belong}}\\&{\mbox{to the same connected component of~$\Omega$, then}}\\
& {(1-a)^2}
\leq  C\,\int_\Omega|e'|^2\,dx,\qquad{\mbox{ for some $C>0$.}}
\end{split}
\end{equation}
Indeed, for~$\bar{x}$ and~$\bar{y}$ as in the assumption
of~\eqref{MS:SDKK023987233j3j44}
we have that
$$ \big( e(\bar{y})-e(\bar{x})\big)^2
=\left(\int_{\bar{x}}^{\bar{y}}e'(t)\,dt\right)^2
\le C\,\int_\Omega|e'(t)|^2\,dt,$$
for some positive~$C$.
Accordingly, we obtain the desired result in~\eqref{MS:SDKK023987233j3j44}.

Now we claim that
\begin{equation}\label{n1numeratore}
(1-a)^4
\le C\,\left(\int_\Omega |e'|^2\,dx+
\beta\iint_\Q \frac{|e(x)-e(y)|^2}{|x-y|^{n+2s}}\,dx\,dy\right),\end{equation}
for a suitable~$C>1$.

To prove this claim, we need to consider different
possibilities according to the possible lack of
connectedness of~$\Omega$. For this, we first remark that, with no loss
of generality, we can suppose that
\begin{equation}\label{544}a<1,\end{equation}
otherwise~\eqref{n1numeratore}
is obviously satisfied. 

Furthermore,
being~$\Omega$
a bounded set with~$C^1$ boundary, necessarily
it can have at most a finite number of connected components
(otherwise, there would be accumulating components,
violating the assumption in~\eqref{PRECIS:C1}).
Hence, if~$\Omega$ is not connected,
we can define~$d_0$ to be the smallest distance
between the different connected components
of~$\Omega$. We also let~$d_1$ to be the diameter of~$\Omega$
and~$d_2$ the smallest diameter of all the connected components
of~$\Omega$
(of course, $d_0$, $d_1$ and~$d_2$ are structural constants,
and the other constants are allowed to depend on them, but
we will write~$d_0$, $d_1$ and~$d_2$
explicitly in the forthcoming
computations
whenever needed to emphasize their roles).
To prove~\eqref{n1numeratore}, we distinguish two cases:
the first case is when
\begin{equation}
\begin{split}
\label{A:CA:01}& {\mbox{$\Omega$ has one connected component,}}\\&{\mbox{or it has more than one connected component, with }}\\&
\sup_{{x,y\in\Omega}\atop{|e(x)-e(y)|\ge \frac{1-a}{10}}}
\frac{|e(x)-e(y)|}{|x-y|}\ge\frac{4}{d_0},\end{split}\end{equation}
and the second case is when
\begin{equation}
\begin{split}&
\label{A:CA:02} {\mbox{$\Omega$ has more than one connected component, with }}
\\ &\sup_{{x,y\in\Omega}\atop{|e(x)-e(y)|\ge \frac{1-a}{10}}}
\frac{|e(x)-e(y)|}{|x-y|}<\frac{4}{d_0}.
\end{split}\end{equation}
Let us first discuss case~\eqref{A:CA:01}.
If~$\Omega$ has one connected component,
then we can exploit~\eqref{MS:SDKK023987233j3j44}
with~$\bar{x}:=x_k$ and~$\bar{y}:=y_k$ with~$k$ sufficiently large,
and the claim in~\eqref{n1numeratore} plainly follows.
Thus, to complete the study of~\eqref{A:CA:01}, we suppose that~$
\Omega$ is not connected
and, in the setting of~\eqref{A:CA:01}, we find~$\bar x$, $\bar y\in\Omega$ with
\begin{equation} \label{SPeispekcif}|e(\bar x)-e(\bar y)|\geq
\frac{1-a}{10}\qquad{\mbox{and}}\qquad
\frac{|e(\bar x)-e(\bar y)|}{|\bar x-\bar y|}\ge\frac{3}{d_0}.\end{equation}
In this framework, we have that
\begin{equation}\label{SPeispekcif2}
{\mbox{$\bar x$ and~$\bar y$ belong to the same connected component.}}
\end{equation}
Indeed, if not, we would have that~$|\bar x-\bar y|\ge d_0$, and thus, by~\eqref{SUSU},
$$ \frac{|e(\bar x)-e(\bar y)|}{|\bar x-\bar y|}\le
\frac{|e(\bar x)|+|e(\bar y)|}{d_0}\le\frac2{d_0},$$
which is in contradiction with~\eqref{SPeispekcif}, thus proving~\eqref{SPeispekcif2}.

Then, by~\eqref{SPeispekcif2}, we can 
exploit~\eqref{MS:SDKK023987233j3j44},
from which one deduces~\eqref{n1numeratore} in this case.

Having completed the analysis of case~\eqref{A:CA:01},
we now focus on the setting provided by case~\eqref{A:CA:02}
and we define
\begin{equation}\label{deferrre}
r:=\frac{1}{100\left(1+\displaystyle\frac{4}{d_0}\right)}\,\min\big\{ 1-a,\;d_2\big\}.
\end{equation}
We observe that, $r>0$, due to~\eqref{544},
and, if~$\vartheta\in\Omega$ with~$|\vartheta
-{x}_k|\le r$, then 
\begin{equation}\label{KSN-o204-9350-01093-1}
|e(\vartheta)-e( x_k)|\le \frac{1-a}{10}.
\end{equation}
Indeed, suppose not. Then,
the assumption in~\eqref{A:CA:02}
guarantees that
$$ \frac{4}{d_0}\ge\frac{|e(\vartheta)-e(x_k)|}{|\vartheta
-{x}_k|}\ge
\frac{|e(\vartheta)-e( x_k)|}{r},
$$
and therefore
$$ |e(\vartheta)-e(x_k)|\le
\frac{4r}{d_0}\le\frac{\displaystyle\frac{4}{d_0}}{100\left(1+\displaystyle\frac{4}{d_0}\right)}\,(1-a)\le\frac{1-a}{100},$$
against the contradiction assumption.

This proves~\eqref{KSN-o204-9350-01093-1} and similarly one can show that
if~$\tau\in\Omega$ with~$|\tau
-{y}_k|\le r$, then
\begin{equation}\label{KSN-o204-9350-01093-2}
|e(\vartheta)-e(y_k)|\le \frac{1-a}{10}.
\end{equation}
Combining~\eqref{KSN-o204-9350-01093-1} and~\eqref{KSN-o204-9350-01093-2},
we find that, for all~$\vartheta$, $\tau\in\Omega$ with~$|\vartheta
-{x}_k|\le r$ and~$|\tau
-{y}_k|\le r$,
$$ |e(\vartheta)-e(\tau)|\ge|e( x_k)-e( y_k)|-|e(\vartheta)-e(x_k)|-
|e(\tau)-e(y_k)|
\ge \frac{1-a}2-\frac{1-a}{5}\ge\frac{1-a}{4}.$$
For this reason, letting~${\mathcal{S}}_1:=B_r({x}_k)\cap\Omega$
and~${\mathcal{S}}_2:=B_r({y}_k)\cap\Omega$, we have that
\begin{equation}\label{KSMD8399f8848294}
\begin{split}&
\iint_\Q \frac{|e(x)-e(y)|^2}{|x-y|^{n+2s}}\,dx\,dy\ge
\iint_{\Omega\times\Omega} \frac{|e(x)-e(y)|^2}{|x-y|^{n+2s}}\,dx\,dy\ge
\iint_{{\mathcal{S}}_1\times {\mathcal{S}}_2} \frac{|e(\vartheta)-e(\tau)|^2}{|\vartheta-\tau|^{n+2s}}\,d\vartheta\,d\tau\\&\qquad\quad
\ge \iint_{{\mathcal{S}}_1\times {\mathcal{S}}_2} \frac{(1-a)^2}{
16\,d_1^{n+2s}}\,d\vartheta\,d\tau=
\frac{(1-a)^2\;|{\mathcal{S}}_1|\;|{\mathcal{S}}_2|}{
16\,d_1^{n+2s}}
\ge\frac{(1-a)^2\;r^2}{
64\,d_1^{n+2s}}.
\end{split}
\end{equation}
We also recall that in case~\eqref{A:CA:02} we have that~$\Omega$ is not connected
and consequently~$\beta\ne0$,
due to~\eqref{KASM:LSKDD}. {F}rom this and~\eqref{KSMD8399f8848294},
up to renaming constants, we deduce that
\begin{eqnarray*}
\int_\Omega |e'|^2\,dx+
\beta\iint_\Q \frac{|e(x)-e(y)|^2}{|x-y|^{n+2s}}\,dx\,dy\ge
\frac{(1-a)^2\;r^2}{C},
\end{eqnarray*}
which, together with~\eqref{deferrre}, proves~\eqref{n1numeratore}, as desired.

We also remark that, testing the weak formulation of~\eqref{probauto}
against a constant function, one sees that
$$\int_\Omega m e\,dx=0,$$ and therefore, recalling~\eqref{SUSU},
\begin{equation}\label{n1disug>}
\overline{m}|D|\geq \overline{m}\int_D e\,dx
=\underline{m}\int_{\Omega\setminus D} e\,dx \geq a\,\underline{m}\,|\Omega\setminus D|.
\end{equation}
Recalling~\eqref{misuraD}, we see that
\begin{equation}\label{-554}
|\Omega\setminus D|=\frac{\overline{m}-m_0}{
\underline{m}+\overline{m}}|\Omega|,
\end{equation}
and therefore
\begin{equation}\label{Posityc3456}\begin{split}
\overline{m}&\int_D e^2\,dx - \underline{m}\int_{\Omega\setminus D}e^2\,dx
\leq \overline{m}|D|-a^2\underline{m}|\Omega\setminus D| \\
&=\overline{m}\,\frac{\underline{m}+m_0}{\underline{m}+\overline{m}}\,|\Omega|
-a^2\underline{m}\frac{\overline{m}-m_0}{\underline{m}+\overline{m}}\,|\Omega|
=\frac{\overline{m}\underline{m}(1-a^2)
+m_0(\overline{m}+a^2\underline{m})}{\underline{m}+\overline{m}}\,
|\Omega| \\
&<\frac{\overline{m}\underline{m}}{\underline{m}+\overline{m}}(1-a^2)\,
|\Omega|,
\end{split}\end{equation}
since $m_0<0$ and~$\overline{m}$, $\underline{m}>0$.
Using this inequality and~\eqref{n1numeratore} in
\eqref{n1lambda1}, we conclude that
\begin{equation}\label{Perfhrsa}
C\,\lambda_1(m)\geq \frac{\underline{m}+\overline{m}}{\overline{m}
\underline{m}}\cdot
\frac{(1-a)^4}{1-a^2}
=\frac{\underline{m}+\overline{m}}{\overline{m}\underline{m}}
\cdot\frac{(1-a)^3}{1+a},
\end{equation}
up to renaming~$C>0$.

Furthermore, from \eqref{n1disug>} we know that
\begin{equation*}
a\leq \frac{\overline{m}}{\underline{m}}\cdot
\frac{\underline{m}+m_0}{\overline{m}-m_0}
.\end{equation*}
Consequently, since the map~$[0,1]\ni t:=\frac{(1-t)^3}{1+t}$
is decreasing, we discover that
$$ \frac{(1-a)^3}{1+a}\ge
\frac{\left(1-\displaystyle\frac{\overline{m}}{\underline{m}}\cdot
\frac{\underline{m}+m_0}{\overline{m}-m_0}\right)^3}{1+\displaystyle\frac{\overline{m}}{\underline{m}}\cdot
\frac{\underline{m}+m_0}{\overline{m}-m_0}}
=-\frac{m_0^3 (\underline{m}+\overline{m})^3}{\underline{m}^2
(\overline{m}-m_0)^2}\cdot\frac1{2\overline{m}\underline{m}+m_0
(\overline{m}-\underline{m})}
.$$
Combining this information and~\eqref{Perfhrsa},
we deduce that
\begin{equation*}
\begin{split}C\,
\lambda_1(m)\,&\geq -\frac{\underline{m}+\overline{m}}
{\overline{m}\underline{m}}\cdot
\frac{m_0^3 (\underline{m}+\overline{m})^3}{\underline{m}^2
(\overline{m}-m_0)^2}\cdot\frac1{2\overline{m}\underline{m}+m_0
(\overline{m}-\underline{m})}\\&=
-\,
\frac{m_0^3 (\underline{m}+\overline{m})^4}{\overline{m}\underline{m}^3
(\overline{m}-m_0)^2\big(2\overline{m}\underline{m}+m_0
(\overline{m}-\underline{m})\big)}
.\end{split}
\end{equation*}
Taking the infimum of this expression, we find the desired result.
\end{proof}

With this, we are in the position to give the proof of Theorem \ref{alpha>0}.

\begin{proof}[Proof of Theorem \ref{alpha>0}]
For any $\underline{m}>0$ and any~$m_0\in(-\underline{m},0)$, we define the function
$g_{\underline{m},m_0}:(0,+\infty) \to (0,+\infty)$ as
\[
g_{\underline{m},m_0}(\emme):=-\,
\frac{m_0^3 (\underline{m}+\emme)^4}{\emme \underline{m}^3
(\emme -m_0)^2\big(\emme(2\underline{m}+m_0)-
\underline{m}m_0\big)}.
\]
We observe that
\[
\lim_{\emme\nearrow +\infty}g_{\underline{m},m_0}(\emme)=
-\,
\frac{m_0^3}{\underline{m}^3(2\underline{m}+m_0)}>0,
\]
and
\begin{equation}\label{a2t4t4y5}
\lim_{\emme\searrow 0}g_{\underline{m},m_0}(\emme)=+\infty.
\end{equation}
In particular,
\begin{equation}\label{chiamala2}
\inf_{\emme\in(0,+\infty)} g_{\underline{m},m_0}(\emme)>0.\end{equation}

Now, by Lemma \ref{n1lemma}, we know that
\begin{equation}\label{chiamala}
\underline{\lambda}(\emme,\underline{m},m_0)\geq
C\,g_{\underline{m},m_0}(\emme).
\end{equation}
As a result, \eqref{viene1} follows from~\eqref{chiamala2}
and~\eqref{chiamala}. Moreover, from~\eqref{a2t4t4y5} and~\eqref{chiamala}
we obtain~\eqref{viene2}.

To prove~\eqref{viene3}, for any~$\overline{m}>0$ and~$m_0<0$,
we define the
function~$\tilde g_{\overline{m},m_0}:(-m_0,+\infty)\to (0,+\infty)$ as
\[
\tilde g_{\overline{m},m_0}(\emme):=
-\,
\frac{m_0^3 (\emme+\overline{m})^4}{\overline{m}\emme^3
(\overline{m}-m_0)^2\big(
\overline{m}(
2\emme+m_0)-
\emme m_0\big)}.
\]
We notice that
\begin{equation*}
\lim_{\emme\searrow -m_0}\tilde g_{\overline{m},m_0}(\emme)=-\,
\frac{\overline{m}-m_0}{\overline{m}m_0}>0
\end{equation*}
and that
\begin{equation*}
\lim_{\emme\nearrow +\infty}\tilde g_{\overline{m},m_0}(\emme)=
-\,
\frac{m_0^3}{\overline{m}
(\overline{m}-m_0)^2(
2\overline{m}- m_0)}>0.
\end{equation*}
Accordingly,
$$ \inf_{\emme\in(-m_0,+\infty)}
\tilde g_{\overline{m},m_0}(\emme)>0.$$
Using this and the fact that,
by Lemma~\ref{n1lemma}, 
\begin{equation*}
\underline{\lambda}(\overline{m},\emme,m_0)\geq
C\,\tilde g_{\overline{m},m_0}(\emme),
\end{equation*}
we obtain the desired result in~\eqref{viene3}.
\end{proof}

Having established Theorem \ref{alpha>0}, we now deal with the case
in which~$\alpha=0$, namely when only the nonlocal dispersal
is active. This case is considered in Theorem~\ref{alpha=0},
according to two different ranges of the fractional parameter~$s$.
For this, we divide the proof of Theorem \ref{alpha=0} into two parts.

\begin{proof}[Proof of Theorem \ref{alpha=0} when $s\in(1/2,1)$]
We denote by $e$ the eigenfunction associated to the
first eigengenvalue of problem \eqref{probauto}, normalized such that
\begin{equation}\label{9siupinf045} a:=\inf_\Omega e\ge0\qquad{\mbox{and}}\qquad
\sup_\Omega e=1.\end{equation}
We recall~\eqref{n1disug>}, \eqref{-554}
and~\eqref{Posityc3456} to write that
\begin{equation}\label{alew013}
a\leq \frac{\overline{m}}{\underline{m}}\cdot
\frac{\underline{m}+m_0}{\overline{m}-m_0}
\end{equation}
and
\begin{equation}\label{alpha0denom}
\overline{m}\int_D e^2\,dx - \underline{m}\int_{\Omega\setminus D}e^2\,dx
\leq\frac{\overline{m}\underline{m}}{\underline{m}+\overline{m}}(1-a^2)|\Omega|.
\end{equation}
We stress that, in view of~\eqref{alew013},
\begin{equation}\label{nonin9203ioejwoeifjg8y87} \delta_0:=1-a\ge1-
\frac{\overline{m}}{\underline{m}}\cdot
\frac{\underline{m}+m_0}{\overline{m}-m_0}=
-\frac{m_0(\underline{m}+\overline{m})}{
\underline{m}(\overline{m}-m_0)}>0.
\end{equation}
In particular, by~\eqref{9siupinf045}, we can find~$\bar x$ and~$\bar y$
in~$\Omega$ such that
\begin{equation}\label{RiC7jforngng0945}
e(\bar x)\le a+\frac{\delta_0}{100}\qquad{\mbox{and}}\qquad
e(\bar y)\ge 1-\frac{\delta_0}{100}
.\end{equation}
Now we claim that
\begin{equation}\label{att163}
\iint_\Q\frac{(e(x)-e(y))^2}{|x-y|^{1+2s}}\,dx\,dy\ge c\,
(1-a)^{\frac{4s+2}{2s-1}},
\end{equation}
for some~$c\in(0,1)$ depending only on~$s$ and~$\Omega$
(in particular, this~$c$ is independent of~$m$).
Indeed, if the left hand side of~\eqref{att163} is larger that~$1$,
we are done, therefore we can suppose, without loss of generality, that
$$ \iint_\Q\frac{(e(x)-e(y))^2}{|x-y|^{1+2s}}\,dx\,dy\le1.$$
As a result,
$$ \iint_{\Omega\times\Omega}\frac{(e(x)-e(y))^2}{|x-y|^{1+2s}}\,dx\,dy+
\|e\|_{L^2(\Omega)}^2\le 1+|\Omega|,$$
and consequently we can exploit Theorem~8.2
in~\cite{MR2944369} and conclude that~$\|e\|_{C^{\frac{2s-1}{2}}
(\Omega)}\le C_0$,
for some~$C_0>0$ depending only on~$s$ and~$\Omega$.

We let~$d_1$ be the diameter of~$\Omega$
and~$d_2$ be the smallest diameter of all the connected
components of~$\Omega$. We also define
\begin{equation}\label{rodefnusuperia95} r_0:=\min\left\{
\left( \frac{\delta_0}{100\,C_0}\right)^{\frac2{2s-1}},\;\frac{d_2}{100}
\right\},\end{equation}
and we observe that, for each~$x\in \Omega\cap B_{r_0}(\bar x)$,
$$ e(x)\le e(\bar x)+|e(x)-e(\bar x)|\le
e(\bar x)+C_0 |x-\bar x|^{\frac{2s-1}{2}}\le
a+\frac{\delta_0}{100}+C_0\,r_0^{\frac{2s-1}{2}}\le
a+\frac{\delta_0}{50},$$
thanks to~\eqref{RiC7jforngng0945}.

Similarly, for each~$y\in \Omega\cap B_{r_0}(\bar y)$,
$$ e(y)\ge 1-\frac{\delta_0}{50},$$
and consequently
\begin{eqnarray*}&& \iint_{(\Omega\cap B_{r_0}(\bar x))\times(
\Omega\cap B_{r_0}(\bar y))}\frac{(e(x)-e(y))^2}{|x-y|^{1+2s}}\,dx\,dy
\\&&\qquad\ge
\left(1-a-\frac{\delta_0}{25}\right)^2
\iint_{(\Omega\cap B_{r_0}(\bar x))\times(
\Omega\cap B_{r_0}(\bar y))}\frac{dx\,dy}{|x-y|^{1+2s}}\\&&\qquad\geq
\left(1-a-\frac{\delta_0}{25}\right)^2
\frac{|\Omega\cap B_{r_0}(\bar x)|\;|
\Omega\cap B_{r_0}(\bar y)|}{d_1^{1+2s}}\\&&\qquad\ge
\frac{(1-a)^2}{ 4\,d_1^{2+2s}}\;|\Omega\cap B_{r_0}(\bar x)|\;|
\Omega\cap B_{r_0}(\bar y)|\\&&\qquad\ge
\frac{(1-a)^2\,r_0^2}{ 4\,d_1^{2+2s}}.
\end{eqnarray*}
{F}rom this and~\eqref{rodefnusuperia95}, noticing that~$\frac{4s+2}{2s-1}>2$,
we obtain~\eqref{att163}, as desired.

Gathering~\eqref{att163}
and~\eqref{alpha0denom} we find that
\begin{equation}\label{4546jhfkerjr3}
\lambda_1(m)=\frac{\displaystyle\frac\beta4
\iint_\Q \frac{|e(x)-e(y)|^2}{|x-y|^{n+2s}}\,dx\,dy
}{\displaystyle\overline{m}\int_D e^2\,dx - \underline{m}\int_{\Omega\setminus D}e^2\,dx}
\geq \frac{\underline{m}+\overline{m}}{\overline{m}
\underline{m}}\cdot
\frac{C\,(1-a)^{\frac{4s+2}{2s-1}}}{1-a^2},
\end{equation}
for some~$C>0$.

We also notice that, in view of~\eqref{nonin9203ioejwoeifjg8y87},
\begin{eqnarray*} \frac{(1-a)^{\frac{4s+2}{2s-1}}}{1-a^2}=
\frac{(1-a)^{\frac{4s+2}{2s-1}}}{(1-a)(1+a)}=
\frac{(1-a)^{\frac{2s+3}{2s-1}}}{1+a}\ge\frac12\,
(1-a)^{\frac{2s+3}{2s-1}}\ge\frac12\,
\left(-\frac{m_0(\underline{m}+\overline{m})}{
\underline{m}(\overline{m}-m_0)}\right)^{\frac{2s+3}{2s-1}}.
\end{eqnarray*}
By inserting this inequality into~\eqref{4546jhfkerjr3},
we conclude that
\begin{equation}\label{4546jhfkerjr3-lANsd}
\lambda_1(m)\geq \frac{C\,(\underline{m}+\overline{m})}{\overline{m}
\underline{m}}\cdot
\left(-\frac{m_0(\underline{m}+\overline{m})}{
\underline{m}(\overline{m}-m_0)}\right)^{\frac{2s+3}{2s-1}},
\end{equation}
up to renaming~$C$.

Now, for any~$m_0<0$ and any~$\underline{m}>-m_0$, we define
the function~$\bar{g}_{\underline{m},m_0}:(0,+\infty)\to (0,+\infty)$
given by
\[
\bar{g}_{\underline{m},m_0}(\emme):=
\frac{\underline{m}+\emme}{\emme
\underline{m}}\cdot
\left(-\frac{m_0(\underline{m}+\emme)}{
\underline{m}(\emme-m_0)}\right)^{\frac{2s+3}{2s-1}}
.
\]
We remark that
\[
\lim_{\emme\searrow0}\bar{g}_{\underline{m},m_0}(\emme)=+\infty
\qquad{\mbox{and}}\qquad
\lim_{\emme\nearrow+\infty}\bar{g}_{\underline{m},m_0}(\emme)=
\frac{1}{\underline{m}}\cdot
\left(-\frac{m_0}{
\underline{m}}\right)^{\frac{2s+3}{2s-1}}>0,
\]
and consequently
\begin{equation}\label{Pstreposinf}
\inf_{\emme\in(0,+\infty)} \bar{g}_{\underline{m},m_0}(\emme)>0.
\end{equation}
Also, by~\eqref{4546jhfkerjr3-lANsd},
and making use of~\eqref{tildem} and Proposition~\ref{prop:bang},
we find that
\[
\underline{\lambda}(\emme,\underline{m},m_0)=
\inf_{m\in\tilde{\mathscr{M}}(\emme,\underline{m},m_0)} \lambda_1(m)\ge C\,
\bar{g}_{\underline{m},m_0}(\emme)
.\]
In particular,
\[
\underline{\lambda}(\emme,\underline{m},m_0)\ge C\,
\inf_{\emme\in(0,+\infty)} \bar{g}_{\underline{m},m_0}(\emme)
,\]
which, combined with~\eqref{Pstreposinf}, proves \eqref{viene4}.

Similarly, we see that
\[\lim_{\emme\searrow0}
\underline{\lambda}(\emme,\underline{m},m_0)\ge C\,\lim_{\emme\searrow0}
\bar{g}_{\underline{m},m_0}(\emme)=+\infty,\]
which establishes~\eqref{viene4bis}.

In addition, given~$m_0<0$ and~$\overline{m}>0$,
if we consider the auxiliary function~$g_{\overline{m},m_0}^\star:(-m_0,+\infty)\to(0,+\infty)$
defined by
\[
g^\star_{\overline{m},m_0}(\emme):=
\frac{\emme+\overline{m}}{\overline{m}
\emme}\cdot
\left(-\frac{m_0(\emme+\overline{m})}{
\emme(\overline{m}-m_0)}\right)^{\frac{2s+3}{2s-1}}
 ,
\]
we see that
\begin{eqnarray*}&&
\lim_{\emme\searrow-m_0}
g^\star_{\overline{m},m_0}(\emme)=
-\frac{-m_0+\overline{m}}{\overline{m}m_0}>0\\
{\mbox{and}}\qquad&&
\lim_{\emme\nearrow+\infty}
g^\star_{\overline{m},m_0}(\emme)=
\frac{1}{\overline{m}}\cdot
\left(-\frac{m_0}{\overline{m}}\right)^{\frac{2s+3}{2s-1}}>0,
\end{eqnarray*}
and these observations allow us to conclude that
\begin{equation}\label{scrinonchoi23t49rjk}
\inf_{\emme\in(-m_0,+\infty)}
g^\star_{\overline{m},m_0}(\emme)>0.
\end{equation}
Moreover,
we deduce from Proposition~\ref{prop:bang},
\eqref{tildem} and~\eqref{4546jhfkerjr3-lANsd} that
\[
\underline{\lambda}(\overline{m},\emme,m_0)=
\inf_{m\in\tilde{\mathscr{M}}(\overline{m},\emme,m_0)} \lambda_1(m)
\ge C\,
g^\star_{\overline{m},m_0}(\emme)
.\]
Therefore
$$ \underline{\lambda}(\overline{m},\emme,m_0)\ge
C\,\inf_{\emme\in(-m_0,+\infty)}
g^\star_{\overline{m},m_0}(\emme),$$
which, together with~\eqref{scrinonchoi23t49rjk},
proves~\eqref{viene5}.
\end{proof}

Now we prove Theorem~\ref{alpha=0} in the case~$s\in (0,1/2]$.
This case is somehow conceptually
related to the case~$n\ge2$, since the problem
boils down to a subcritical situation.

We suppose without loss of generality that
\begin{equation}\label{B22}
(-2,2)=B_2\subset\Omega,\end{equation}
and we define the function
\begin{equation}\label{VARPHI44}
\varphi(x):=
\begin{dcases}
c_\star+1  & {\mbox{ if }}x\in B_\rho,
\\ \displaystyle
c_\star+\frac{\log|x|}{\log\rho}  & {\mbox{ if }}x\in B_1\setminus B_\rho,
\\
c_\star  &{\mbox{ if }}x\in \R\setminus B_1.
\end{dcases}
\end{equation}
Here, $c_\star>0$ is the constant introduced in~\eqref{VECC7},
and we set
\begin{equation}\label{VARPHI45}
D:=B_\rho.\end{equation}

For our purposes, we recall the following basic inequality:

\begin{lem}
For every~$x$, $y\in\R^n\setminus \{0\}$, we have that
\begin{equation}\label{disuglog}
\big|\log|x|-\log|y|\big|\leq \frac{\big|\,|x|-|y|\,\big|}{\min\{|x|,|y|\}}.
\end{equation}
\end{lem}

\begin{proof}
Without loss of generality, we assume that~$|y|\leq |x|$.
To check~\eqref{disuglog}, we take~$t:=\frac{|x|}{|y|}-1$, and
we see that
\begin{equation*}
\big|\log|x|-\log|y|\big|=\log|x|-\log|y|=\log\frac{|x|}{|y|}=\log(1+t)\le t=
\frac{|x|}{|y|}-1=\frac{|x|-|y|}{|y|},
\end{equation*}
as desired.
\end{proof}

With this, we now list some properties of the auxiliary function~$\varphi$
in~\eqref{VARPHI44}.

\begin{lem}\label{gradlogto0}
Let $n=1$, $s\in (0,1/2]$ and~$\varphi$ be as in~\eqref{VARPHI44}. Then,
\[ \lim_{\rho\searrow0}
\iint_\Q \frac{|\varphi(x)-\varphi(y)|^2}{|x-y|^{1+2s}}\,dx\,dy =0.
\]
\end{lem}

\begin{proof}
Without loss of generality, we suppose that~$\rho<1/4$.
We observe that
\begin{equation}\label{EESEC1}
\iint_{B_\rho \times B_\rho}
\frac{|\varphi(x)-\varphi(y)|^2}{|x-y|^{1+2s}}\,dx\,dy=0
\end{equation}
and 
\begin{equation}\label{SEC2}
\iint_{(\R^n\setminus B_1) \times(\R^n\setminus B_1)}
\frac{|\varphi(x)-\varphi(y)|^2}{|x-y|^{1+2s}}\,dx\,dy=0.
\end{equation}

Moreover, 
\begin{equation}\label{SEC3}\begin{split}
&\lim_{\rho\searrow0}
\iint_{B_\rho \times (\R^n\setminus B_1)}
\frac{|\varphi(x)-\varphi(y)|^2}{|x-y|^{1+2s}}\,dx\,dy=\lim_{\rho\searrow0}
\iint_{B_\rho \times (\R^n\setminus B_1)}
\frac{dx\,dy}{|x-y|^{1+2s}} \\&\qquad
\leq \lim_{\rho\searrow0}
\int_{B_\rho} dx \int_{\R^n\setminus B_{\frac{1}{2}}}
\frac{1}{|z|^{1+2s}}\,dz
\leq \lim_{\rho\searrow0} C \rho=0.
\end{split}\end{equation}

Now we observe that if~$x\in B_1 \setminus B_\rho$
and~$y\in B_\rho$, then
\[
|\varphi(x)-\varphi(y)|^2=\frac{1}{(\log\rho)^2}|\log|x|-\log\rho|^2.
\]
As a consequence,
changing variables $X:=x/\rho$ and $Y:=y/\rho$, and
taking~$k\in\N$ such that~$2^{k-1}\leq 1/\rho \leq 2^k$,
we see that
\begin{equation}\begin{split}\label{drifgetacvZN000}
\iint_{(B_1\setminus B_{2\rho})\times B_\rho}
\frac{|\varphi(x)-\varphi(y)|^2}{|x-y|^{1+2s}}\,dx\,dy
=\;&\frac{1}{(\log\rho)^2}\iint_{(B_1\setminus B_{2\rho})\times B_\rho}
\frac{|\log |x|-\log\rho|^2}{|x-y|^{1+2s}}\,dx\,dy\\
= \;&\frac{\rho^{1-2s}}{(\log\rho)^2}
\iint_{(B_{1/\rho}\setminus B_{2})\times B_1}
\frac{|\log |X||^2}{ |X-Y|^{1+2s}}\,dX\,dY  \\ \le\;& 
\frac{\rho^{1-2s}}{(\log\rho)^2} \sum_{j=2}^k
\iint_{(B_{2^j}\setminus B_{2^{j-1}})\times B_1}
\frac{|\log (2^j)|^2}{ (2^{j-1}-1)^{1+2s}}\,dX\,dY \\
\leq\;& \frac{C\,\rho^{1-2s}}{(\log\rho)^2}
\sum_{j=2}^k \frac{2^j j^2}{2^{j(1+2s)}}\\ \le\;&
\frac{C\,\rho^{1-2s}}{(\log\rho)^2}
\sum_{j=1}^k\frac{j^2}{2^{2sj}}\\ \le\;&
\frac{C\,\rho^{1-2s}}{(\log\rho)^2},
\end{split}\end{equation}
up to renaming~$C>0$.

In addition, using \eqref{disuglog} (with~$|y|:=\rho$)
and changing variable~$z:=x-y$,
\begin{eqnarray*}
\iint_{ (B_{2\rho}\setminus B_\rho)\times B_\rho}
\frac{|\varphi(x)-\varphi(y)|^2}{|x-y|^{1+2s}}\,dx\,dy
&=&\frac{1}{(\log\rho)^2}
\iint_{(B_{2\rho}\setminus B_\rho)\times B_\rho}
\frac{|\log |x|-\log\rho|^2}{|x-y|^{1+2s}}\,dx\,dy\\
&\le& \frac{1}{(\log\rho)^2}
\iint_{(B_{2\rho}\setminus B_\rho)\times B_\rho}
\frac{(|x|-\rho)^2}{\rho^2\,|x-y|^{1+2s}}\,dx\,dy
\\
&\le& \frac{1}{(\log\rho)^2}
\iint_{(B_{2\rho}\setminus B_\rho)\times B_\rho}
\frac{|x-y|^{1-2s}}{\rho^2}\,dx\,dy
\\&\le& \frac{1}{(\log\rho)^2}\iint_{B_{3\rho}\times B_\rho}
\frac{|z|^{1-2s}}{\rho^2}\,dz\,dy
\\
&=&\frac{C\,\rho^{1-2s}}{(\log\rho)^2},
\end{eqnarray*} 
for some~$C>0$.

{F}rom this and~\eqref{drifgetacvZN000}, we deduce that
\begin{equation}\label{SEC4} \lim_{\rho\searrow0}
\iint_{(B_1\setminus B_{\rho})\times B_\rho}
\frac{|\varphi(x)-\varphi(y)|^2}{|x-y|^{1+2s}}\,dx\,dy=0.\end{equation}

Moreover, recalling~\eqref{B22},
\begin{align*}
&\iint_{(B_1 \setminus B_\rho)\times (\R^n\setminus \Omega)}
\frac{|\varphi(x)-\varphi(y)|^2}{|x-y|^{1+2s}}\,dx\,dy=\frac{1}{
(\log\rho)^2}
\iint_{(B_1 \setminus B_\rho)\times (\R^n\setminus \Omega)}
\frac{\big|\log |x|\big|^2}{|x-y|^{1+2s}}\,dx\,dy
\\&\qquad \leq \frac{1}{(\log\rho)^2}
\int_{B_1 \setminus B_\rho}\big|\log |x|\big|^2\,dx
\int_{\R^n\setminus B_1}\frac{dz}{|z|^{1+2s}} \le
\frac{C}{(\log\rho)^2},
\end{align*}
for some~$C>0$. This implies
\begin{equation}\label{SEC5} \lim_{\rho\searrow0}
\iint_{(B_1 \setminus B_\rho)\times (\R^n\setminus \Omega)}
\frac{|\varphi(x)-\varphi(y)|^2}{|x-y|^{1+2s}}\,dx\,dy=0.\end{equation}

Now, we observe that
\begin{equation}\begin{split}\label{0r450yhjhrondfk}
&\iint_{(B_{1/2}\setminus B_\rho)\times (\Omega\setminus B_1)}
\frac{|\varphi(x)-\varphi(y)|^2}{|x-y|^{1+2s}}\,dx\,dy=
\frac{1}{(\log\rho)^2}
\iint_{(B_{1/2}\setminus B_\rho)\times (\Omega\setminus B_1)}
\frac{\big|\log |x|\big|^2}{|x-y|^{1+2s}}\,dx\,dy \\&\qquad
\leq \frac{1}{(\log\rho)^2}
\int_{B_{1/2}\setminus B_\rho}\big|\log |x|\big|^2\,dx
\int_{\Omega\setminus B_1}\frac{1}{|z|^{1+2s}}\,dz \le
\frac{C}{(\log\rho)^2},
\end{split}\end{equation}
for a suitable~$C>0$.

Furthermore, taking $R>0$ such that $\Omega\subset B_R$,
\begin{align*}&
\iint_{(B_1 \setminus B_{1/2})\times (\Omega\setminus B_1)}
\frac{|\varphi(x)-\varphi(y)|^2}{|x-y|^{1+2s}}\,dx\,dy
=
\frac{1}{(\log\rho)^2}
\iint_{(B_1 \setminus B_{1/2})\times (\Omega\setminus B_1)}
\frac{\big|\log |x|\big|^2}{|x-y|^{1+2s}}\,dx\,dy \\&\qquad
\leq \frac{4(\log2)^2}{(\log\rho)^2}
\int_\frac{1}{2}^1 \int_1^R (y-x)^{-1-2s}\,dx\,dy 
=\frac{2(\log2)^2}{s(\log\rho)^2}
\int_\frac{1}{2}^1 [(1-x)^{-2s}-(R-x)^{-2s}]\,dx \\
&\qquad=\frac{2(\log2)^2}{s(1-2s)(\log\rho)^2}
\left[(R-1)^{1-2s}+\left(\frac{1}{2}\right)^{1-2s}
-\left(R-\frac{1}{2}\right)^{1-2s}\right].
\end{align*}
This and~\eqref{0r450yhjhrondfk} give that
\begin{equation}\label{SEC6} \lim_{\rho\searrow0}
\iint_{(B_1 \setminus B_\rho)\times ( \Omega\setminus B_1)}
\frac{|\varphi(x)-\varphi(y)|^2}{|x-y|^{1+2s}}\,dx\,dy=0.\end{equation}

Now, we take~$k\in\N$ such that
\begin{equation}\label{LOGrh}
\frac1{2^{k+1}}<\rho\le\frac1{2^k},
\end{equation}
and we observe that
\begin{equation}\label{EQ:3}
\begin{split}&
\iint_{(B_1\setminus B_\rho)\times(B_1\setminus B_\rho)}
\frac{|\varphi(x)-\varphi(y)|^2}{|x-y|^{1+2s}}\,dx\,dy\\
=\;&\frac{1}{(\log\rho)^2}\iint_{(B_1\setminus B_\rho)\times(B_1\setminus B_\rho)}
\frac{\big|\log|x|-\log|y|\big|^2}{|x-y|^{1+2s}}\,dx\,dy\\
=\;&\frac{2}{(\log\rho)^2}
\iint_{{(B_1\setminus B_\rho)\times(B_1\setminus B_\rho)}\atop{\{|x|\ge|y|\}}}
\frac{\big|\log|x|-\log|y|\big|^2}{|x-y|^{1+2s}}\,dx\,dy\\
\le\;&\frac{2}{(\log\rho)^2}
\iint_{ {\left(B_1\setminus B_{1/2^{k+1}}\right)\times\left(
B_1\setminus B_{1/2^{k+1}}\right)}\atop{\{|x|\ge|y|\}} }
\frac{\big|\log|x|-\log|y|\big|^2}{|x-y|^{1+2s}}\,dx\,dy\\
=\;&\frac{2}{(\log\rho)^2}
\sum_{i,j=0}^{k}
\iint_{
{\left(B_{1/2^i}\setminus B_{1/2^{i+1}}\right)
\times\left(B_{1/2^j}\setminus B_{1/2^{j+1}}\right) }\atop{\{|x|\ge|y|\}} 
}
\frac{\big|\log|x|-\log|y|\big|^2}{|x-y|^{1+2s}}\,dx\,dy.
\end{split}
\end{equation}
Moreover, we remark that if~$x\in B_{1/2^i}$, $y\not\in B_{1/2^{j+1}}$
and~$|x|\ge|y|$,
we have that
$$ \frac1{2^i}\ge |x|\ge|y|\ge \frac1{2^{j+1}},$$
and accordingly~$i\le j+1$.

This observation and~\eqref{EQ:3} yield that
\begin{equation}\label{EQ:4}
\begin{split}&
\frac{(\log\rho)^2}2\iint_{(B_1\setminus B_\rho)\times(B_1\setminus B_\rho)}
\frac{|\varphi(x)-\varphi(y)|^2}{|x-y|^{1+2s}}\,dx\,dy\\
\le\;&
\sum_{j=0}^{k}\sum_{i=0}^{j+1}
\iint_{{\left(B_{1/2^i}\setminus B_{1/2^{i+1}}\right)
\times\left(B_{1/2^j}\setminus B_{1/2^{j+1}}\right)}\atop{\{|x|\ge|y|\}}}
\frac{\big|\log|x|-\log|y|\big|^2}{|x-y|^{1+2s}}\,dx\,dy\\
\le\;&
I_1+I_2,
\end{split}
\end{equation}
where
\begin{eqnarray*}
I_1&:=&\sum_{j=0}^{k}\sum_{j-4\le i\le j+1}
\iint_{{\left(B_{1/2^i}\setminus B_{1/2^{i+1}}\right)
\times\left(B_{1/2^j}\setminus B_{1/2^{j+1}}\right)}
\atop{\{|x|\ge|y|\}}}
\frac{\big|\log|x|-\log|y|\big|^2}{|x-y|^{1+2s}}\,dx\,dy\\
{\mbox{and }}\qquad I_2&:=&
\sum_{j=0}^{k}\sum_{0\le i\le j-4}
\iint_{{\left(B_{1/2^i}\setminus B_{1/2^{i+1}}\right)
\times\left(B_{1/2^j}\setminus B_{1/2^{j+1}}\right)}
\atop{\{|x|\ge|y|\}}}
\frac{\big|\log|x|-\log|y|\big|^2}{|x-y|^{1+2s}}\,dx\,dy.
\end{eqnarray*}
We point out that, if~$x\in B_{1/2^i}$ and~$|y|\le|x|$,
then
$$ |x-y|\le|x|+|y|\le 2|x|\le\frac{1}{2^{i-1}}.$$
In light of this fact and~\eqref{disuglog}, we have that
\begin{equation}\label{23USJ}
\begin{split}
I_1\,&\le
\sum_{j=0}^{k}\sum_{j-4\le i\le j+1}
\iint_{{{\left(B_{1/2^i}\setminus B_{1/2^{i+1}}\right)
\times\left(B_{1/2^j}\setminus B_{1/2^{j+1}}\right)}}\atop{\{|x|\ge|y|\}}}
\frac{|x-y|^{1-2s}}{|y|^{2}}\,dx\,dy\\&\le
\sum_{j=0}^{k}\sum_{j-4\le i\le j+1}
\iint_{{{\left(B_{1/2^i}\setminus B_{1/2^{i+1}}\right)
\times\left(B_{1/2^j}\setminus B_{1/2^{j+1}}\right)}}}
\frac{2^{(1-i)(1-2s)}}{2^{-2(j+1)}}\,dx\,dy\\&=
\sum_{j=0}^{k}\sum_{j-4\le i\le j+1}
\frac{2^{-i}\,2^{-j}\,2^{(1-i)(1-2s)}}{2^{-2(j+1)}}
\\&=2^{3-2s}\sum_{j=0}^{k}\sum_{j-4\le i\le j+1}2^{-2i(1-s)}\,2^{j}
\\&\le 3\cdot 2^{4-2s}\sum_{j=0}^{k}2^{-2(j-4)(1-s)}\,2^{j}
\\&= 3\cdot 2^{12-10s}\sum_{j=0}^{k}2^{(2s-1)j}.
\end{split}\end{equation}
In addition, if~$x\in
B_{1/2^i}$ and~$y\not\in B_{1/2^{j+1}}$
and~$|x|\ge|y|$,
\begin{eqnarray*}
\big|\log|x|-\log|y|\big|=
\log|x|-\log|y|\le \log\frac1{2^i}-\log\frac1{2^{j+1}}=
\big( j-i+1\big)\,\log 2.
\end{eqnarray*}
As a result,
\begin{equation}\label{E3214}
\begin{split}
I_2\,&\le
\log^2 2
\sum_{j=0}^{k}\sum_{0\le i\le j-4}
\iint_{{\left(
B_{1/2^i}\setminus B_{1/2^{i+1}}\right)\times\left(
B_{1/2^j}\setminus B_{1/2^{j+1}}\right)}}
\frac{(j-i+1)^2}{|x-y|^{1+2s}}\,dx\,dy.
\end{split}\end{equation}
Furthermore, if~$x\in B_{1/2^i}\setminus B_{1/2^{i+1}}$ and~$y\in B_{1/2^j}$, with~$i\le j-4$,
we see that
$$ |x-y|\ge|x|-|y|\ge \frac{1}{2^{i+1}}-\frac1{2^j}=
\frac{1}{2^{i+1}}\left(1-\frac1{2^{j-i-1}}\right)\ge\frac{1}{2^{i+2}}\ge\frac{|x|}4.
$$
Then, we insert this information into~\eqref{E3214} and we conclude that
\begin{eqnarray*}
I_2&\le& 4^{1+2s}\,
\log^2 2
\sum_{j=0}^{k}\sum_{0\le i\le j-4}
\iint_{{\left(B_{1/2^i}\setminus B_{1/2^{i+1}}\right)
\times\left(B_{1/2^j}\setminus B_{1/2^{j+1}}\right)}}
\frac{(j-i+1)^2}{|x|^{1+2s}}\,dx\,dy\\
&\le& 4^{1+2s}\,
\log^2 2
\sum_{j=0}^{k}\sum_{0\le i\le j-4}
\iint_{{\left(B_{1/2^i}\setminus B_{1/2^{i+1}}\right)
\times\left(B_{1/2^j}\setminus B_{1/2^{j+1}}\right)}}
\frac{(j-i+1)^2}{2^{-(i+1)(1+2s)}}\,dx\,dy
\\&=&
4^{1+2s}\,
\log^2 2
\sum_{j=0}^{k}\sum_{0\le i\le j-4}
\frac{2^{-i}\,2^{-j}\,(j-i+1)^2}{2^{-(i+1)(1+2s)}}\\
\\&=&
2^{3(1+2s)}\,
\log^2 2
\sum_{j=0}^{k}\sum_{0\le i\le j-4}
2^{2si}\,2^{-j}\,(j-i+1)^2\\&=&
2^{3(1+2s)+1}\,
\log^2 2
\sum_{j=0}^{k}\sum_{0\le i\le j-4}
2^{(2s-1)i}\,2^{i-j-1}\,(j-i+1)^2.
\end{eqnarray*}
Hence, changing index of summation by posing~$\ell:=j-i+1$,
\begin{eqnarray*}
I_2&\le&
2^{3(1+2s)+1}\,
\log^2 2
\sum_{i=0}^{k}\sum_{\ell\ge5}
2^{(2s-1)i}\,2^{-\ell}\,\ell^2\\&\le&
\bar C\,
\sum_{i=0}^{k}
2^{(2s-1)i},
\end{eqnarray*}
where
$$ \bar{C}:=2^{3(1+2s)+1}\,
\log^2 2\;\sum_{\ell=0}^{+\infty}2^{-\ell}\,\ell^2.$$
We plug this information and~\eqref{23USJ} into~\eqref{EQ:4} and we find
that
\begin{equation}\label{A40}
(\log\rho)^2\iint_{(B_1\setminus B_\rho)\times(B_1\setminus B_\rho)}
\frac{|\varphi(x)-\varphi(y)|^2}{|x-y|^{1+2s}}\,dx\,dy\le C_\star
\,\sum_{m=0}^{k}
2^{(2s-1)m},\end{equation}
where~$C_\star:=2(\bar{C}+3\cdot 2^{12-10s}).$

We observe that
$$ \sum_{m=0}^{k}
2^{(2s-1)m}=k  \qquad{\mbox{ if }}s=\frac12,$$
while
$$ \sum_{m=0}^{k}
2^{(2s-1)m}\le \sum_{m=0}^{+\infty}
2^{(2s-1)m}=:C_\sharp\qquad
{\mbox{ if }}s\in\left(0,\frac12\right),
$$
and consequently, by~\eqref{A40},
$$ (\log\rho)^2\iint_{(B_1\setminus B_\rho)\times(B_1\setminus B_\rho)}
\frac{|\varphi(x)-\varphi(y)|^2}{|x-y|^{1+2s}}\,dx\,dy\le\begin{dcases}
C_\star \,k & {\mbox{ if }}s=\displaystyle\frac12,\\
C_\star \,C_\sharp & {\mbox{ if }}s\in\displaystyle\left(0,\frac12\right).
\end{dcases}$$
{F}rom this and~\eqref{LOGrh}, it follows that
$$ (\log\rho)^2\iint_{(B_1\setminus B_\rho)\times(B_1\setminus B_\rho)}
\frac{|\varphi(x)-\varphi(y)|^2}{|x-y|^{1+2s}}\,dx\,dy\le\begin{dcases}
C_\star \,\displaystyle\frac{|\log\rho|}{\log2} & {\mbox{ if }}s=\displaystyle\frac12,\\
C_\star \,C_\sharp & {\mbox{ if }}s\in\displaystyle\left(0,\frac12\right).
\end{dcases}$$
This implies that
$$ \lim_{\rho\searrow0}\iint_{(B_1\setminus B_\rho)\times(B_1\setminus B_\rho)}
\frac{|\varphi(x)-\varphi(y)|^2}{|x-y|^{1+2s}}\,dx\,dy=0.$$

{F}rom this, \eqref{EESEC1},
\eqref{SEC2}, \eqref{SEC3}, \eqref{SEC4}, \eqref{SEC5} and~\eqref{SEC6}
we obtain the desired result.
\end{proof}

An additional useful property of the function~$\varphi$
defined in~\eqref{VARPHI44} is the following:

\begin{lem}\label{denomlog}
Let $n=1$, $s\in (0,1/2]$ and~$\varphi$ be as in~\eqref{VARPHI44}. Then,
\[ \lim_{\rho\searrow0}
\overline{m}\int_D \varphi^2\,dx-\underline{m}
\int_{\Omega\setminus D}\varphi^2\,dx
=-\frac{\underline{m}(\underline{m}+m_0)}{m_0}
|\Omega|>0 
.\]
\end{lem}

\begin{proof}
{F}rom~\eqref{VARPHI44} and~\eqref{VARPHI45},
and exploiting~\eqref{misuraD} and~\eqref{VECC7},
we see that
\[
\overline{m}\int_D \varphi^2\,dx
=\overline{m}(c_\star +1)^2|D|=\overline{m}(c_\star +1)^2\,
\frac{\underline{m}+m_0}{\underline{m}+
\overline{m}}|\Omega|=
\frac{\overline{m}\underline{m}^2(
\underline{m}+m_0)}{m_0^2(\underline{m}+
\overline{m})}|\Omega|
\]
and that
\begin{eqnarray*}
&&\underline{m}\int_{\Omega \setminus D}\varphi^2\,dx
=\underline{m}\int_{B_1\setminus B_\rho}\varphi^2\,dx+
\underline{m}\int_{\Omega \setminus B_1}\varphi^2\,dx\\&&\qquad=
\underline{m}\int_{B_1\setminus B_\rho}\left(c_\star+\frac{\log |x|}{\log\rho}\right)^2\,dx+
\underline{m}\int_{\Omega \setminus B_1} c_\star^2\,dx.
\end{eqnarray*}
We also remark that
\begin{equation}\label{n4rfgbgbx8ougc93tyevg}
\left(\frac{\log |x|}{\log\rho}\right)^2\,
\chi_{B_1\setminus B_\rho}\le 1.\end{equation}
Therefore, by the Dominated Convergence Theorem, and recalling~\eqref{EQUIVARHO},
\begin{eqnarray*}&&
\lim_{\rho\searrow0}\overline{m}\int_D \varphi^2\,dx
-
\underline{m}\int_{\Omega \setminus D}\varphi^2\,dx\\&&
\qquad=
\lim_{\overline{m}\nearrow+\infty}\frac{\overline{m}\underline{m}^2(
\underline{m}+m_0)}{m_0^2(\underline{m}+
\overline{m})}|\Omega|- \lim_{\rho\searrow0}\left(
\underline{m}\int_{B_1\setminus B_\rho}\left(c_\star+\frac{\log |x|}{\log\rho}\right)^2\,dx+
\underline{m}\int_{\Omega \setminus B_1} c_\star^2\,dx\right)\\
&&\qquad=
\frac{\underline{m}^2(
\underline{m}+m_0)}{m_0^2}|\Omega|
-\underline{m}\,c_\star^2\,|\Omega|\\
&&\qquad =
\frac{\underline{m}^2(
\underline{m}+m_0)}{m_0^2}|\Omega|
-\underline{m}\,
\left(\frac{\underline{m}+m_0}{m_0}
\right)^2\,|\Omega|\\&&\qquad=
\frac{\underline{m}(\underline{m}+m_0)}{m_0^2}\left(\underline{m}-
(\underline{m}+m_0)
\right)
|\Omega|\\&&\qquad
=-\frac{\underline{m}(\underline{m}+m_0)}{m_0}
|\Omega|,
\end{eqnarray*}
which is positive, since~$m_0\in(-\underline{m},0)$,
as desired.
\end{proof}

In the case~$s\in\left(0,\frac12\right]$, it is also convenient
to introduce the function
\begin{equation}\label{PSI44}
\psi(x):=
\begin{dcases}
c_\sharp-1  & {\mbox{ if }}x\in B_\rho,
\\
c_\sharp-\displaystyle\frac{\log|x|}{\log\rho}  & {\mbox{ if }}x\in B_1\setminus B_\rho,
\\
c_\sharp  & {\mbox{ if }}x\in\R^2\setminus B_1,
\end{dcases}
\end{equation}
where~$c_\sharp$ is defined in~\eqref{sharp}.
We also set
\begin{equation}\label{PSI45}
D:=\Omega\setminus B_\rho,\end{equation} and we study the main properties
of the auxiliary function~$\psi$.

Comparing~\eqref{VARPHI44} with~\eqref{PSI44}, we observe that~$
|\psi(x)-\psi(y)|=|\varphi(x)-\varphi(y)|$, and therefore,
from Lemma~\ref{gradlogto0}, we obtain that:

\begin{lem}\label{gradlogto02}
Let $n=1$, $s\in (0,1/2]$ and~$\psi$ be as in~\eqref{PSI44}. Then,
\[ \lim_{\rho\searrow0}
\iint_\Q \frac{|\psi(x)-\psi(y)|^2}{|x-y|^{1+2s}}\,dx\,dy =0.
\]
\end{lem}

We also have that:

\begin{lem}\label{denomlog2}
Let $n=1$, $s\in (0,1/2]$ and~$\psi$ be as in~\eqref{PSI44}. Then,
\[ \lim_{\rho\searrow0}
\int_D \psi^2\,dx-\underline{m}\int_{\Omega\setminus D} \psi^2\,dx
=
-\frac{\overline{m}(\overline{m}-m_0)}{m_0}|\Omega|>0.
\]
\end{lem}

\begin{proof}
In view of~\eqref{PSI44}, \eqref{PSI45}, \eqref{sharp} and~\eqref{misuraD},
\[
\underline{m}\int_{\Omega\setminus D} \psi^2\,dx=\underline{m}
(c_\sharp-1)^2 |\Omega\setminus D|=
\underline{m}\left( \frac{m_0-\overline{m}}{m_0}-1\right)^2
\frac{\overline{m}-m_0}{\underline{m}+\overline{m}}
|\Omega|=
\frac{\underline{m}\overline{m}^2(\overline{m}-m_0)}{m_0^2(\underline{m}+\overline{m})}
|\Omega|,
\]
and
\[\overline{m}
\int_D \psi^2\,dx =\overline{m}
\int_{B_1\setminus B_\rho}\left(c_\sharp-\frac{\log|x|}{\log\rho}\right)^2\,dx+\overline{m}
\int_{\Omega\setminus B_1}c_\sharp^2\,dx
.
\]
Hence, recalling~\eqref{CHIRO}
and~\eqref{n4rfgbgbx8ougc93tyevg}, and using the Dominated
Convergence Theorem,
\begin{eqnarray*}
&&\lim_{\rho\searrow0}\overline{m}
\int_D \psi^2\,dx-\underline{m}\int_{\Omega\setminus D} \psi^2\,dx\\
&=&\lim_{\rho\searrow0}
\overline{m}
\int_{B_1\setminus B_\rho}\left(c_\sharp-\frac{\log|x|}{\log\rho}\right)^2\,dx+\overline{m}
\int_{\Omega\setminus B_1}c_\sharp^2\,dx-\lim_{\underline{m}\nearrow+\infty}
\frac{\underline{m}\overline{m}^2(\overline{m}-m_0)}{m_0^2(\underline{m}+\overline{m})}
|\Omega|\\&=&
\overline{m}c_\sharp^2|\Omega|-
\frac{\overline{m}^2(\overline{m}-m_0)}{m_0^2}
|\Omega|\\&=&\frac{
\overline{m}(m_0-\overline{m})^2}{m_0^2}|\Omega|
-\frac{\overline{m}^2(\overline{m}-m_0)}{m_0^2}|\Omega|
\\&=&\frac{\overline{m}(\overline{m}-m_0)}{m_0^2}
\big((\overline{m}-m_0)-\overline{m}\big)
|\Omega|\\&=&-\frac{\overline{m}(\overline{m}-m_0)}{m_0}|\Omega|,
\end{eqnarray*}
which is positive since $m_0<0<\overline{m}$.
\end{proof}

We are now ready to complete the proof of Theorem~\ref{alpha=0} in the case~$s\in(0,1/2]$.

\begin{proof}[Proof of Theorem \ref{alpha=0} when~${s\in(0,1/2]}$]
To prove \eqref{LK:LK:01}, we exploit the auxiliary function~$\varphi$
introduced in~\eqref{VARPHI44} and the choice of
the resource~$m\in\tilde{\mathscr{M}}$, as defined in~\eqref{tildem},
with~$D$ as in~\eqref{VARPHI45}.

In this way, in light of~\eqref{LAMSPOTTP},
$$ \underline{\lambda}(\overline{m},\underline{m},m_0)\le
\frac{\displaystyle\frac\beta4\,\iint_\Q \frac{|\varphi(x)-\varphi(y)|^2}{|x-y|^{1+2s}}\,dx\,dy}
{\overline{m}\displaystyle\int_D \varphi^2\,dx-\underline{m}\int_{\Omega\setminus D} \varphi^2\,dx}.
$$
Hence, taking~$\emme:=\overline{m}$ and utilizing
Lemmata~\ref{gradlogto0} and~\ref{denomlog}, we find that
\[ \lim_{ \emme\nearrow+\infty}
\underline{\lambda}(\emme,\underline{m},m_0) \leq \lim_{ 
\rho\searrow0}
\frac{\displaystyle\frac\beta4\,\iint_\Q \frac{|\varphi(x)-\varphi(y)|^2}{|x-y|^{1+2s}}\,dx\,dy}{
\emme\displaystyle\int_D \varphi^2\,dx-\underline{m}\int_{\Omega\setminus D} \varphi^2\,dx}=0.
\]
This proves~\eqref{LK:LK:01}, and we now focus on the
proof of~\eqref{LK:LK:02}.
To this end, we exploit
the auxiliary function~$\psi$
introduced in~\eqref{PSI44} and the choice of
the resource~$m\in\tilde{\mathscr{M}}$, as defined in~\eqref{tildem},
with~$D$ as in~\eqref{PSI45}.

In this framework, in light of~\eqref{LAMSPOTTP},
$$ \underline{\lambda}(\overline{m},\underline{m},m_0)\le
\frac{\displaystyle\frac\beta4\,\iint_\Q \frac{|\psi(x)-\psi(y)|^2}{|x-y|^{1+2s}}\,dx\,dy}
{\overline{m}\displaystyle\int_D \psi^2\,dx
-\underline{m}\int_{\Omega\setminus D} \psi^2\,dx}.
$$
As a result, taking~$\emme:=\underline{m}$ and utilizing
Lemmata~\ref{gradlogto02} and~\ref{denomlog2}, we conclude that
\[ \lim_{ \emme\nearrow+\infty}
\underline{\lambda}(\overline{m},\emme,m_0) \leq \lim_{ 
\rho\searrow0}
\frac{\displaystyle\frac\beta4\,\iint_\Q \frac{|\psi(x)-\psi(y)|^2}{|x-y|^{1+2s}}\,dx\,dy}{
\overline{m}\displaystyle\int_D \psi^2\,dx-\emme
\int_{\Omega\setminus D} \psi^2\,dx}=0,
\]
thus establishing~\eqref{LK:LK:02}.
\end{proof}

\section{Badly displayed resources,
hectic oscillations and proof
of Theorem~\ref{oink8495-0243905i8uhjssjjcjnj7}}\label{oink8495-0243905i8uhjssjjcjnj7SS}

This section contains
the proof of
Theorem~\ref{oink8495-0243905i8uhjssjjcjnj7},
relying on an explicit
example of sequence of highly
oscillating resources which make the first
eigenvalue diverge. The technical
details go as follows.

\begin{proof}[Proof of Theorem~\ref{oink8495-0243905i8uhjssjjcjnj7}]
We suppose that~$B_4\subset\Omega$
and we consider~$\eta\in C^\infty_0(B_{3/2},[0,1])$ with~$\eta=1$
in~$B_1$ and~$\|\eta\|_{C^1(\R^n)}\le 8$.
We let
$$ m_\omega:=m_0-\frac{\Lambda}{|\Omega|}\int_\Omega\eta(x)\sin(\omega x_1)\,dx$$
and
$$ m(x):=m_\omega+\Lambda\eta(x)\,\sin(\omega x_1),$$
with~$\omega>0$ to be taken arbitrarily large in what follows.

We remark that
\begin{equation}\label{TB-x1} \frac{1}{|\Omega|}\int_\Omega m(x)\,dx=
m_\omega+\frac{\Lambda}{|\Omega|}\int_\Omega
\eta(x)\,\sin(\omega x_1)\,dx=m_0.\end{equation}
Moreover, integrating by parts,
\begin{equation*}
\begin{split}&
|m_\omega-m_0|=
\frac{\Lambda}{|\Omega|}\left|\int_\Omega\eta(x)\sin(\omega x_1)\,dx\right|=
\frac{\Lambda}{|\Omega|\,\omega}\left|\int_{\Omega}\eta(x)\frac{d}{dx_1}\cos(\omega x_1)\,dx\right|\\
&\qquad=
\frac{\Lambda}{|\Omega|\,\omega}\left|\int_{\Omega}\partial_1\eta(x)\cos(\omega x_1)\,dx\right|
\le\frac{8\Lambda}{\omega}
\end{split}\end{equation*}
which is arbitrarily small provided that~$\omega$ is large enough:
in particular, we can suppose that
\begin{equation}\label{RYCbb956}
2m_0\le m_\omega\le \frac{m_0}{2}.
\end{equation}
Also, for every~$x\in\Omega$,
\begin{equation}\label{TB-x2} |m(x)|\le |m_\omega|+\Lambda\le
2|m_0|+\Lambda\le2\Lambda.
\end{equation}
Furthermore, for large~$\omega$ we have that
$$ p^\pm:=\left(\pm\frac\pi{2\omega},0,\dots,0\right)\in B_1\subset\Omega\cap\{\eta=1\}.$$
Therefore
\begin{equation}\label{TB-x3}
\begin{split}
& \sup_\Omega m\ge m(p^+)=
m_\omega+\Lambda\ge\Lambda-2|m_0|\ge\frac{\Lambda}2\\
{\mbox{and }}\qquad&
\inf_\Omega m\le m(p^-)=
m_\omega-\Lambda\le -\Lambda.
\end{split}
\end{equation}
In view of~\eqref{TB-x0}, \eqref{TB-x1}, \eqref{TB-x2} and~\eqref{TB-x3},
we obtain that
\begin{equation}\label{andacxoapo056}
m\in {\mathscr{M}}^\sharp_{\Lambda,m_0}.
\end{equation}

Now, we take into account a function~$\varphi\in X_{\alpha,\beta}$ such that
$$ \int_\Omega m(x)\varphi^2(x)\,dx=1.$$
Then, integrating by parts, we see that
\begin{eqnarray*}
1&=& \int_\Omega \Big(m_\omega+\Lambda\eta(x)\,\sin(\omega x_1)\Big)\varphi^2(x)\,dx\\
&=& m_\omega\int_\Omega \varphi^2(x)\,dx
+\Lambda\int_{\R^n} \eta(x)\,\sin(\omega x_1)\varphi^2(x)\,dx\\
&=& m_\omega\int_\Omega \varphi^2(x)\,dx
-\frac\Lambda\omega\int_{\R^n} \eta(x)\,
\frac{d}{dx_1}
\cos(\omega x_1)\varphi^2(x)\,dx\\
&=& m_\omega\int_\Omega \varphi^2(x)\,dx
+\frac\Lambda\omega\int_{\R^n} \partial_1\eta(x)\,
\cos(\omega x_1)\varphi^2(x)\,dx
+\frac{2\Lambda}\omega\int_{\R^n} \eta(x)\,
\cos(\omega x_1)\varphi(x)\,\partial_1\varphi(x)\,dx\\
&\le& m_\omega\int_\Omega \varphi^2(x)\,dx
+\frac{8\Lambda}\omega\int_{\R^n}\varphi^2(x)\,dx
+\frac{2\Lambda}\omega\int_{\R^n}\varphi(x)\,
|\nabla \varphi(x)|\,dx
\\
&\le& m_\omega\int_\Omega \varphi^2(x)\,dx
+\frac{9\Lambda}\omega\int_{\R^n}\varphi^2(x)\,dx
+\frac{\Lambda}\omega\int_{\R^n}|\nabla \varphi(x)|^2\,dx.
\end{eqnarray*}
As a consequence, if~$\frac{9\Lambda}\omega\le-m_\omega$
(which is the case for~$\omega$ large, in view of~\eqref{RYCbb956}),
$$ 1\le \frac{\Lambda}\omega\int_{\R^n}|\nabla \varphi(x)|^2\,dx,$$
and accordingly
\begin{equation}\label{Alp78-923840492} \frac{\displaystyle\frac\alpha2\,
\int_\Omega|\nabla\varphi(x)|^2\,dx
}{\displaystyle\int_\Omega m(x)\varphi^2(x)\,dx}\ge
\frac{\alpha\omega}{2\Lambda}.\end{equation}
Let now~$\zeta\in(-1,1)$
and~$E'\in\R^{n-1}$ with~$|E'|\le|\zeta|$ and~$E:=(\zeta,E')\in\R^n$.
We use the trigonometric identity
$$ \frac{\cos(y) \cos\zeta-\cos(y+\zeta)}{\sin \zeta}=\sin y,\qquad{\mbox{ for all $
\zeta\in\R\setminus(\pi\Z)$
and~$y\in\R$,}}$$
together with the notation~$\Phi:=\eta\varphi^2$
and the change of variable
$$X:=x+\frac1\omega
\left(\zeta,E' \right)=x+\frac{E}{\omega},$$ to write that
\begin{eqnarray*}
1&=& \int_\Omega \Big(m_\omega+\Lambda\eta(x)\,\sin(\omega x_1)\Big)\varphi^2(x)\,dx\\
&=& m_\omega\int_\Omega \varphi^2(x)\,dx
+\Lambda\int_{\R^n} \sin(\omega x_1)\Phi(x)\,dx\\
&=& m_\omega\int_\Omega \varphi^2(x)\,dx
+\frac\Lambda{\sin \zeta}\;\int_{\R^n} 
\Big(\cos(\omega x_1) \cos\zeta-\cos(\omega x_1+\zeta)\Big)
\Phi(x)\,dx\\&=& m_\omega\int_\Omega \varphi^2(x)\,dx\\&&\qquad
+\frac\Lambda{\sin \zeta}\left[
\int_{\R^n} 
\cos(\omega x_1) \cos\zeta\;\Phi(x)\,dx
-\int_{\R^n} 
\cos(\omega X_1)\;\Phi\left(
X-\frac{E}\omega\right)\,dX
\right]\\
&=& m_\omega\int_\Omega \varphi^2(x)\,dx\\&&\qquad
+\frac\Lambda{\sin \zeta}
\int_{\R^n} 
\cos(\omega x_1)\left[ \cos\zeta\;\Phi(x)-
\Phi\left(x-\frac{E}\omega\right) \right]\,dx\\
&=& m_\omega\int_\Omega \varphi^2(x)\,dx
+\frac\Lambda{\sin \zeta}
\int_{\R^n} 
\cos(\omega x_1)\left( \cos\zeta-1\right)\,\Phi(x)\,dx
\\&&\qquad+\frac\Lambda{\sin \zeta}
\int_{\R^n} 
\cos(\omega x_1)\left(\Phi(x)-
\Phi\left(x-\frac{E}\omega\right) \right)\,dx.
\end{eqnarray*}
Since
\begin{eqnarray*}
&&\left|\Phi(x)-
\Phi\left(x-\frac{E}\omega\right) \right|
\\
&\le& \varphi^2(x)\,
\left|\eta(x)-
\eta\left(x-\frac{E}\omega\right) \right|
+\eta\left(x-\frac{E}\omega\right)\,\left| \varphi^2(x)-
\varphi^2\left(x-\frac{E}\omega\right)\right|
\\&\le&
\frac{8\varphi^2(x)\,|E|}{\omega}
+\left| \varphi^2(x)-
\varphi^2\left(x-\frac{E}\omega\right)\right|,
\end{eqnarray*}
we thereby discover that, if~$E\in B_1$ and~$\omega\ge2$,
\begin{eqnarray*}
1&\le& \left(m_\omega+\frac{8\Lambda\,|E|}{\omega\,|\sin\zeta|}\right)\int_\Omega \varphi^2(x)\,dx
+\frac\Lambda{|\sin \zeta|}
\int_{\Omega} \left(1- \cos\zeta\right)\varphi^2(x)\,dx
\\&&\qquad+\frac\Lambda{|\sin \zeta|}
\int_{B_2} \left| \varphi^2(x)-
\varphi^2\left(x-\frac{E}\omega\right)\right|\,dx
\\
&\le& \left(m_\omega+\frac{C\Lambda}{\omega}\right)\int_\Omega \varphi^2(x)\,dx
+C\Lambda\,|\zeta|
\int_{\Omega} \varphi^2(x)\,dx
\\&&\qquad+\frac{C\Lambda}{|\zeta|}
\int_{B_2} \left| \varphi^2(x)-
\varphi^2\left(x-\frac{E}\omega\right)\right|\,dx
,\end{eqnarray*}
for some~$C>0$.

In particular, recalling also~\eqref{RYCbb956},
it follows that
there exists~$r_0\in(0,1)$, possibly depending
on~$m_0$, $\Lambda$ and~$n$, such that, if~$\zeta\in (-r_0,r_0)$
and~$\omega$ is sufficiently large,
\begin{equation}\label{osommSObbB89ir}
1\le \frac{m_\omega}2\,\int_\Omega \varphi^2(x)\,dx+\frac{C\Lambda}{|\zeta|}
\int_{B_2} \left| \varphi^2(x)-
\varphi^2\left(x-\frac{E}\omega\right)\right|\,dx
.\end{equation}
We also observe that, given an additional parameter~$\kappa>0$,
to be taken conveniently small in what follows,
\begin{eqnarray*}&&
2\,\left| \varphi^2(x)-
\varphi^2\left(x-\frac{E}\omega\right)\right|=2\,
\left| \varphi(x)+
\varphi\left(x-\frac{E}\omega\right)\right|\,
\left| \varphi(x)-
\varphi\left(x-\frac{E}\omega\right)\right|\\&&\qquad\le
\kappa|\zeta|^{n+2s-1}\,
\left| \varphi(x)+
\varphi\left(x-\frac{E}\omega\right)\right|^2+\kappa^{-1}
|\zeta|^{1-n-2s}\,\left| \varphi(x)-
\varphi\left(x-\frac{E}\omega\right)\right|^2\\&&\qquad\le
4\kappa|\zeta|^{n+2s-1}\varphi^2(x)+
4\kappa|\zeta|^{n+2s-1}\varphi^2\left(x-\frac{E}\omega\right)+\kappa^{-1}
|\zeta|^{1-n-2s}\left| \varphi(x)-
\varphi\left(x-\frac{E}\omega\right)\right|^2.
\end{eqnarray*}
Then, we plug this information into~\eqref{osommSObbB89ir}
and we conclude that, if~$r_0$ is small enough and~$\omega$ is large enough,
\begin{equation}\label{osommSObbB89ir2}\begin{split}
1\,&\le \frac{m_\omega}2\,\int_\Omega \varphi^2(x)\,dx\\
&\qquad+\frac{C\Lambda}{|\zeta|}
\Bigg[\kappa\int_{B_2}|\zeta|^{n+2s-1}
\varphi^2(x)\,dx+\kappa\int_{B_2}|\zeta|^{n+2s-1}
\varphi^2\left(x-\frac{E}\omega\right)\,dx\\&\qquad\qquad+\frac1\kappa\int_{B_2}|\zeta|^{1-n-2s}\,
\left| \varphi(x)-
\varphi\left(x-\frac{E}\omega\right)\right|^2
\,dx\Bigg]\\
&\le \frac{m_\omega}2\,\int_\Omega \varphi^2(x)\,dx\\
&\qquad+\frac{C\Lambda}{|\zeta|}
\left[2\kappa\,|\zeta|^{n+2s-1}\int_{B_4}
\varphi^2(x)\,dx+\frac1\kappa\,\int_{B_2}
|\zeta|^{1-n-2s}\,\left| \varphi(x)-
\varphi\left(x-\frac{E}\omega\right)\right|^2
\,dx\right].
\end{split}\end{equation}
We also remark that, in our notation, $E_1=\zeta$, and accordingly
\begin{eqnarray*}&& \iint_{B_4\times (B_{r_0}\cap\{|E'|\le|E_1|\})}|\zeta|^{n+2s-2}
\varphi^2(x)\,dx\,dE\le
\iint_{B_4\times B_{r_0}}|E_1|^{n+2s-2}
\varphi^2(x)\,dx\,dE\\&&\qquad\qquad\le \tilde C r_0^{2n+2s-2}
\int_{B_4}\varphi^2(x)\,dx,
\end{eqnarray*}
for some~$\tilde C>0$.

For this reason, letting~$\iota$ be the measure of the set~$\{x=(x_1,x')\in
B_1{\mbox{ s.t. }}|x'|\le|x_1|\}$, we obtain that
\begin{equation*}
\begin{split}&
\frac{m_\omega}2\,\iint_{\Omega\times (B_{r_0}\cap\{|E'|\le|E_1|\})} \varphi^2(x)\,dx\,dE
+2C\kappa \Lambda
\iint_{B_4\times (B_{r_0}\cap\{|E'|\le|E_1|\})}|\zeta|^{n+2s-2}
\varphi^2(x)\,dx\,dE\\&\qquad
\le
\frac{m_\omega\,\iota\,r_0^n}2\,\int_{\Omega} \varphi^2(x)\,dx+
2C\tilde{C}\kappa \Lambda r_0^{2n+2s-2}
\int_{B_4}\varphi^2(x)\,dx
\le 0,
\end{split}
\end{equation*}
by choosing
$$ \kappa:=\min\left\{
1,\;-\frac{m_\omega\,\iota}{C\tilde{C}\Lambda r^{n+2s-2}_0}
\right\}.$$
Therefore, we can integrate~\eqref{osommSObbB89ir2}
over~$E\in B_{r_0}\cap\{|E'|\le|E_1|\}$
and up to renaming constants we find that
\begin{equation*}\begin{split}
r_0^n\,&\le C\Lambda\iint_{B_2\times(B_{r_0}\cap\{|E'|\le|E_1|\})}
|E_1|^{-n-2s}\,\left| \varphi(x)-
\varphi\left(x-\frac{E}\omega\right)\right|^2
\,dx\,dE\\&
\le C\Lambda\iint_{B_2\times(B_{r_0}\cap\{|E'|\le|E_1|\})}
|E|^{-n-2s}\,\left| \varphi(x)-
\varphi\left(x-\frac{E}\omega\right)\right|^2
\,dx\,dE\\&=
C\Lambda\omega^{-2s} \iint_{B_2\times(B_{r_0/\omega}\cap\{|z'|\le|z_1|\})}
|z|^{-n-2s}\,\left| \varphi(x)-
\varphi(x-z)\right|^2
\,dx\,dz\\&
\le C\Lambda\omega^{-2s} \iint_{B_2\times\R^n}
\frac{|\varphi(x)-
\varphi(x-z)|^2}{|z|^{n+2s}}
\,dx\,dz\\&
\le C\Lambda\omega^{-2s} \iint_{\Q}
\frac{|\varphi(x)-
\varphi(y)|^2}{|x-y|^{n+2s}}
\,dx\,dy,
\end{split}\end{equation*}
and consequently
$$ \frac{\displaystyle\frac\beta4\iint_{\Q}
\frac{|\varphi(x)-
\varphi(y)|^2}{|x-y|^{n+2s}}
\,dx\,dy}{\displaystyle\int_\Omega m(x)\varphi^2(x)\,dx}\ge
\frac{\beta r_0^n\omega^{2s}}{C\Lambda},$$
up to renaming~$C>0$.

This and~\eqref{Alp78-923840492}, recalling~\eqref{andacxoapo056}, give that
$$ \lambda_1(m)\ge
\frac{\alpha\omega}{2\Lambda}+
\frac{\beta r_0^n\omega^{2s}}{C\Lambda},$$
which, taking~$\omega$ as large as we wish, yields the desired result.
\end{proof}

\appendix

\section{Proofs of Theorems \ref{msopratoinfty} and
\ref{msottotoinfty} when $n=2$}\label{KASN:ALSL}

The main strategy followed in this part is similar
to the case~$n\ge3$, but when~$n=2$ we have to define different
auxiliary
functions. We start with the proof of Theorem~\ref{msopratoinfty}.
For this, we recall the setting in~\eqref{B2cont}
and~\eqref{VECC7}, and we define
\begin{equation}\label{vARPHI2}
\varphi(x):=
\begin{dcases}
c_\star+1  & {\mbox{if }}\;x\in B_\rho,
\\
c_\star+\displaystyle\frac{\log|x|}{\log\rho}  &{\mbox{if }}\; x\in
B_1\setminus B_\rho,
\\
c_\star  & {\mbox{if }}\;x\in\R^2\setminus B_1.
\end{dcases}
\end{equation}
We set 
\begin{equation}\label{DDEFVB}
D:=B_\rho\end{equation}
and we list below some interesting properties
of~$\varphi$:

\begin{lem}\label{n2gradphito0}
Let $n=2$ and~$\varphi$ be as in~\eqref{vARPHI2}. Then 
\[ \lim_{\rho\searrow0}
\int_\Omega |\nabla \varphi|^2 = 0
.\]
\end{lem}

\begin{proof}
We compute
\[
\frac{1}{(\log\rho)^2}
\int_{B_1\setminus B_\rho} \frac{1}{|x|^2}\,dx
=\frac{2\pi}{(\log\rho)^2}
\int_\rho^1 \frac{1}{r}\,dr
=-\frac{2\pi}{\log\rho} \longrightarrow 0
\]
as $\rho \searrow 0$.
\end{proof}

\begin{lem}\label{n2gradsphito0}
Let $n=2$ and~$\varphi$ be as in~\eqref{vARPHI2}. Then
\[ \lim_{\rho\searrow0}
\iint_\Q \frac{|\varphi(x)-\varphi(y)|^2}{|x-y|^{2+2s}}\,dx\,dy = 0
.\]
\end{lem}

\begin{proof}
As for the proof of Lemma \ref{gradsphito0}, we have to consider several integral
contributions (given the different expressions of the competitors
the technical computations here are different from those
in Lemma \ref{gradsphito0}). First of all, we have that 
\[
\iint_{B_\rho\times B_\rho}
\frac{|\varphi(x)-\varphi(y)|^2}{|x-y|^{2+2s}}\,dx\,dy=0	
\]
and
\[
\iint_{(\R^2\setminus B_1)\times (\R^2\setminus B_1)}
\frac{|\varphi(x)-\varphi(y)|^2}{|x-y|^{2+2s}}\,dx\,dy=0.	
\]
Moreover, assuming~$\rho<1/4$,
\begin{eqnarray*}
&& \iint_{B_\rho\times (\R^2\setminus B_1)}
\frac{|\varphi(x)-\varphi(y)|^2}{|x-y|^{2+2s}}\,dx\,dy
=\iint_{B_\rho\times (\R^2\setminus B_1)}
\frac{1}{|x-y|^{2+2s}}\,dx\,dy\\
&&\qquad\leq \int_{B_\rho}dx\int_{\R^2\setminus B_{\frac{1}{2}}}\frac{1}{|z|^{2+2s}}\,dz
\leq C|B_1|\rho^2 \longrightarrow 0
\end{eqnarray*}
as $\rho \searrow 0$.

Furthermore, if $(x,y)\in(B_1\setminus B_\rho)\times B_\rho$ we have
\[
|\varphi(x)-\varphi(y)|^2=\frac{1}{(\log\rho)^2}|\log|x|-\log\rho|^2
.
\]
Consequently, from \eqref{disuglog} (used here with~$|y|=\rho$),
\begin{align*}
\iint_{(B_1\setminus B_\rho)\times B_\rho}
\frac{|\varphi(x)-\varphi(y)|^2}{|x-y|^{2+2s}}\,dx\,dy
\,&=
\frac{1}{(\log\rho)^2}\iint_{(B_1\setminus B_\rho)\times B_\rho}
\frac{|\log|x|-\log\rho|^2}{|x-y|^{2+2s}}\,dx\,dy \\
&\leq \frac{1}{\rho^2(\log\rho)^2}
\iint_{(B_1\setminus B_\rho)\times B_\rho}
\frac{(|x|-\rho)^2}{|x-y|^{2+2s}}\,dx\,dy \\
&\leq \frac{1}{\rho^2(\log\rho)^2}
\iint_{(B_1\setminus B_\rho)\times B_\rho}|x-y|^{-2s}\,dx\,dy \\
&\leq \frac{1}{\rho^2(\log\rho)^2}
\int_{B_\rho}dy\int_{B_2}|z|^{-2s}\,dz \\
&\leq \frac{C }{(\log\rho)^2}\longrightarrow 0
\end{align*}
as $\rho \searrow 0$.

Also, exploiting~\eqref{B2cont}, we have
\begin{align*}
\iint_{(B_1\setminus B_\rho)\times (\R^2\setminus \Omega)}
\frac{|\varphi(x)-\varphi(y)|^2}{|x-y|^{2+2s}}\,dx\,dy\,&=
\frac{1}{(\log\rho)^2}\iint_{(B_1\setminus B_\rho)\times (\R^2\setminus
\Omega)}
\frac{(\log|x|)^2}{|x-y|^{2+2s}}\,dx\,dy \\
&\leq \frac{1}{(\log\rho)^2}
\int_{B_1\setminus B_\rho}(\log|x|)^2\,dx
\int_{\R^2\setminus B_1}\frac{1}{|z|^{2+2s}}\,dz \\
&\leq \frac{C}{(\log\rho)^2}
\int_{B_1\setminus B_\rho}(\log|x|)^2\,dx \\
&\le\frac{C}{(\log\rho)^2}\longrightarrow 0
\end{align*}
as $\rho \searrow 0$.

Now from \eqref{disuglog}, used here with~$|y|=1$, we have
\begin{align*}
\iint_{(B_1\setminus B_\rho)\times (\Omega\setminus B_1)}
\frac{|\varphi(x)-\varphi(y)|^2}{|x-y|^{2+2s}}\,dx\,dy
\,&=
\frac{1}{(\log\rho)^2}
\iint_{(B_1\setminus B_\rho)\times (\Omega\setminus B_1)}
\frac{(\log|x|)^2}{|x-y|^{2+2s}}\,dx\,dy \\
&\leq \frac{1}{(\log\rho)^2}
\iint_{(B_1\setminus B_\rho)\times (\Omega\setminus B_1)}
\frac{(1-|x|)^2}{|x|^2|x-y|^{2+2s}}\,dx\,dy \\
&\leq \frac{1}{(\log\rho)^2}
\int_{B_1\setminus B_\rho}\frac{1}{|x|^2}\,dx
\int_{B_{R+1}}|z|^{-2s}\,dz \\
&\leq \frac{C}{(\log\rho)^2}
\int_{B_1\setminus B_\rho}\frac{1}{|x|^2}\,dx
=-\frac{C}{\log\rho} \longrightarrow 0
\end{align*}
as $\rho \searrow 0$, where we took~$R>2$ sufficiently
large such that~$\Omega \subset B_R$.

In addition, utilizing again~\eqref{disuglog}, we first notice that
\begin{align*}
\iint_{(B_1\setminus B_\rho)\times (B_1\setminus B_\rho)}
\frac{|\varphi(x)-\varphi(y)|^2}{|x-y|^{2+2s}}\,dx\,dy
\,&=\frac1{(\log\rho)^2}
\iint_{(B_1\setminus B_\rho)\times (B_1\setminus B_\rho)}
\frac{(\log|x|-\log|y|)^2}{|x-y|^{2+2s}}\,dx\,dy\\&
=\frac2{(\log\rho)^2}\iint_{(B_1\setminus B_\rho)\times (B_1\setminus B_\rho)
\atop |x|\leq |y|}
\frac{(\log|x|-\log|y|)^2}{|x-y|^{2+2s}}\,dx\,dy.
\\&\leq \frac{2}{(\log\rho)^2}
\iint_{(B_1\setminus B_\rho)\times (B_1\setminus B_\rho)
\atop |x|\leq |y|}
\frac{|x-y|^2}{|x|^2|x-y|^{2+2s}}\,dx\,dy \\
&\leq \frac{2}{(\log\rho)^2}
\int_{B_1\setminus B_\rho}\frac{1}{|x|^2}\,dx
\int_{B_2}|z|^{-2s}\,dz \\
&\leq \frac{C}{(\log\rho)^2}
\int_{B_1\setminus B_\rho}\frac{1}{|x|^2}\,dx
=-\frac{C}{\log\rho}\longrightarrow 0
\end{align*}
as $\rho \searrow 0$, which concludes the proof.
\end{proof}

\begin{lem}\label{n2denomphi}
Let $n=2$ and~$\varphi$ be as in~\eqref{vARPHI2}. Then
\[ \lim_{\rho\searrow0}\;
\overline{m}\int_D \varphi^2\,dx-\underline{m}\int_{\Omega\setminus D} \varphi^2\,dx
=
-
\frac{\underline{m}(\underline{m}+m_0)|\Omega|}{m_0}>0
.\]
\end{lem}

\begin{proof}
By~\eqref{vARPHI2} and~\eqref{DDEFVB},
and recalling~\eqref{misuraD},
we have that 
\[
\overline{m}\int_D \varphi^2\,dx=
\overline{m}(c_\star+1)^2 |B_\rho|=
\overline{m}(c_\star+1)^2 \;\frac{\underline{m}+m_0}{\underline{m}+\overline{m}}\;|\Omega|.
\]
Hence, in light of~\eqref{VECC7} and~\eqref{EQUIVARHO},
$$ \lim_{\rho\searrow0}
\overline{m}\int_D \varphi^2\,dx=
\lim_{\overline{m}\nearrow+\infty}
\overline{m}\left(-\frac{\underline{m}+m_0}{m_0}+1\right)^2 \;\frac{\underline{m}+m_0}{\underline{m}+\overline{m}}\;|\Omega|
=
\left(-\frac{\underline{m}+m_0}{m_0}+1\right)^2 (\underline{m}+m_0)\,|\Omega|
.$$
Moreover, by~\eqref{VECC7} and the Dominated Convergence Theorem,
\begin{eqnarray*}&&
\underline{m}\int_{\Omega\setminus D} \varphi^2\,dx=
\underline{m}\int_{B_1\setminus B_\rho} \left(c_\star+\frac{\log|x|}{\log\rho}\right)^2\,dx
+
\underline{m}\, c_\star^2\,|\Omega\setminus B_1|
\\ &&\qquad\longrightarrow
\underline{m}\int_{B_1} \left(
-\frac{\underline{m}+m_0}{m_0}
\right)^2\,dx
+
\underline{m}\,\left(-\frac{\underline{m}+m_0}{m_0}\right)^2\,|\Omega\setminus B_1|
\\&&=
\underline{m}\,\left(-\frac{\underline{m}+m_0}{m_0}\right)^2\,|\Omega|,\end{eqnarray*}
as $\rho \searrow 0$.

As a result,
\begin{eqnarray*}
&&
\lim_{\rho\searrow0}\;
\overline{m}\int_D \varphi^2\,dx-\underline{m}\int_{\Omega\setminus D} \varphi^2\,dx
=\left[
\left(-\frac{\underline{m}+m_0}{m_0}+1\right)^2 (\underline{m}+m_0)
-\underline{m}\,\left(-\frac{\underline{m}+m_0}{m_0}\right)^2\right]
\,|\Omega|\\&&\qquad=
\left[
\underline{m}^2(\underline{m}+m_0)-(\underline{m}+m_0)^2\underline{m}
\right]\,\frac{|\Omega|}{m_0^2}=
\left[
\underline{m}-(\underline{m}+m_0)
\right]\,\frac{\underline{m}(\underline{m}+m_0)|\Omega|}{m_0^2}\\
&&\qquad=-
\frac{\underline{m}(\underline{m}+m_0)|\Omega|}{m_0}
,\end{eqnarray*}
which is positive, since~$m_0\in(-\underline{m},0)$.
\end{proof}

With this preliminary work, we can complete the
proof of Theorem~\ref{msopratoinfty} in dimension~$n=2$,
by arguing as follows:

\begin{proof}[Proof of Theorem \ref{msopratoinfty} when $n=2$]
We use the function~$\varphi$ in~\eqref{vARPHI2} and
the resource~$m:=\overline{m}\chi_D-\underline{m}\chi_{\Omega\setminus D}$,
with~$D$ as in~\eqref{DDEFVB}, as competitors 
in the minimization problem in~\eqref{LAMSPOTTP}.
In this way, we find that
\begin{equation}\label{n2infmsopra}
\underline{\lambda}(\overline{m},\underline{m},m_0)\leq 
\frac{\displaystyle\alpha\int_\Omega |\nabla \varphi|^2\,dx+
\beta\iint_\Q \frac{|\varphi(x)-\varphi(y)|^2}{|x-y|^{2+2s}}\,dx\,dy}
{\displaystyle\overline{m}\int_D \varphi^2\,dx-
\underline{m}\int_{\Omega\setminus D} \varphi^2\,dx}.
\end{equation}
{F}rom Lemmata~\ref{n2gradphito0} and~\ref{n2gradsphito0},
we have that
$$ \lim_{\rho\searrow0}
\alpha\int_\Omega |\nabla \varphi|^2\,dx+
\beta\iint_\Q \frac{|\varphi(x)-\varphi(y)|^2}{|x-y|^{2+2s}}\,dx\,dy=0.$$
Combining this with 
Lemma~\ref{n2denomphi} and~\eqref{EQUIVARHO}, we obtain
the desired result in Theorem~\ref{msopratoinfty}.
\end{proof} 

We now focus on the proof of Theorem \ref{msottotoinfty}
when~$n=2$. For this, we introduce the
function
\begin{equation}\label{OASDHCV77576887JGJ}
\psi(x):=
\begin{dcases}
c_\sharp-1  & {\mbox{if }}\;x\in B_\rho,
\\
c_\sharp-\displaystyle\frac{\log|x|}{\log\rho}  & {\mbox{if }}\;
x\in B_1\setminus B_\rho,
\\
c_\sharp  &{\mbox{if }}\;x\in \R^2\setminus B_1,
\end{dcases}
\end{equation}
where~$c_\sharp$ is the constant introduced in~\eqref{sharp}.
Moreover, we set
\begin{equation}\label{o407586ouhjdfkj}
D:=\Omega\setminus B_\rho.\end{equation} 

The proofs of the next two results follow directly from 
Lemmata~\ref{n2gradphito0} and~\ref{n2gradsphito0}, since,
comparing~\eqref{vARPHI2}
and~\eqref{OASDHCV77576887JGJ},
$$|\nabla \varphi|^2= |\nabla \psi|^2\qquad {\mbox{ and }}\qquad 
|\varphi(x)-\varphi(y)|^2=|\psi(x)-\psi(y)|^2.$$

\begin{lem}\label{n2gradpsito0}
Let $n=2$ and~$\psi$ be as in~\eqref{OASDHCV77576887JGJ}. Then 
\[
\lim_{\rho\searrow0}\int_\Omega |\nabla \psi|^2= 0.
\]
\end{lem}

\begin{lem}\label{n2gradspsito0}
Let $n=2$ and~$\psi$ be as in~\eqref{OASDHCV77576887JGJ}. Then
\[ \lim_{\rho\searrow0}
\iint_\Q \frac{|\psi(x)-\psi(y)|^2}{|x-y|^{2+2s}}\,dx\,dy = 0.
\]
\end{lem}

\begin{lem}\label{n2denompsi}
Let $n=2$ and~$\psi$ be as in~\eqref{OASDHCV77576887JGJ}. Then
\[
\lim_{\rho\searrow0}\overline{m}
\int_D \psi^2\,dx-\underline{m}\int_{\Omega\setminus D} \psi^2\,dx 
= \frac{\overline{m}\;(m_0-\overline{m})}{m_0}|\Omega|>0. 
\]
\end{lem}

\begin{proof}
By~\eqref{OASDHCV77576887JGJ}, \eqref{o407586ouhjdfkj}
and~\eqref{sharp}, we see that
\begin{equation}\label{dertuigri}
\underline{m}\int_{\Omega\setminus D} \psi^2\,dx
=\underline{m}\,(c_\sharp-1)^2|\Omega\setminus D|=
\underline{m}\,\left(\frac{m_0-\overline{m}}{m_0}-1\right)^2
|\Omega\setminus D|.
\end{equation}
Also, recalling~\eqref{misuraD},
we see that
$$ |\Omega\setminus D|= \frac{\overline{m}-m_0}{\underline{m}
+\overline{m}}|\Omega|.$$
Plugging this information into~\eqref{dertuigri}, we conclude that
\begin{equation}\label{irtu5yfjbnfdjm}\begin{split}
\underline{m}\int_{\Omega\setminus D} \psi^2\,dx=
\underline{m}\,\left(\frac{m_0-\overline{m}}{m_0}-1\right)^2
\frac{\overline{m}-m_0}{\underline{m}
+\overline{m}}|\Omega|
=
\frac{\underline{m}\;\overline{m}^2\;(\overline{m}-m_0)}{m_0^2\;
(\underline{m}
+\overline{m})}|\Omega|.
\end{split}\end{equation}
In addition, by the Dominated Convergence Theorem and~\eqref{CHIRO},
\begin{eqnarray*}
\lim_{\rho\searrow0}
\int_D \psi^2\,dx&=&\lim_{\rho\searrow0}
\int_{B_1\setminus B_\rho}\left(c_\sharp-\frac{
\log|x|}{\log\rho}\right)^2\,dx +
\int_{\Omega\setminus B_1} c_\sharp^2\,dx\\&=&
c_\sharp^2\;|B_1| +c_\sharp^2\;|\Omega\setminus B_1|\\&=&
c_\sharp^2\;|\Omega|\\
&=&\frac{(m_0-\overline{m})^2}{m_0^2}\;|\Omega|.
\end{eqnarray*}
This and~\eqref{irtu5yfjbnfdjm} give that
\begin{eqnarray*}&&
\lim_{\rho\searrow 0}
\overline{m}\int_D \psi^2\,dx-\underline{m}\int_{\Omega\setminus D} \psi^2\,dx 
= \frac{\overline{m}\;(m_0-\overline{m})^2}{m_0^2}|\Omega|
-\lim_{\underline{m}\nearrow+\infty}
\frac{\underline{m}\;\overline{m}^2\;(\overline{m}-m_0)}{m_0^2\;
(\underline{m}
+\overline{m})}|\Omega|\\&&\qquad
=\frac{\overline{m}\;(m_0-\overline{m})^2}{m_0^2}|\Omega|
-\frac{\overline{m}^2\;(\overline{m}-m_0)}{m_0^2}|\Omega|
=\frac{\overline{m}\;(\overline{m}-m_0)}{m_0^2}
\big( (\overline{m}-m_0)-\overline{m}
\big)|\Omega|\\&&\qquad=-
\frac{\overline{m}\;(\overline{m}-m_0)}{m_0}|\Omega|,
\end{eqnarray*}
which is positive since $m_0<0<\overline{m}$.
\end{proof}

With this preliminary work, we are in the position of completing
the proof of Theorem~\ref{msottotoinfty}
in the case~$n=2$.

\begin{proof}[Proof of Theorem~\ref{msottotoinfty} when $n=2$]
We use the function~$\psi$ in~\eqref{OASDHCV77576887JGJ}
and the resource~$m:=\overline{m}\chi_{D}-\underline{m}\chi_{\Omega
\setminus D}$, with~$D$ as in~\eqref{o407586ouhjdfkj}, as a competitor
in the minimization problem in~\eqref{LAMSPOTTP}, thus obtaining that
\begin{equation}\label{n2infmsotto}
\underline{\lambda}(\overline{m},\underline{m},m_0)\leq 
\frac{\displaystyle\alpha\int_\Omega |\nabla \psi|^2\,dx+
\beta\iint_\Q \frac{|\psi(x)-\psi(y)|^2}{|x-y|^{2+2s}}\,dx\,dy}
{\displaystyle\overline{m}
\int_D \psi^2\,dx-\underline{m}\int_{\Omega\setminus D} \psi^2\,dx}.
\end{equation}
{F}rom
Lemmata~\ref{n2gradpsito0} and~\ref{n2gradspsito0},
and recalling~\eqref{CHIRO}, we have that
$$ \lim_{\underline{m}\nearrow+\infty}
\alpha\int_\Omega |\nabla \psi|^2\,dx+
\beta\iint_\Q \frac{|\psi(x)-\psi(y)|^2}{|x-y|^{2+2s}}\,dx\,dy=0.$$
This, together with Lemma~\ref{n2denompsi}, gives the desired result.
\end{proof} 

\section{Probabilistic motivations
for the superposition of elliptic operators with different orders}\label{APPEB}

The goal of this appendix is to provide
a natural framework in which sums of local/nonlocal
operators naturally arise. Though the argument provided
can be extended to more general superpositions of operators,
for the sake of concreteness we limit ourselves to the operator
in~\eqref{problogis}.

For this, extending a presentation in~\cite{MR2584076},
we consider a discrete stochastic process
on the lattice~$h\Z^n$, with time increment~$\tau$.
The space scale~$h>0$ and the time step~$\tau>0$ will be conveniently chosen to
be infinitesimal in what follows.

We denote by~${\mathcal{B}}_0:=\{e_1,\dots,e_n\}$ the standard
Euclidean basis of~$\R^n$, we let~$
{\mathcal{B}}:={\mathcal{B}}_0\cup(-
{\mathcal{B}}_0)=
\{e_1,\dots,e_n,-e_1,\dots,-e_n\}$,
and
we suppose that a particle moves on~$h\Z^n$, and, given~$p\in[0,1]$,
$\lambda\in\N$, 
and~$s\in(0,1)$, its probability
of jumping from a point~$hk$ to~$h\tilde k$ (with~$k$, $\tilde k\in\Z^n$)
is given by
\begin{equation}\label{PROBA} {\mathcal{P}}(k,\tilde k):=
\frac{p}{c\,|k-\tilde k|^{n+2s}}+\frac{(1-p)\,{\mathcal{U}}(k-\tilde k)}{2n},\end{equation}
where
$$ c:=\sum_{j\in\Z^n\setminus\{0\}}\frac{1}{|j|^{n+2s}}$$
and
$$ {\mathcal{U}}(j):=\begin{dcases}
1 & {\mbox{ if }}j\in \lambda{\mathcal{B}}=\{\lambda e_1,\dots,\lambda e_n,
-\lambda e_1,\dots,-\lambda e_n
\},\\
0 &{\mbox{ otherwise}}.
\end{dcases}
$$
We point out that~${\mathcal{P}}(\tilde k,k)={\mathcal{P}}(k,\tilde k)=
{\mathcal{P}}(k-\tilde k,0)={\mathcal{P}}(\tilde k- k,0)$, and that
\begin{equation}\label{SOM1}
\begin{split}&
\sum_{j\in\Z^n\setminus\{0\}}
{\mathcal{P}}(j,0)\,
=\sum_{j\in\Z^n\setminus\{0\}}\left(
\frac{p}{c\,|j|^{n+2s}}+\frac{(1-p)\,{\mathcal{U}}(j)}{2n}\right)
\\&\qquad=p+\sum_{j\in\lambda{\mathcal{B}}}
\frac{(1-p)}{2n}
=p+(1-p)=1.
\end{split}
\end{equation}
The heuristic interpretation of the probability described in~\eqref{PROBA}
is that, at any time step, the particle has a probability~$p$
of following a jump process,
and a probability~$1-p$ of following a classical random walk.
Indeed, with probability~$p$, the particle experiences
a jump governed by the power law~$\frac{1}{c|j|^{n+2s}}$,
while with probability~$1-p$ it walks to one of the closest
neighbors scaled by the additional parameter~$\lambda$ (being all 
closest
neighbors equally probable, and being the probability
for the particle of not moving at all equal to zero).

Therefore, given~$x\in h\Z^n$ and~$t\in\tau\N$,
we define~$u(x,t)$ to be the probability density
of the particle to be at the point~$x$ at the time~$t$,
and we write that the probability of being somewhere,
say at~$x$,
at the subsequent time step is equal to
the superposition of the probabilities of being at another 
point of the lattice,
say~$x+hj$,
at the previous time step times the probability of going
from~$x+hj$ to~$x$, namely, letting~$k:=\frac{x}h\in\Z^n$,
\begin{eqnarray*}
u(x,t+\tau)=\sum_{j\in\Z^n\setminus\{0\}}
u(x+hj,t)\,{\mathcal{P}}(k,k+j)
=\sum_{j\in\Z^n\setminus\{0\}}
u(x+hj,t)\,{\mathcal{P}}(j,0).
\end{eqnarray*}
As a result, in view of~\eqref{SOM1},
\begin{equation}\label{CA-0}
\begin{split}&
u(x,t+\tau)-u(x,t)\\=\;&\sum_{j\in\Z^n\setminus\{0\}}
\Big(u(x+hj,t)-u(x,t)\Big)\,{\mathcal{P}}(j,0)\\=\;&
\sum_{j\in\Z^n\setminus\{0\}}
\Big(u(x+hj,t)-u(x,t)\Big)\,\left(
\frac{p}{c\,|j|^{n+2s}}+\frac{(1-p)\,{\mathcal{U}}(j)}{2n}\right)\\
=\;&\frac{p}c\,
\sum_{j\in\Z^n\setminus\{0\}}
\frac{u(x+hj,t)-u(x,t)}{|j|^{n+2s}}+\frac{1-p}{2n}\,
\sum_{j\in\lambda{\mathcal{B}}}
\big(u(x+hj,t)-u(x,t)\big)\\=\;&\frac{p}{2c}\,
\sum_{j\in\Z^n\setminus\{0\}}
\frac{u(x+hj,t)+u(x-hj,t)-2u(x,t)}{|j|^{n+2s}}\\&\qquad+\frac{1-p}{4n}\,
\sum_{j\in\lambda{\mathcal{B}}}
\big(u(x+hj,t)+u(x-hj,t)-2u(x,t)\big).
\end{split}\end{equation}
Now we consider two specific situations, namely the one in
which
\begin{equation}\label{CA-01}
\tau:=h^{2s},\qquad\lambda:=h^{s-1}\in\N
\end{equation}
and~$p$ is independent on the time step, and the one in which
\begin{equation}\label{CA-02}
\tau=h^{2},\qquad p:=\alpha h^{2-2s},
\qquad{\mbox{and}}\qquad\lambda:=1,
\end{equation}
for a given~$\alpha>0$, independent on the time step.

We observe that the case in~\eqref{CA-01} corresponds to
having the closest neighborhood walk scaled by a suitably large factor
(for small~$h$), while the case in~\eqref{CA-02} corresponds to
having the usual notion of closest neighborhood random walk,
with the probability~$1-p$ that the particle follows it being large (for small~$h$).

In case~\eqref{CA-01}, 
we consider~$N\in\N$ and define~$h:=N^{\frac1{s-1}}$.
In this way, taking~$N\nearrow+\infty$, one has that~$h\searrow0$,
and thus
we deduce from~\eqref{CA-0}
that
\begin{equation}\label{CA-3}
\begin{split}\frac{
u(x,t+\tau)-u(x,t)}\tau\,&=\,\frac{ph^n}{2c}\,
\sum_{j\in\Z^n\setminus\{0\}}
\frac{u(x+hj,t)+u(x-hj,t)-2u(x,t)}{|hj|^{n+2s}}\\&\qquad+\frac{1-p}{4n}\,
\sum_{j\in h^{s-1}{\mathcal{B}}}
\frac{u(x+hj,t)+u(x-hj,t)-2u(x,t)}{h^{2s}}
.
\end{split}\end{equation}
With a formal Taylor expansion, we observe that
$$ u(x+hj,t)+u(x-hj,t)-2u(x,t)=
h^2 D^2_x u(x,t) j\cdot j+O(h^3),$$
therefore the latter sum in~\eqref{CA-3} can be written as
\begin{eqnarray*}&&
\sum_{j\in h^{s-1}{\mathcal{B}}}
h^{2(1-s)} D^2_x u(x,t) j\cdot j+O(h^{3-2s})=
\sum_{j\in N{\mathcal{B}}}
D^2_x u(x,t) \frac{j}{N}\cdot \frac{j}{N}+O\left(\frac1{N^{\frac{3-2s}{1-s}}}\right)\\&&\qquad=
\sum_{i\in {\mathcal{B}}}
D^2_x u(x,t) i\cdot i+o(1)=2\Delta u(x,t)+o(1),
\end{eqnarray*}
as $N\nearrow+\infty$ (i.e., as~$h\searrow0$).

Hence, recognizing a Riemann sum in the first term of the right hand side
of~\eqref{CA-3}, taking the limit as~$h\searrow0$ (that is~$\tau\searrow0$),
we formally conclude that
\begin{eqnarray*}&&
\partial_tu(x,t)=\,\frac{p}{2c}\,\int_{\R^n}
\frac{u(x+y,t)+u(x-y,t)-2u(x,t)}{|y|^{n+2s}}\,dy
+\frac{1-p}{2n}\,\Delta u(x,t)
,\end{eqnarray*}
which is precisely the heat equation associated to the operator in~\eqref{problogis}
(up to defining correctly the structural constants).

A similar argument can be carried out in case~\eqref{CA-02}.
Indeed, in this situation one deduces from~\eqref{CA-0}
that
\begin{eqnarray*}\frac{
u(x,t+\tau)-u(x,t)}\tau&=&\frac{\alpha h^n}{2c }\,
\sum_{j\in\Z^n\setminus\{0\}}
\frac{u(x+hj,t)+u(x-hj,t)-2u(x,t)}{|hj|^{n+2s}}\\&&\qquad+\frac{1-\alpha h^{2-2s}}{4n}\,
\sum_{j\in{\mathcal{B}}}
\frac{u(x+hj,t)+u(x-hj,t)-2u(x,t)}{h^2}.
\end{eqnarray*}
Hence, since
\begin{eqnarray*} \sum_{j\in{\mathcal{B}}}
\frac{u(x+hj,t)+u(x-hj,t)-2u(x,t)}{h^2}&=&
\sum_{j\in{\mathcal{B}}}
D^2_x u(x+hj,t)j\cdot j+O(h)\\&=&
2\Delta u(x,t)+o(1)\end{eqnarray*}
as~$h\searrow0$, we conclude that in this case
$$ \partial_tu(x,t)=
\frac{\alpha }{2c}\,\int_{\R^n}
\frac{u(x+y,t)+u(x-y,t)-2u(x,t)}{|y|^{n+2s}}\,dy+\frac{1}{2n}\,\Delta u(x,t),$$
which, once again, constitutes the
parabolic equation associated to the operator in~\eqref{problogis}
(up to renaming the structural constants).

\end{document}